\documentclass[10pt]{article}
\usepackage[Bjornstrup]{fncychap}
\usepackage[T1]{fontenc}
\usepackage[latin1]{inputenc}
\usepackage[main=english,french]{babel}
\usepackage{mathtools,amssymb,amsthm,upgreek}
\usepackage{graphicx}
\usepackage{geometry}
\usepackage{enumitem}
\usepackage{mathrsfs}
\usepackage{amsmath}
\usepackage{url}
\usepackage{wrapfig}
\usepackage{fancybox}
\usepackage{fullpage}
\usepackage{hyperref}
\usepackage{multicol} 
\usepackage{array}
\usepackage{xargs}
\usepackage[capitalise]{cleveref}

\usepackage[prependcaption]{todonotes}
\newcommandx{\fab}[2][1=]{\todo[inline, author={Fabien}, linecolor=blue,backgroundcolor=blue!25,bordercolor=blue,#1]{#2}}
\newcommandx{\Max}[2][1=]{\todo[inline, author={Maxime}, linecolor=green,backgroundcolor=green!25,bordercolor=green,#1]{#2}}

\newtheorem{theoreme}{Theorem}[section]

\newtheorem{prop}{Proposition}[section]
\newtheorem{lem}{Lemma}[section]
\newtheorem{cor}{Corollary}[section]

\theoremstyle{remark}
\newtheorem{rem}{\bf Remark}[section]

\renewcommand{\l}{\left}
\renewcommand{\r}{\right}
\newcommand{\un}{\underline}
\newcommand{\1}{c_{1}} 
\newcommand{\2}{c_2} 
\newcommand{\3}{c_3} 
\newcommand{\4}{c_4} 

\newcommand{\X}{X}
\newcommand{\Xge}{\bar{X}}
\newcommand{\B}{B}
\newcommand{\pas}{\gamma}

\newcommand{\laminf}{\alpuc}
\newcommand{\E}{\mathbb{E}}
\newcommand{\PE}{\mathbb{P}}
\newcommand{\ind}{r}
\newcommand{\lev}{R}
\newcommand{\ER}{\mathbb{R}}
\newcommand{\ES}{\mathbb{E}}
\newcommand{\tcr}{\textcolor{red}}

\newcommand{\HUN}{\mathbf{(H_{1})}}
\newcommand{\HDEUX}{\mathbf{(H_{2})}}
\newcommand{\bea}{\mathfrak{a}}
\newcommand{\HTROIS}{\mathbf{(H_{3})}}
\newcommand{\HQU}{\mathbf{(H_4)}}

\newcommand{\Cs}{\mathbf{(C_{\alpha})}}
\newcommand{\cfrak}{\mathfrak{c}}
\newcommand{\blip}{\bun}

\newcommand{\alpuc}{\alpha}

\newcommand{\bun}{L}
\newcommand{\cuniv}{\mathfrak{c}_{\mathfrak{u}}}
\newcommand{\tauun}{\tau_1}
\newcommand{\taudeux}{\tau_2}
\newcommand{\cbea}{\mathfrak{c}_{\bea}}
\newcommand{\dbea}{\mathfrak{d}_{\bea}}
\newcommand{\Tfrak}{\mathfrak{t}}

\author{Maxime Eg\'ea~\thanks{Universit\'e d'Angers, CNRS, LAREMA, SFR Mathstic, F-49000 Angers, France. E-mail: \texttt{maxime.egea@univ-angers.fr}}  $\;$and  Fabien Panloup~\thanks{Universit\'e d'Angers, CNRS, LAREMA, SFR Mathstic, F-49000 Angers, France. E-mail: \texttt{fabien.panloup@univ-angers.fr}}
}

\begin{document}

\title{Multilevel-Langevin pathwise average for Gibbs approximation}

\maketitle
\begin{abstract}
We propose and study a new multilevel method for the numerical approximation of a Gibbs distribution $\pi$ on $\mathbb{R}^d$, based on  (overdamped) Langevin diffusions.  This method inspired by   \cite{mainPPlangevin} and \cite{giles_szpruch_invariant} relies on a multilevel occupation measure, $i.e.$ on an appropriate combination  of $\lev$ occupation measures of (constant-step) Euler schemes with respective steps $\pas_\ind=\pas_0 2^{-\ind}$, $\ind=0,\ldots,\lev$. We first state a  quantitative result under general assumptions which guarantees an \textit{$\varepsilon$-approximation} (in a $L^2$-sense) with a cost  of the order $\varepsilon^{-2}$ or $\varepsilon^{-2}|\log \varepsilon|^3$ under less contractive assumptions. 

We then apply it to overdamped Langevin diffusions with strongly convex potential $U:\ER^d\rightarrow\ER$ and obtain an  \textit{$\varepsilon$-complexity} of the order ${\cal O}(d\varepsilon^{-2}\log^3(d\varepsilon^{-2}))$ or ${\cal O}(d\varepsilon^{-2})$ under additional assumptions on $U$.  More precisely, up to universal constants, an appropriate choice of the parameters  leads to a cost controlled by ${(\bar{\lambda}_U\vee 1)^2}{\underline{\lambda}_U^{-3}} d\varepsilon^{-2}$ (where $\bar{\lambda}_U$ and $\underline{\lambda}_U$ respectively denote the supremum and the infimum of the largest and lowest eigenvalue of $D^2U$).\smallskip

{We finally complete these theoretical results with some numerical illustrations including comparisons to other algorithms in Bayesian learning and opening to non strongly convex setting.}
\end{abstract}
{{\footnotesize \textit{Mathematics Subject Classification:} Primary 65C05-37M25 Secondary 65C40-93E35.}

}

\section{Introduction}

Let  $ \l(B_t \r)_{t \ge 0} $ denote a $d$-dimensional standard Brownian motion. Let $(X_t)_{t\ge0}$ denote the solution of  the  stochastic differential equation (\textbf{SDE}) 
\begin{equation}\label{eq:SDE}
\mathrm{d} X_t = b(X_t) \mathrm{d}t + \sigma(X_t) \mathrm{d}B_t ,
\end{equation}
{where} $b$ : $\mathbb{R}^d \rightarrow \mathbb{R}^d$ and $\sigma:\ER^d\rightarrow \mathbb{M}_{d,d}$ (space of $d$-squared matrices) are Lipschitz continuous function.  Under these assumptions, strong existence and uniqueness classically hold and $\l( X_t \r)_{t \ge 0 }$ is a Markov process whose {semi-group} will be denoted by $\l( P_t \r)_{t \ge 0 }$. {Throughout this} paper, we assume that $(X_t)_{t\ge0}$ has a unique invariant distribution {denoted by $\pi$}. Such {a} property arises in particular under Lyapunov assumptions and non-degeneracy of the diffusion coefficient $\sigma$ (for background, see $e.g.$ \cite{meyn_tweedie,pages_survey}). \smallskip

For such a diffusion process, we denote by $\l( \bar{X}_{n\pas}^{\pas,x_0} \r)_{n \in \mathbb{N}} $ the related Euler (or Euler-Maruyama) scheme with constant step $\pas$ and starting point $x_0$: for $\pas>0$ and $x_0\in\ER^d$, the discretization scheme $\l( \bar{X}_{n\pas}^{\pas,x_0} \r)_{n \in \mathbb{N}} $ is  recursively defined by $\bar{X}_0^{\pas,x_0}  = x_0$ and
\begin{equation}
{\forall \; n\ge0,\quad}\bar{X}_{(n+1)\pas}^{\pas,x_0} = \bar{X}_{n\pas}^{\pas,x_0} + \pas b(\bar{X}_{n\pas}^{\pas,x_0}) + \sigma(\bar{X}_{n\pas}^{\pas,x_0} )\l( B_{(n+1)\pas}-B_{n\pas} \r).
\end{equation}
We also {introduce} one of its continuous-time extensions, sometimes {called} \textit{genuine} continuous-time Euler scheme given by: for all $n\in\mathbb{N}$ and for all $t \in [ n\pas, (n+1)\pas )$,

\begin{equation*}
\bar{X}_t^{\pas,x_0}: = \bar{X}_{n\pas}^{\pas,x_0} + (t-n\pas)b \l( \bar{X}_{n\pas}^{\pas,x_0} \r) + \sigma(\bar{X}_{n\pas}) \l( B_t - B_{(n+1)\pas} \r). 
\end{equation*}
{This continuous-time extension is sometimes called \textit{pseudo-diffusion} since it satisfies}
\begin{equation}\label{pagesbargamma}
 \bar{X}_t^{\pas,x_0} = x_0 + \int_0^t b(\bar{X}_{\underline{s}_\pas}^{\pas,x_0}) \mathrm{d}s+  \int_0^t \sigma(\bar{X}_{\underline{s}_\pas}^{\pas,x_0}) \mathrm{d}B_s,
\end{equation}
for all $t\ge0$, where for $\eta>0${,}\begin{equation}
\underline{t}_\eta{:= \mathrm{max} \{k \ge 0,  k\eta \le t \}}{.}
\end{equation}
If no confusion arises, we will sometimes write $\underline{t}$ instead of $\underline{t}_\pas$, and $\bar{X}_t$ or $\bar{X}_t^{\pas}$  instead of $\bar{X}_t^{\pas,x_0}$, in order to alleviate the notations.\smallskip


\noindent  Now, let us come back {to} the literature on numerical approximation of invariant distributions of diffusion processes and on multilevel methods.\smallskip

\noindent \textbf{Ergodic approximation and Gibbs approximation}. There exists a huge literature on the numerical approximation of the invariant distribution $\pi$ based on such discretization schemes. For a general diffusion process, \cite{talay} studies the convergence  of an algorithm based on the occupation measure of the Euler scheme almost surely defined (with continuous-time notations) by:  
$$\nu_T^\pas:=\frac{1}{T}\int_0^T \delta_{\bar{X}_{\un{s}_\pas}^\pas} ds, \quad T>0.$$
In order to manage the long-time and discretization errors in the same time, \cite{LP1,LP2} develop the same type of algorithms for Euler schemes with decreasing step sequence (in the same spirit, see \cite{lemaire2,mattingly_stuart,PP1,PP3,panloup1} for refinements or extensions to more general models). In the previous references, it is worth noting that the objective is to approximate the generally unknown physical equilibrium of a given stochastic dynamical system. The aim is thus different from the MCMC algorithms which aim at sampling a given explicit probability $\pi$ ({in} the most efficient way). \smallskip

\noindent {Nevertheless,  the above methods can certainly be used in view of MCMC-type objectives when they are applied to diffusions with an explicit invariant distribution. This is the case when one considers the (overdamped) Langevin diffusion 
\begin{equation}\label{langevin_sde}
dX_t=-\sigma^2\nabla U(X_t)dt+\sqrt{2}\sigma dB_t,
\end{equation}
where $\sigma$ is a positive number and $U:\ER^d\rightarrow\ER$ is a coercive  function (such that $e^{-{U}}$ is integrable on $\ER^d$). It is well-known that the unique invariant distribution of \eqref{langevin_sde} is the \textit{Gibbs  distribution} $\pi$ defined by, 
$$\pi(dx):=\frac{1}{Z_{U}} e^{-{U(x)}}\lambda_d(dx), \quad Z_{U}=\int e^{-{U(x)}}\lambda_d(dx).$$}
The study of the long-time behavior of Euler-Maruyama schemes of \eqref{langevin_sde} has been the topic of numerous papers in the last years. Among others, we can refer to \cite{durmus_moulines,durmus_moulines2,Durmus_Majewski, dalalyan,dalalyan_karagulyan,Flammarion_Wainwright} where the authors generally focus on the (Wasserstein, Total Variation,\ldots) distance $d$ between the distribution of the Euler scheme and $\pi$ and optimize the step and the time in order to minimize the number of iterations of the Euler scheme which is necessary to obtain $d({\cal L}(\bar{X}_{n_\varepsilon\pas_\varepsilon}^{\pas_\varepsilon}),\pi)\le \varepsilon$ (for a given $\varepsilon$).
In particular, these papers focus more on the bias than on the variance. In view of applications in machine learning, the authors generally emphasize the dependence {on} the dimension $d$ of the cost of computation. We will come back later {to} this point and {to} the existing results compared with ours (see Remark \ref{rem:comparison}).\\

\noindent \textbf{Multilevel Langevin and $\varepsilon$-approximations.} Multilevel methods, pionnered {by \cite{heinrich_multilevel}} and \cite{giles} (see also \cite{kebaier_ahmed}), and based on appropriate combinations of rough and refined approximations of the target, belong to the family of strategies for speeding up Monte-Carlo methods by bias reduction. The main idea of multilevel methods is to {(try to)} bring correcting layers {with low variance} to a rough approximation of a target. Multilevel methods received a lot of success in numerical applications, especially in discretization methods for diffusions (but also in other problems such as the approximation of nested expectations). For instance, {in} the classical problem of computing $\ES[f(X_T)]$ ({with $T>0$}), such methods are known to produce a complexity which is (almost) proportional to unbiased\footnote{When a random variable $Y$ can be simulated exactly ($i.e.$ without bias), getting an $\varepsilon$-approximation  of $\ES[Y]$ with the standard Monte-Carlo approximation $N^{-1}\sum_{k=1}^N Y_k$ (where $(Y_k)_k$ is an $i.i.d.$ sequence such that $Y_1\sim Y$) requires $N_{\varepsilon}={\rm Var}(Y) \varepsilon^{-2}$ simulations of $Y$.} methods. More precisely, for a given $\varepsilon>0$, the parameters of the multilevel procedure can be {calibrated} in such a way that the {required} number of iterations of the Euler scheme {for an \textit{$\varepsilon$-approximation} (see \cref{sec:notations} for a definition)  is} proportional to $\varepsilon^{-2}\log^2\varepsilon$ in general or to $\varepsilon^{-2}$ under additional assumptions (which are true for additive diffusions).\\

\noindent {Concerning} the computation of the invariant distribution of a diffusion, multilevel methods have already been studied in {\cite{fang-giles}}, \cite{giles_szpruch_invariant} and \cite{mainPPlangevin}. In \cite{fang-giles} and \cite{giles_szpruch_invariant}, the procedure is based on a standard multilevel Monte-Carlo approach with discretization schemes of a Langevin diffusion and adapted time horizons and produces a complexity proportional to $\varepsilon^{-2}|\log \varepsilon|^3$ {or to $\varepsilon^{-2}$ under additional assumptions}. In \cite{mainPPlangevin}, written in  a {multiplicative} setting and based on a so-called Multilevel-Romberg weighted combination of occupation measures of discretization schemes with decreasing step, the algorithm has a complexity proportional to $\varepsilon^{-2}|\log \varepsilon|$ (see Section \ref{sec:notations} below for our definition of complexity). In terms of $\varepsilon$, these approaches generate a real gain compared with the non multilevel ones (mentioned above) which generally  produce a complexity proportional to $\varepsilon^{-3}$. However, {the above references do not calibrate the dependence of the procedure with respect to the other parameters and especially with respect to the dimension, which may be of first importance in applications. In this paper, our objective is thus to provide a procedure and some related results which exhibit an $O(\varepsilon^{-2})$-complexity combined with some sharp bounds on these parameters.} \\

\noindent \textbf{Idea of the algorithm.} \noindent Before detailing our contributions, let us briefly describe the construction of the algorithm (the precise procedure will be detailed in Section \ref{sec:setmain}) and give some comments. Our procedure is based on occupation measures  as in  \cite{mainPPlangevin}. {With such an approach, we thus aim to take advantage of the (pathwise) convergence of the occupation measure of a Markov process towards its invariant distribution.  Compared with \cite{mainPPlangevin}, we use a  simpler multilevel approach since we will use Euler schemes with constant steps and do not introduce weights in the average (in particular, the weighted approach used in \cite{mainPPlangevin} introduces many technical difficulties which seem to be hard to overcome in view of quantitative bounds)}. 

More precisely, our strategy is based on an almost telescopic sum of differences of occupations measures of Euler schemes with step $\pas_r=\pas_0 2^{-r}$, $r=0,\ldots,R$ with  $R\in\mathbb{N}^*$ (this means that there are $R+1$ levels). Recall that $\pi^\pas$ denotes the invariant distribution of the Euler scheme with step $\pas$. Let $f:\ER^d\rightarrow\ER$. At the starting point, 
we try to mimick the telescopic sum  
$$\pi^{\pas_R}(f)= \pi^{\pas_0}(f)+\sum_{r=1}^R \pi^{\pas_r}(f)-\pi^{\pas_{r-1}}(f)$$
in order to generate a procedure with a bias close to $\pi(f)-\pi^{\pas_R}(f)$ but with a probability
$\pi^{\pas_R}$ viewed as a correction of $\pi^{\pas_0}$ by a sequence of (correcting) levels. With an ``occupation measure point of view'', we mimick the above decomposition  by considering the  procedure
\begin{equation}\label{eq:heuristic}
\nu_{\tau,T_0}^{\pas_0}(f)+\sum_{r=1}^R \nu_{\tau,T_r}^{\pas_r}(f)-\nu_{\tau,T_r}^{\pas_{r-1}}(f),
\end{equation}
where $\nu_{\tau,T}^\pas(f)=\frac{1}{T-\tau}\int_\tau^T f(\bar{X}_{\un{s}_\pas}^\pas) ds$ and $\tau,$ $T_0$, \ldots, $T_R$ are  positive numbers.  {The terms from $r=1$ to $R$ play the role of the correcting layers} and are based on couplings of Euler schemes with steps $\gamma_{r-1}$ and $\gamma_r=\gamma_{r-1}/2$. Without going more into the details of the construction, let us give several hints.
\begin{itemize}
\item{} The parameter $\tau$ must be {viewed} as a \textit{warm-start}: we choose to average the path after a time where the starting point has been slightly forgotten in order to reduce the bias induced by the long-time error.  In fact, this slight modification of the average is necessary to capture the complexity in ${\cal O}(\varepsilon^{-2})$ ($i.e.$ to cancel the logarithmic terms of \cite{mainPPlangevin}).
\item{} In a classical Monte-Carlo Multilevel setting, $i.e.$ based on spatial average and not on time average ($i.e.$ on occupation measures), the idea is to simulate a large number of Euler schemes with {large} step and less and less paths of Euler schemes with {thin} step (since the cost of simulation increases with the refinement of the step). With an occupation measure point of view, this heuristic is replaced by the following assumption:
$$T_0>T_1>\ldots >T_R,$$
which means that the length of the path of the Euler scheme decreases with the step size {so that the number of copies of the Euler schemes is here encoded by the length of the path. In particular, the largest horizon corresponds to the Euler scheme with largest step which has the role of controlling the variance induced by the empirical mean.} We will see that $(T_r)_r$ is geometrically decreasing; this is {consistent} with the geometric decrease of the step size. 
\item{} In \eqref{eq:heuristic}, we did not {specify} the choice of the Brownian motions inside each level. Here, we will adopt the classical strategy: {we consider a} sequel of independent Brownian Motions $B^{(r)}$, $r=0,\ldots,R$ and {at each level}, we build  $\nu_{\tau,T_r}^{\pas_r}$ and $\nu_{\tau,T_r}^{\pas_{r-1}}$ with this Brownian Motion $B^{(r)}$.  {This means that on the one hand, the levels are independent and on the other hand, the Euler schemes involved in a level $r\ge1$ are built with a \textit{synchronous coupling} (since they are driven by the same Brownian motion).}
\end{itemize}

\noindent \textbf{Contributions and plan of the paper.} The first objective of this work is {to provide a general diffusion setting in which the invariant distribution can be approximated by this combination of occupation measures with a complexity proportional to $\varepsilon^{-2}$}. In the same time, we also {want} to give quantitative bounds and to answer to the following question: Is a multilevel method able to reduce the cost in $\varepsilon$ without worsening the dependence {on} the other parameters (and especially {on} the dimension) ? The answers to these questions, given in Sections \ref{sec:setmain2} and \ref{sec:langevinsc}, {are summarized} below.
\begin{itemize}
\item{} {We answer to the first question in} Theorem \ref{maintheo}, {which is stated under general} assumptions on the behavior of the Euler scheme (ergodicity, long-time \textit{confluence} of the paths, bounds on the distance between $\pi^\pas$ and $\pi$ and on the moments). This result shows that for a given positive $\varepsilon$, we can tune the parameters of the multilevel procedure in such a way that for any $1$-Lipschitz function $f:\ER^d\rightarrow\ER$, an $\varepsilon$-approximation of $\pi(f)$ can be obtained with $\mathfrak{C}_2\varepsilon^{-2}$  or $\mathfrak{C}_2\varepsilon^{-2}|\log \varepsilon|^3$ iterations of the Euler scheme. {Furthermore,} $\mathfrak{C}_2$ is an (almost)\footnote{By ``almost'', we mean ``up to universal constants'' (see Section \ref{sec:notations} for details).}  explicit function of the parameters involved in the assumptions. {Based on some classical estimates of the occupation measure of a Markov process, the main novelty here  is to exhibit some precise conditions on the algorithm and some quantitative bounds on the complexity which guarantee  an $\varepsilon^{-2}$-complexity.}
\item{}  Even though Theorem \ref{maintheo} potentially  applies to general diffusions (see Remark \ref{rem:theo21}), we choose to focus on additive diffusions with strongly contractive drift, and especially on over-damped Langevin diffusions (with strongly uniformly convex potential) in view of applications to Gibbs sampling.  {This is the purpose of} Section \ref{sec:langevinsc} where a series of results {provide concrete}  multilevel procedures for $\varepsilon$-approximations in $L^2$ with (almost) explicit bounds on the complexity. The results are divided into two parts. {In the first one}, $b$ is a (contractive) Lispchitz ${\cal C}^1$-vector field, or equivalently, $U$ is ${\cal C}^2$ with Lipschitz gradient in the case of Langevin diffusions. {In the second part},  $b$ is a ${\cal C}^2$-vector field with bounded derivatives up to order $2$. In this case, refined expansions lead to $L^2$-bounds \textit{of order 1} for the Euler scheme and  allow to apply Theorem \ref{maintheo} with  friendlier parameters.
The main results of this section are Theorems \ref{maincoro1} and \ref{maincoro2}  in the case of a general vector field $b$, and Corollaries \ref{cor:gibbsI} and  \ref{cor:gibbs2} in the case of Langevin diffusions. In the first part (when $b$ is only ${\cal C}^1$), we obtain some bounds on the complexity in ${\cal O}(d\varepsilon^{-2}|\log (d\varepsilon)|^3)$ whereas in the (refined) second part, the complexity is bounded by ${\cal O}(d\varepsilon^{-2})$. \\

\noindent However, the complexity also depends on the (intrinsic) parameters of the model: the Lipschitz constant $L$ of $b$ and the contraction parameter $\alpha$ (corresponding respectively to the largest and lowest eigenvalues $\bar{\lambda}_U$ and $\underline{\lambda}_U$ of the Hessian of $U$ when $b=-\nabla U$). We thus also {detail} the dependence {on} these parameters, which up to a logarithmic term, is proportional to $L^2/\alpha^3$. \\

\noindent Finally, in Corollaries \ref{cor:gibbsI} and \ref{cor:gibbs2}, we apply these results to the particular case $b=-\nabla U$ and  optimize the choice of the diffusion coefficient $\sigma$ in order to kill the logarithmic term and to obtain a normalized procedure where the parameters have a nice and simple form (for instance, $\pas_0=1/2$). {These quantitative bounds with respect to $d$, $\varepsilon$, $L$ and $\alpha$ are the main contributions of this paper. They rely on a careful study of the multilevel strategy given in Theorem \ref{maintheo} combined with sharp bounds on the long-time behavior of the Euler scheme in the strongly convex setting (see in particular \cref{lem:boundEuler} and \cref{prop:L2error}).}
\end{itemize}
{In Section \ref{sec:numillu}, we propose several numerical illustrations with {different}  models which allow to test the efficiency of our methods 
with respect to the parameters {and to compare with other classical methods (in a Bayesian example, see \cref{subsec:MCMCbayesian})}. We also open to some  perspectives for reducing the influence of $\alpha$ and $L$ on the complexity of the method {and finally test our algorithm in a simple non convex setting in order to show that theoretical extensions may be tackled in such a setting (in a future work).}\\

\noindent Sections \ref{sec:quantitativecontrol}, \ref{sec:maintheo} and \ref{sec:maintheo23} are devoted to the proofs of the main theorems. In Sections \ref{sec:quantitativecontrol} and \ref{sec:maintheo}, we prove Theorem \ref{maintheo} {with the help of an accurate} study of the bias/variance errors. The proofs related to additive diffusions with strongly contractive drift (including Langevin diffusions) are written in Section \ref{sec:maintheo23}. 

\section{Setting and main results}\label{sec:setmain}
\subsection{Notations/Definitions}\label{sec:notations}
{We list below the main notations. A list of all the specific symbols is also given in \cref{sec:listsymbols}.}
\begin{itemize}
\item{} The usual scalar product on $\ER^d$ is denoted by $\langle\,,\,\rangle$ and the induced Euclidean norm by $|\,.\,|$. The set $\mathbb{M}_{d,d}$ refers to the set of real $d \times d$. The Frobenius norm on $\mathbb{M}_{d,d}$ is denoted by $\|\,.\,\|_F$: for any $A\in\mathbb{M}_{d,d}$, $\|A\|_F^2=\sum_{1\le i,j\le d} A_{i,j}^2.$
\item{} The Lipschitz constant of a given (Lipschitz) function $f:\ER^d\rightarrow\ER$ is denoted by $[f]_1$: $[f]_1=\sup_{x,y\in\ER^d} |f(x)-f(y)|.|x-y|^{-1}$.
A function $f:\ER^d\rightarrow\ER$ is ${\cal C}^k$, $k\in\mathbb{N}$, if all its partial derivatives are well-defined and continuous up to order $k$. The gradient and Hessian matrix of $f$ are respectively denoted by $\nabla f$ and $D^2f$. 
\item{} The probability space is denoted by $(\Omega,{\cal F},\PE)$. The $L^p$-norm on $(\Omega,{\cal F},\PE)$ is denoted by $\|\,.\,\|_p$.
\item{} \textbf{$\varepsilon$-approximation}: We say that ${\cal Y}$ is an  $\varepsilon$-approximation of a real number $a$ (for the $L^2$-norm), if $\|{\cal Y}-a\|_2=\ES[|{\cal Y}-a|^2]^{\frac{1}{2}}\le \varepsilon$. Equivalently, ${\cal Y}$ is said to be an  $\varepsilon$-approximation of $a$ if the related Mean-Squared Error (MSE) is lower than $\varepsilon^2$.

\item{} \textbf{Complexity/$\varepsilon$-complexity}: For a random variable ${\cal Y}$ built with some iterations of a standard Euler scheme, we {call complexity and } denote by ${\cal C}(\cal Y)$, the number of iterations of the Euler scheme which is needed to compute ${\cal Y}$. For instance, ${\cal C}(\bar{X}_{n\pas}^{\pas})=n$. The $\varepsilon$-complexity is the complexity of an algorithm which produces an $\varepsilon$-approximation. 

\item{} \textbf{Universal constant}: A positive number which does not depend on any parameter of the problem is called a universal constant and is denoted by $\cuniv$. We will write $a\lesssim_{uc} b$ if $a\le \cuniv b$.\\
\end{itemize}
\subsection{Design of the algorithm}
We now {detail the} construction of the multilevel procedure. Let $x_0\in\ER^d$ (starting point of the Euler scheme),  $\lev \in \mathbb{N}^*$ (number of correcting levels), $ \l(T_\ind\r)_{0 \le \ind \le \lev}\label{pagetr}$ be {a decreasing} sequence of positive times and   $\l(\pas_\ind\r)_{0 \le \ind \le \lev}\label{pagepasr}$ be the decreasing sequence of step sizes defined by $\pas_r=\pas_0 2^{-r}$ and $\tau$ be a positive number.\\

\noindent For these parameters, we denote by $\mathcal{Y}(\lev,\l(\pas_\ind\r)_\ind,\tau,\l(T_\ind\r)_\ind,x,.)$, the empirical probability measure defined by: for any Borel measurable function $f:\ER^d\rightarrow\ER$,
 \begin{equation}\label{eq:estimat}
 \begin{split}
\mathcal{Y}(\lev,\l(\pas_\ind\r)_\ind,\tau,\l(T_\ind\r)_\ind,x,f) &:= \frac{1}{T_0-\tau}\int_{\tau}^{T_0} f(\bar{X}_{\un{s}_{\pas_0}}^{\pas_0,x_0, B^{(0)}}) ds\\
&+ \sum_{\ind=1}^\lev \frac{1}{T_r-\tau}\int_\tau^{T_r} f(\bar{X}_{\un{s}_{\pas_{r-1}}}^{\pas_{r},x_0, B^{(r)}}) -f(\bar{X}_{\un{s}_{\pas_{r-1}}}^{\pas_{r-1},x_0, B^{(r)}}) ds,
\end{split}
\end{equation}
{where $\{B^{(\ind)}, \;\ind=0,\ldots,\lev\}$, denotes a sequence of $\lev+1$ independent Brownian motions. In particular, if $\tau$ and $T_0$ (resp. $T_r$, $r=1,\ldots,R$) are  multiples of $\pas_0$ (resp.  of $\pas_{r-1}$), the above definition takes the form
\begin{equation*}\label{eq:estimat2}
 \begin{split}
\mathcal{Y}(\lev,\l(\pas_\ind\r)_\ind,\tau,\l(T_\ind\r)_\ind,x,f) &:= \frac{1}{n_{\pas_0}(T_0)-n_{\pas_0}(\tau)}\sum_{k=n_{\pas_0}(\tau)}^{n_{\pas_0}(T_0)-1} f(\bar{X}_{k{\pas_0}}^{\pas_0,x_0, B^{(0)}}) \\
&+ \sum_{\ind=1}^\lev \frac{1}{n_{\pas_{r-1}}(T_r)-n_{\pas_{r-1}}(\tau)}\sum_{k=n_{\pas_{r-1}}(\tau)}^{n_{\pas_{r-1}}(T_r)-1} \l(f(\bar{X}_{k\pas_{r-1}}^{\pas_r,x_0, B^{(r)}}) -f(\bar{X}_{k{\pas_{r-1}}}^{\pas_{r-1},x_0, B^{(r)}})\r).
\end{split}
\end{equation*}
where for some given positive $T$ and $\pas$, $n_\pas(T)=\max\{k,k\pas\le T\}\label{pagengamma}$, $i.e.$ $n_\pas(T)$ is the discretization index related to $T$ when the step is equal to $\pas$ (in the general case, the border terms of the above expression must be modified).
\begin{rem} 
It is worth noting that in the correcting levels ($r=1,\ldots,R$), the Euler schemes of steps $\pas_r$ and $\pas_{r-1}$ are averaged at times $k\pas_{r-1}$ only, $i.e.$ at discretization times of the Euler scheme with thickest step. In particular, one could wonder why one does not average {over all} discretization times for the Euler scheme with finest step $\pas_r$. In fact, such an average would generate an additional error (with a size proportional to $\sqrt{\pas_r}$) which is not negligible in the case $\bea>1$ of Theorem \ref{maintheo} below (case which leads to a complexity proportional to $\varepsilon^{-2}$).
\end{rem}

\noindent \textbf{Complexity of the algorithm.} With respect to the definition given in Section \ref{sec:notations}, we remark  that when $T_r$, $r=0,\ldots,R$ are multiples of the given step sequence, the complexity of ${\cal Y}$ satisfies:
$${\cal C}({\cal Y})=\frac{T_0}{\pas_0}+\sum_{r=1}^R \l(\frac{{T_r}}{\pas_r}+\frac{{T_r}}{\pas_{r-1}}\r).$$
Note that for the sake of simplicity, we do not recall  the parameters of ${\cal Y}(\lev,\l(\pas_\ind\r)_\ind,\tau,\l(T_\ind\r)_\ind,x,.)$ in the notation ${\cal C}({\cal Y})$.

\subsection{A general result}\label{sec:setmain2}
The aim of this section is to state a  result about the $\varepsilon$-complexity of our multilevel ergodic strategy under appropriate general assumptions.\\

\noindent Let $f:\ER^d\rightarrow\ER$ be a Lipschitz continuous function with Lipschitz constant denoted by $[f]_1$. We introduce the following series of assumptions depending on $f$, on a positive $\eta_0\label{pageetazero}$ (which in the sequel can be taken as the largest step size used in the multilevel procedure), and on $x_0\label{pagexzero}$ which is the starting point of each Euler scheme involved in the multilevel procedure. In the following assumptions, we recall that the invariant distribution $\pi^\pas\label{pagepipas}$ of $(\bar{X}_{n\pas})_{n\ge0}$ is assumed to exist and to be unique.\\

\noindent {The following assumption is assumed to hold for a positive $\alpha$.}\\

\noindent $\HUN\label{HUN}$ (Convergence to equilibrium): For all $x\in\ER^d$, there exists a finite constant $\1(x)$ such that,  for every $\pas \in(0,\eta_0]$, for every $t\ge0$,
\begin{equation*}
\l| \mathbb{E}_x\l[f\l(\bar{X}_{\un{t}_\pas}^{\pas,x}\r)\r] - \pi^\pas(f) \r| \le \1(x) [f]_1 e^{- {\alpha \un{t}_\pas} }.
\end{equation*}
Such an assumption is adapted to the case, where exponential convergence  in $1$-Wasserstein distance holds for the diffusion and extends to the Euler scheme with sufficiently small step (with a contraction parameter independent of the step). Also note that in the above assumption and in what follows, we implicitly assume that $\pi^{\pas}$ exists (and is unique) for every $\pas\in(0,\eta_0]$. \\

\noindent The following assumption is assumed to hold for $\bea\in[1,2]$:\\

\noindent  $\HDEUX$ ($L^2$-confluence)  A positive constant $\2$ exists such that for all $\pas\in(0, \eta_0]$, 
\begin{equation}\nonumber\label{defbea}
\sup_{t\ge0}\l\| \bar{X}_{\un{t}_\pas}^{\pas,x_0}- \bar{X}_{\un{t}_\pas}^{\pas/2,x_0} \r\|_2\le \2 \pas^{\frac{\bea}{2}}.
\end{equation} 
Note that such an assumption is usually proved by controlling the distance between the Euler scheme and the diffusion (by dividing the error into two parts). In the next section, we will see that for additive diffusion with \textit{strongly contractive} drift,   $\HDEUX$ can be proved with $\bea=1$ or $\bea=2$ with two alternative proofs but leading to constants $\2$ which are strongly different, depending on the Lipschitz constant of $b$ in the first case and {on} the size of the Jacobian matrix $\nabla b$  and the Laplacian $\Delta b$ of the drift $b$ in the second case (see Propositions \ref{prop:stronconvbea1} and \ref{prop:stronconvbea2} for details).\\

\noindent The next assumption is a weak error bound on the distance of the invariant distribution of the diffusion and the one of the Euler scheme {(depending on $\bea$ and $\delta$)}:

\noindent  $\HTROIS$:  there exist{s} a positive constant $\3$  such that for every $\pas\in(0,\eta_0]$,
\begin{equation*}
\l|\pi(f)-\pi^\pas(f)\r| \le \3 [f]_1 \pas^\delta,  
\end{equation*} 
where $\delta\in[ 1/2,1]$ if {$\bea=1$} and $\delta\in(\frac{1+\bea}{4},1]$ if {$\bea>1$}.\\


\noindent The last assumption below is related to the control of the moments of the Euler scheme. It also involves the function $c_1$ defined in $\HUN$.\\

\noindent  $\HQU$: There exists a constant $\4 \ge 0$   such that   for all $\pas\in(0,\eta_0]$, 
\begin{equation*}
\sup_{t\ge0} \left(\|\bar{X}_{t}^{\pas,x_0}-x_0\|_2+\| \1(\bar{X}_{t}^{\pas,x_0})\|_2\right) \le \4.
\end{equation*} 
\begin{theoreme} \label{maintheo}
Assume {$\mathbf{(H_i)}$, $i=1,\ldots,4$} with $\alpha\in(0,1]$, $\bea\in[1,2]$ and for some given $x_0\in\ER^d$ and $\eta_0\in(0,1/2]$. 
Let $f:\ER^d\rightarrow\ER$ be a Lipschitz continuous function. For $\varepsilon \in(0,1)$,  assume that 
\begin{equation}\label{eq:choixpar}
\begin{split}
&\pas_\ind=\pas_0 2^{-\ind}, \quad  \lev_{\varepsilon}= \lceil \frac{1}{\delta}\log_2(r_0{{\varepsilon}}^{-1} )\rceil, \quad T_\ind= \begin{cases}\Tfrak {\varepsilon}^{-2}  2^{-\frac{1+\bea}{2}\ind}&\textnormal{if $\bea>1$}\\
\Tfrak {\varepsilon}^{-2} R_\varepsilon^2 2^{-\ind}&\textnormal{if $\bea=1$.}
\end{cases}
\end{split}
\end{equation}
with $\pas_0\in(0,\eta_0]$, $r_0\ge 1$ and $\Tfrak>0$.\\
\noindent (i) Assume that $\tau\in[\tau_1\log(\varepsilon^{-1})\wedge \frac{1}{2}T_{R_\varepsilon}, \frac{1}{2}T_{R_\varepsilon}]$ with $\tau_1> \frac{1+\bea-2\delta}{2\alpha\delta}$. Then, there exist some constants $\mathfrak{C}_1$ and  $\mathfrak{C}_2$ which do not depend on $f$ (which can be made explicit in terms of the parameters) such that for any $\varepsilon\in(0,1]$,
\begin{equation}\label{eq:pluspetitqueepsilon}
\| \mathcal{Y}(\lev,\l(\pas_\ind\r)_\ind,\tau,\l(T_\ind\r)_\ind,f)- \pi (f) \|_2 \le\mathfrak{C}_1 [f]_1 \varepsilon,
\end{equation}
with a complexity cost,
\begin{equation}\label{relatedcosteps}
\mathcal{C}(\mathcal{Y}) :=\mathcal{C}({\cal Y}(\lev,\l(\pas_\ind\r)_\ind,\tau,\l(T_\ind\r)_\ind,f))\le \mathfrak{C}_2 \begin{cases}  {\varepsilon}^{-2}&\textnormal{if $\bea>1$}\\
 {\varepsilon}^{-2} R_\varepsilon^3&\textnormal{if $\bea=1$}.
\end{cases}
\end{equation} 
where  $\mathfrak{C}_2 = {\cbea} \gamma_0^{-1}{{\Tfrak}}$ with ${\cbea}= \left((1+\frac{3}{2}(2^{\frac{\bea-1}{2}}-1)^{-1}\right)$ if $\bea>1$ and ${\cbea}=\frac{5}{2}$ if $\bea=1$.\\

\noindent (ii) {Assume that $r_0\,{\ge} \,1\vee (\3 \pas_0^\delta)$ and  $\Tfrak\ge \Tfrak_0$, $\tau\in[ \tauun|\log \varepsilon|+\taudeux, \frac{1}{2}T_{R_\varepsilon}]$, $\varepsilon\in(0,\varepsilon_0)$ where $\Tfrak_0$, $\tau_1$, $\tau_2$ and $\varepsilon_0$ are given by \eqref{eq:choiceoftheparameters} and \eqref{tauuntaudeux}. Then,  \eqref{eq:pluspetitqueepsilon} and \eqref{relatedcosteps} hold true for any $\varepsilon\in(0,\varepsilon_0)$ with
$\mathfrak{C}_1\lesssim_{uc} 1$.}
\end{theoreme}
The above theorem  exhibits a family of parameters which lead to a complexity proportional to $\varepsilon^{-2}$ ($resp.$ $\varepsilon^{-2} |\log \varepsilon|^3$) if $\bea>1$ ($resp.$ $\bea=1$).  The first part is adapted to the case where we have few informations about the parameters of the assumption and thus on the constants $\mathfrak{C}_1$ and $\mathfrak{C}_2$ involved in the result. In particular, if $\alpha$ is unknown, we suggest to choose $\tau=\rho T_{R_\varepsilon}$ with $\rho\in(0,1/2]$ (note that in view of alleviating the notations, we omit the dependence {on} $\varepsilon$ for $\tau$).\\

\noindent In the second part, we show that one can tune the parameters  of the procedure in order to obtain an $\varepsilon$-approximation\footnote{More precisely, when $\mathfrak{C}_1=1$, the $L^2$-error is lower than $[f]_1\varepsilon$. To obtain $\varepsilon$, it is certainly enough to replace $\varepsilon$ by $\varepsilon/[f]_1$.}
, up to a universal constant, with an explicit complexity. Note that this universal constant could be avoided with a slight adaptation of the proof (see in particular the proof of Proposition \ref{prop:theo1precis}) which should lead in particular to $\mathfrak{C}_1=1$  instead of  $\mathfrak{C}_1\lesssim_{uc} 1$ (which in turn would modify  ${\cbea}$). Nevertheless, this still introduces many technicalities in the result. We thus chose to introduce universal constants for the sake of readability and will show in some numerical illustrations that this approximation is reasonable. \\

%

\noindent It is also important to remark {that} this second stage certainly depends on the knowledge of the parameters of the diffusion (which is not always accessible in practice). In the next section, we will show that in the strongly convex setting, we can obtain some bounds on the parameters which lead to an accurate estimation of $\mathfrak{C}_2$ in terms of the dimension. 

\noindent Finally, note that in the second part, the result holds true for any $\varepsilon\in(0,\varepsilon_0)$. This  technical  constraint ensures that the warm-start $\tau$ is lower than $\frac{1}{2}T_{R_\varepsilon}$ (which is necessary to get a ``real'' occupation measure). In practice, the simplest is to replace $\tau$ by $\tau\wedge \frac{1}{2}T_{R_\varepsilon}$ in order to avoid such a problem.
\begin{rem} \label{rem:theo21}
$\rhd$ We chose to state this result for a given function $f$.  {Nevertheless, if $\1$ and $\3$ in $\HUN$ and $\HTROIS$ do not depend on $f$, this is certainly possible to write the result, uniformly in the class of Lipschitz functions. Note that with the help of the Kantorovich-Rubinstein representation of $1$-Wassertein distances (see \emph{e.g.} \cite{villani} for background), assuming that $\1$ and $\3$ are uniform in the class of Lipschitz functions is equivalent to suppose that $\HUN$ and $\HTROIS$ are replaced by 
$${\cal W}_1({\cal L}(\bar{X}_{\un{t}_\pas}^{\pas,x}), \pi^\pas)\le \1(x)e^{- {\alpha \un{t}_\pas} }\quad \textnormal{and by}\quad{\cal W}_1(\pi,\pi^\pas) \le \3 \pas^\delta.$$
In this case, Theorem \ref{maintheo} still holds true replacing \eqref{eq:pluspetitqueepsilon} by
\begin{equation}
\sup_{f,[f]_1\le 1} \ES[|\mathcal{Y}(\lev,\l(\pas_\ind\r)_\ind,\tau,\l(T_\ind\r)_\ind,f)- \pi (f) |^2]^{\frac{1}{2}} \le\mathfrak{C}_1  \varepsilon.
\end{equation}
}

%
$\rhd$  In the next section, we will apply the result to additive diffusions in order to be able to get quantitative bounds. Nevertheless, it is worth noting that the result may apply to any (non degenerated) multiplicative diffusion satisfying $\mathbf{(H_i)}$, $i=1,\ldots,4$. For instance, it could be shown that if the diffusion satisfies the \textit{strong confluence} Assumption  $\mathbf{(C_s)}$ of \cite{mainPPlangevin}, the assumptions hold with $\bea=1$ and $\delta\in[1/2,1]$ (see \cite{panloup_pages_lemaire} for results on confluence of diffusions). More generally, the result is in fact not specific to diffusions and may hold for any non degenerated Markov process, equipped with Markovian discretization schemes satisfying assumptions $\mathbf{(H_i)}$, $i=1,\ldots,4$. \\

\noindent {$\rhd$ In the references \cite{McLeish_2011,glynn_rhee_2014,vihola_multilevel}, some debiased Multilevel Monte-Carlo methods have been introduced and studied. In these papers, the idea is to randomize the number of layers (typically with a Poisson distribution) in order to produce completely unbiased estimators of the target. However, such a method seems to rely on the property that the last layer is asymptotically without bias. In this infinite horizon problem, such an adaptation seems to be complicated since even the last layer contains a long-time error which does not vanishes when the step size goes to $0$ (see \cref{rem:multilevelstand} for other details in this direction). Nevertheless, in the spirit of \cite{giles_szpruch_invariant}, in order to get a vanishing long-time error, an idea could be to restart each layer from the final time of the previous one. Note that such an idea is unfortunately incompatible with some parallelization of the layers but it may still deserve to be studied.}
\end{rem}

\subsection{Application to {uniformly} strongly convex additive diffusions}\label{sec:langevinsc}
In this section, we want to focus on the effect of our multilevel strategy on the numerical approximation of the invariant distribution of an additive diffusion when $b$ is strongly contractive, $i.e.$ satisfying  the following Assumption $\Cs$: %
%
%


\noindent $\Cs$ For all $x,y\in\mathbb{R}^d$,
\begin{equation*}
\langle b(y)-b(x) , y-x \rangle \le -\alpuc |x-y|^2.
\end{equation*}
In the context of numerical approximation of Gibbs distributions, this corresponds to the case where $U$ is uniformly strongly convex. If $U$ is ${\cal C}^2$, this is equivalent to suppose that  the Hessian matrix $D^2U$ satisfies: $D^2U\ge \alpha I_d$ with $\alpha>0$ (in a sense of symmetric matrices).  Note that  this assumption can be viewed as restrictive. However, our main objective in this paper is to sharply evaluate the effect of such multilevel strategies in this nice and benchmark setting (see Remark \ref{rem:comparison} for a discussion about potential extensions). \smallskip

Now, note that, with the help of the inequality $\langle u,v\rangle\le (2\alpuc)^{-1}|u|^2+(\alpuc/2)|v|^2$), $\Cs$ (applied with $y=0$)  implies the Lyapunov (or stability) assumption 
$$ 2 \langle b(x),x\rangle \le 2 \langle b(0),x\rangle-2\alpuc|x|^2\le \frac{|b(0)|^2}{\alpuc}-\alpuc|x|^2.$$
Such a Lyapunov assumption classically implies the existence of $\pi$ and that of $\pi^\pas$ for $\pas\in(0,\alpha/(2\blip^2)]$ (see Lemma \ref{lem:boundEuler}$(i)$). Uniqueness follows from the non-degeneracy of the dynamical system since $\sigma>0$.


%
\noindent The following parts are dedicated to the study, in this uniformly contractive setting, of the complexity of the  related multilevel procedure and to the dependency in the dimension of their constants . As mentioned before, the results will strongly depend on the value of $\bea$. In the two next sections, we propose two types of results, with $\bea=1$ and $\bea=2$ respectively. The first case is based on simpler bounds and only requires $U$ to be ${\cal C}^2$ (with bounded Hessian matrix) but the related dependence {on} $\varepsilon$ is not completely optimal (proportional to $\varepsilon^{-2} |\log^3 \varepsilon|$). The second case will lead to a complexity proportional to $\varepsilon^{-2}$ but with refinements which require slightly more constraining  assumptions on $U$. Note that $\bea=2$ is really specific to additive diffusions whereas $\bea=1$ may extend to multiplicative diffusions (as mentioned above in Remark \ref{rem:theo21}). 



\subsubsection{\texorpdfstring{$\bea=1$}{a=1} and \texorpdfstring{$\delta=1/2$}{delta=1/2}}

\begin{prop}\label{prop:stronconvbea1} Assume $\Cs$ and  $b$ $L$-Lipschitz {with $\alpha\in(0, L\wedge 1]$}.  Let $x_0 \in \ER^d$ and $\eta_0 \in (0,{\frac{\alpha}{2\bun^2}\wedge \frac{1}{2}}]$.  Then, $\HUN$, $\HDEUX$, $\HTROIS$ and $\HQU$ hold true for $\bea=1$, $\delta=1/2$ (and for any Lipschitz continuous function $f:\ER^d\rightarrow\ER)$ with
\begin{equation}\label{eq:upsilon}
\begin{cases}
\1(x)\le |x-x_0|^2+\sqrt{2|b(x_0)|^2\l(\frac{1}{\blip^2} + \frac{2}{\laminf^2}\r)+\frac{\sigma^2 d}{\laminf}},\\
{\max(\frac{\alpha \2^2}{L^2}, \frac{\alpha\3^2}{L^2},\4^2 )\lesssim_{uc} 1\vee \l(\alpha^{-2} |b(x_0)|^2+ \alpha^{-1}\sigma^2d\r)=:{\Upsilon_1^2}}.
\end{cases}
\end{equation} 

\end{prop}

\begin{rem} Note that when $b=-\nabla U$ and $x_0={\rm Argmin}_{x\in\ER^d}U$, the above bounds are simplified and have  a better dependence {on} $\alpha$. In particular, this clearly suggests  to start with this value of $x_0$. However, we kept the general bounds since $x_0$ is not always known in practice.
\end{rem}
As a corollary of this proposition and of Theorem \ref{maintheo}, the following theorem provides a first estimate of the cost of the multilevel procedure under $\Cs$:
\begin{theoreme} \label{maincoro1}
Let $f:\ER^d\rightarrow\ER$ be a Lipschitz continuous function and assume $\Cs$ with $\alpha\in(0,L\wedge 1]$. Let $x_0 \in \ER^d$ and suppose $\pas_0 \in (0,{\frac{\alpha}{2\bun^2}\wedge \frac{1}{2}}]$. Let ${\Upsilon_1}$ be defined by \eqref{eq:upsilon}. 
For $\varepsilon>0$, let  
$$r_0={{\Upsilon_1}},\quad R_\varepsilon= \lceil 2\log_2(r_0\varepsilon^{-1})\rceil,\quad T_r= \frac{\Upsilon_1^2\log(\pas_0^{-1})}{\alpha} \varepsilon^{-2} R_\varepsilon^2 2^{-r}, \quad r\in\{0,\ldots,R_\varepsilon\},$$
and $\tau=\tau_1|\log(\varepsilon)|+\tau_2$ with
$\tau_1=2\alpha^{-1}$ and $\tau_2={\alpha^{-1}\log{\Upsilon_1}}$. Set 
\begin{equation}\label{def:epsilonzero}
{\varepsilon_0:=\max\{\varepsilon\in(0,1], 2\log({\Upsilon_1}\varepsilon^{-2})\le \log(\pas_0^{-1}) R_\varepsilon^2\}.}
\end{equation}
 Then,  \eqref{eq:pluspetitqueepsilon}  holds true for any $\varepsilon\in(0,\varepsilon_0)$ with
$\mathfrak{C}_1\lesssim_{uc} 1$ and 
\begin{equation}\label{eq:complex2}
{{\cal C}({\cal Y})\le \frac{5\log(\pas_0^{-1})}{2\alpha\pas_0}\Upsilon_1^2 \varepsilon^{-2} \lceil \log^3_2(\Upsilon_1^2 \varepsilon^{-2})\rceil,}
\end{equation}
Furthermore, if $\alpha/L^2\le 1$, if $\sigma^2 \alpha^{-1} d\ge 1$  and if $|b(x_0)|^2\lesssim_{uc} \sigma^2 \alpha d$, then the result is true with $\tilde{\Upsilon}_1^2= \sigma^2 \alpha^{-1} d$ and $\gamma_0=\alpha/(2L^2)$ with a complexity satisfying
\begin{equation}\label{eq:complex22}
{{\cal C}({\cal Y})\le 5\log\l(\frac{2 L^2}{\alpha}\r)\frac{L^2}{2\alpha^3}\sigma^2 d\varepsilon^{-2}\lceil\log_2\l( \frac{\sigma^2d}{\alpha} \varepsilon^{-2}\r)\rceil^3.}
\end{equation}
%
%
%
%
%
\end{theoreme}
\begin{rem} 
$\rhd$ {Note that the choice of $\pas_0$ does not depend on $d$ (but only on $\alpha$ and $L$). This is due to the fact that, in this strongly convex setting, this is possible to control uniformly the ergodicity and the contraction properties of the dynamics of the Euler scheme as soon as $\gamma_0\le \alpha/(2L^2)$. It is not clear that such a  property remains true in the weakly convex setting where such controls are generally difficult to obtain especially in the discretized setting.}\\

$\rhd$ If we {skip} the dependence {on} $\alpha$ and $L$ and choose to only focus on the one {on} $d$ and $\varepsilon$\footnote{This is usual in the literature.}, one can remark that, {as soon as $|b(x_0)|\le C \sqrt{d}$}, the cost of the procedure is of order
$d\varepsilon^{-2}\left(\log^3 d+\log^3\varepsilon\right)$. We can thus say that we are at a ``logarithmic  distance'' of the ``optimal'' cost $d\varepsilon^{-2}$.\\
\end{rem}
\noindent \textbf{Gibbs distribution approximation I:} Let us apply the above result to the approximation of $\pi_U =Z_U^{-1} e^{-U}d\lambda_d$. Set
$$\bar{\lambda}_U=\sup_{x\in\ER^d}\bar{\lambda}_{D^2 U(x)}\quad \textnormal{and} \quad \underline{\lambda}_U=\inf_{x\in\ER^d}\underline{\lambda}_{D^2 U(x)}$$
where for a symmetric matrix $A$, $\bar{\lambda}_A$ and $\underline{\lambda}_A$ respectively denote the largest and  lowest eigenvalues of $A$. In the sequel, we assume that
\begin{equation}\label{hyp:valeurpropre}
0<\underline{\lambda}_U\le \bar{\lambda}_U<+\infty.
\end{equation}
In this case, $\Cs$ holds with $\alpha_U=\underline{\lambda}_U\wedge 1$ and $\nabla U$ is $L_U$-Lipschitz with $L_U=\bar{\lambda}_U$. 

\noindent Furthermore, for any $\sigma_0>0$, $\pi_U$ is the invariant distribution of the diffusion given with $b_{\sigma_0}=-\sigma_0^{2}\nabla U$ and diffusion coefficient ${\sigma}=\sqrt{2}\sigma_0$. Then, it is natural to ask about the choice of $\sigma_0$, especially in terms of $\alpha_U$ and ${L_U}$. We obtain the nice following result:
\begin{cor}\label{cor:gibbsI} Let $(X_t^{(\sigma_0)})_{t\ge0}$ denote the solution to $dX_t^{(\sigma_0)}=b_{\sigma_0}(X_t^{(\sigma_0)}) dt+\sqrt{2}{\sigma_0} dB_t$ with $b_{\sigma_0}=-\sigma_0^{2}\nabla U$. Assume that \eqref{hyp:valeurpropre} holds true. Then, for any $\sigma_0>0$, $(X_t^{(\sigma_0)})_{t\ge0}$ admits $\pi_U$ as an unique invariant distribution. Furthermore,   $\Cs$ holds with $\alpha_{\sigma_0}=\sigma_0^2{\alpha_U}$ and $b_{\sigma_0}$ is $L_{\sigma_0}$-Lipschitz with $L_{\sigma_0}=\sigma_0^2 {L_U}$.
Then, if 
$$\sigma_0^2=\frac{{\alpha_U}}{{L_U^2}},\quad \gamma_0=\frac{1}{2},\quad r_0={ \sqrt{\frac{d}{2\alpha_U}}},\quad  {|\nabla U(x_0)|\lesssim_{uc} \alpha_U^3 L_U^{-4} d}\quad \textnormal{and}\quad \Upsilon_1^2= d{\alpha_U^{-1}},
$$
 \eqref{eq:pluspetitqueepsilon}  holds true  for any $\varepsilon\in(0,\varepsilon_0)$ ($\varepsilon_0$ being defined by \eqref{def:epsilonzero})  with
$\mathfrak{C}_1\lesssim_{uc} 1$ and
\begin{equation*}
{{\cal C}({\cal Y})\le \frac{ 5 \log 2}{2}\frac{L_U^2}{\alpha_U^3} d  \varepsilon^{-2} \lceil\log_2\l(  {\alpha_U^{-1}} d\varepsilon^{-2}\r)\rceil^3.}
\end{equation*}
\end{cor}
Keeping in mind that $\alpha_{\sigma_0}=\sigma_0^2 \alpha_{U}$, we remark that in the previous corollary,
$T_r= L_U^2\alpha_U^{-3}\log(2)\varepsilon^{-2} R_\varepsilon^2 2^{-r}$.
\begin{proof} 
Let $\sigma_0>0$. First remark, that since $\alpha_U\le L_U$, then $\alpha_{\sigma_0}\le L_{\sigma_0}$ and hence, $\alpha_{\sigma_0}\in(0,L_{\sigma_0}\wedge 1]$ if $\sigma_0^2\le \alpha_{U}^{-1}$. Then, by Theorem \ref{maincoro1}, for any $\sigma_0\in(0,\alpha_U^{-1}]$ such that $\alpha_{\sigma_0}/L_{\sigma_0}^2\le 1$,
\begin{equation}\label{eq:controleoptimall}
{{\cal C}({\cal Y})\le 5\log\l(\frac{2 \sigma_0^2 L_U^2}{\alpha_U}\r)\frac{L_U^2}{2\alpha_U^3} d\varepsilon^{-2}\lceil\log_2\l( \frac{d}{\alpha_U} \varepsilon^{-2}\r)\rceil^3.}
\end{equation}
But $\alpha_{\sigma_0}/L_{\sigma_0}^2\le 1$ if and only if $\sigma_0^2\ge {\alpha_U}/{L_U^2}$ so that one can set $\sigma_0^2= {\alpha_U}/{L_U^2}$ in the above inequality. In this case, one remarks that 
$$\frac{\alpha_{\sigma_0}}{L_{\sigma_0}^2}=\frac{ {\alpha_U}}{\sigma_0^2 {L_U^2}}=1$$
so that $\pas_0={1}/{2}$ and $\Upsilon_1^2=d{\alpha_U}^{-1}$. The result follows.
\end{proof}
\begin{rem} {In the above proof, we {choose the lowest} value of $\sigma_0$ {under which}  $\alpha_{\sigma_0}/L_{\sigma_0}^2\le 1$. The theoretical interest is to {remove} a logarithmic dependence {on} $L$ and $\alpha$. From a practical point of view, this normalization leads to a simplification of the parameters.}\\

\noindent {The simplest choice for $x_0$ is certainly $x_0={\rm Argmin}_{x\in\ER^d} U(x).$ In the case where $x_0$ is unknown, we suggest to introduce an optimization preprocess  in order to start the procedure with an initial point which is not so far from the minimizer of $U$ (or, more precisely, which sastisfies ${|\nabla U(x_0)|\le  \alpha_U^3 L_U^{-4} d}$).}

\end{rem}


%
%
%

\subsubsection{\texorpdfstring{$\bea=2$}{a=2} and \texorpdfstring{$\delta=1$}{delta=1}: optimal complexity with slightly more constraining assumptions}
Let us assume that $b$ is ${\cal C}^2$   and let us introduce the following notations: $\nabla b=[\partial_j b_i]_{1\le i\le j\le d}$, the Jacobian matrix of $b$ and $\Delta b=(\Delta b_i)_{i=1}^d$, the vector of Laplacians of $b_1,\ldots,b_d$ where we recall that  for a given function $\phi:\ER^d\rightarrow\ER$,
$$\Delta \phi=\sum_{i=1}^d \partial^2_{x_i^2} \phi.$$
If $b$ has bounded partial derivatives up to order $2$, we can define:
\begin{equation}\label{def:jaclap}
\|\nabla b\|_{2,\infty}^2=\sup_{x\in\ER^d} \|\nabla b(x)\|_F\quad\textnormal{and}\quad \|\Delta b\|_{2,\infty}
=\sup_{x\in\ER^d}|\Delta b(x)|^2=\sup_{x\in\ER^d}\sum_{i=1}^d |\Delta b_i(x)|^2,
\end{equation}
where $\|\,.\,\|_F$ stands for the Frobenius norm (see Section \ref{sec:notations} for a definition). We are now ready to provide some new bounds related to $\HDEUX$ and $\HTROIS$ when $\bea=2$ and $\delta=1$ (The results for $\HUN$ and $\HQU$ obtained in Proposition \ref{prop:stronconvbea2} still hold true).

\begin{prop}\label{prop:stronconvbea2} Assume that $\Cs$ holds true and that $b$ is $L$-Lipschitz and ${\cal C}^2$ with bounded partial derivatives. Let $x_0 \in \ER^d$ and suppose $\pas \in (0, {\frac{\alpha}{2\bun^2}\wedge \frac{1}{2}}]$.   Then, $\HDEUX$, $\HTROIS$ and $\HQU$ hold for $\bea=2$, $\delta=1$ (and for any Lipschitz continuous function $f:\ER^d\rightarrow\ER)$ with $\max\l(\frac{\alpha^2 \2^2}{L^4}, \frac{\alpha^2\3^2}{L^4},\4^2\r)\lesssim_{uc}\Upsilon_2^2$ where, 
\begin{equation}\label{eq:upsilon2}
{\Upsilon_2^2=\max\l(1,\frac{1}{\alpha^2} |b(x_0)|^2+ \frac{\sigma\sqrt{\alpha}}{L^3}\|\nabla b\|_{2,\infty}|b(x_0)|+\frac{\sigma^4}{L^4}\|\Delta b\|_{2,\infty}^2+ \frac{{\sigma^2\alpha d^{\frac{1}{2}}}}{L^3}\|\nabla b\|_{2,\infty}+\frac{\sigma^{2} d }{\alpha}\r) }.\end{equation} 
\end{prop}
\begin{rem} In the Ornstein-Uhlenbeck case ($b(x)=-x$ and $\sigma=\sqrt{2}$), which can be viewed as the simplest toy-model, we remark that if $x_0=0$, then $\Upsilon_2^2=\max(1,2\sigma^2d)$  since $|b(0)|=0$, $\alpha=L=1$, $\|\Delta b\|_{2,\infty}^2=0$ and $\|\nabla b\|_{2,\infty}=d^\frac{1}{2}$. In the general case,  these constants strongly depend on $b$  and on the behavior  of $\nabla b$ and $\Delta b$. However, there are model-specific and it is difficult to state a general result taking really into account this dependency. Nevertheless, in Theorem  \ref{maincoro2}, we will provide some {fairly explicit} conditions on $\|\nabla b\|_{2,\infty}$ and on $\|\Delta b\|_{2,\infty}^2$ {under which} these dependencies are controlled (Note that the operator $\Delta b$ also appears in Assumption $\mathbf{H3}$ of \cite{durmusmoulineshigh}).

\noindent Let us also remark that if $b$ has the following form:
\begin{equation}\label{eq:particularcasebi}
b_i(x)={\phi}_i(x_{j_1},\ldots,x_{j_m}), \quad 1\le m\le d,
\end{equation}
where $\phi_1$,\ldots, $\phi_d$ are ${\cal C}^2$-functions with  partial derivatives (up to order $2$) bounded by dimension-free constants, then $\|\nabla b\|_{2,\infty}d^{-\frac{1}{2}}+ \|\Delta b\|_{2,\infty}^2 d^{-1}\le C_m$ where $C_m$ does not depend on $d$. 
\end{rem}
\noindent As in the preceding part, we can now deduce a result as a corollary of this proposition and of Theorem \ref{maintheo}.
\begin{theoreme} \label{maincoro2}
Assume that $\Cs$ holds true with $\alpha\in(0,L\wedge 1]$ and that $b$ is $L$-Lipschitz and ${\cal C}^2$ with bounded partial derivatives. Let $f:\ER^d\rightarrow\ER$ be a Lipschitz continuous function. Let $x_0 \in \ER^d$ and suppose that $\pas_0 \in (0,{\frac{\alpha}{2\bun^2}\wedge \frac{1}{2}}]$. Let ${\Upsilon_2}$ be defined by \eqref{eq:upsilon2}. 
For $\varepsilon>0$, let  
$${r_0= {\Upsilon_2}},\quad, R_\varepsilon= \lceil \log_2(r_0\varepsilon^{-1})\rceil,\quad T_r= \frac{{\Upsilon_2^2}{\log(\pas_0^{-1})} }{\alpha} \varepsilon^{-2}  2^{-\frac{3}{2}r}, \quad r\in\{0,\ldots,R_\varepsilon\},$$
and $\tau=\tau_1|\log(\varepsilon)|+\tau_2$ with
$\tau_1=\alpha^{-1}$ and ${\tau_2=(2\alpha)^{-1}\log{\Upsilon_2}}$. Set 
$${\varepsilon_0:=\max\{\varepsilon\in(0,1], \sqrt{2}\log({\Upsilon_2}\varepsilon^{-2})\le \log(\pas_0^{-1}) (\Upsilon_2 \varepsilon^{-1})^{\frac{1}{2}}\}.}$$
 Then,  \eqref{eq:pluspetitqueepsilon}  holds true for any $\varepsilon\in(0,\varepsilon_0)$ with
$\mathfrak{C}_1\lesssim_{uc} 1$ and 
\begin{equation}\label{eq:complex3}
{{\cal C}({\cal Y})\le \frac{\frac{1}{2}+\sqrt{2}}{\sqrt{2}-1}\frac{{\Upsilon_2^2} \log(\pas_0^{-1})}{\gamma_0\alpha} \varepsilon^{-2}.}
\end{equation}
In particular, if  $\alpha/L^2\le 1$, $\sigma^2 \alpha^{-1} d\ge 1$, {$|b(x_0)|^2\lesssim_{uc} \sigma^2 \alpha d $},  $\sigma^2\|\Delta b\|_{2,\infty}^2 \lesssim_{uc} \alpha^{-1} L^4 d$ and 
$\|\nabla b\|_{2,\infty}\lesssim_{uc} \alpha^{-2} L^3  \sqrt{d}$, then for
 $\gamma_0=\alpha/(2L^2)$, { the conclusion is true with  $\tilde{\Upsilon}_2^2=$
  $ \sigma^2 \alpha^{-1} d$ }leading to the following complexity bound: 
\begin{equation}\label{eq:complex4}
{{\cal C}({\cal Y})\le \frac{2\sqrt{2}+1}{\sqrt{2}-1}  \frac{\sigma^2 L^2}{\alpha^3} { \log \left( \frac{2 L^2}{\alpha}\right)} d \varepsilon^{-2}.}
\end{equation}
%
%
%
%
%
\end{theoreme}
The coefficient $\frac{\frac{1}{2}+\sqrt{2}}{\sqrt{2}-1}$ corresponds to {$\cbea$} in Theorem  \ref{maintheo}. Note that in  all the results, we give the explicit complexities since it may be convenient for practice. From a theoretical point of view, these explicit bounds do not give more information than some bounds up to universal constants since the $\varepsilon$-approximation is always obtained up to a universal constant $\mathfrak{C}_1$ which changes with the normalization of $\Upsilon_2^2$ (on this point, see Remark \ref{rem:withnonexplicitconstants}). 

\begin{rem}\label{rem:tehorr} {
\smallskip
It is worth noting that in this result, we attain a complexity proportional to  $d\varepsilon^{-2}$. As mentioned in the introduction, this means that if we forget for a moment, the intrinsic dependence {on} $L$ and $\alpha$, one attains a complexity which is of the same order as a Monte-Carlo method without bias.}



%


\end{rem}
\noindent \textbf{Gibbs distribution approximation II:} As in the previous section, we apply this theorem to the approximation of the Gibbs distribution (with the same notations) and obtain the following result. We use the same notations as in Corollary \ref{cor:gibbs2} introducing for a positive $\sigma_0$, the diffusion  $dX_t=b_{\sigma_0}(X_t)dt+\sqrt{2}\sigma_0 dB_t$ with $b_{\sigma_0}=-\sigma_0^{2}\nabla U$, which admits $\pi_U =Z_U^{-1} e^{-U}d\lambda_d$ as a unique invariant distribution (for any $\sigma_0>0$). We obtain the following result:

{ \begin{cor}\label{cor:gibbs2} Let the assumptions of Corollary \ref{cor:gibbsI}  be in force with $U:\ER^d\rightarrow\ER$ ${\cal C}^3$ with bounded partial derivatives. Set  $$\varepsilon_0:=\max\left\{\varepsilon\in(0,1], \sqrt{2}\log\left((d \alpha_U^{-1})^{\frac{1}{2}}\varepsilon^{-2}\right)\le \log(2)( d{\alpha_U^{-1}})^{\frac{1}{4}} \varepsilon^{-\frac{1}{2}}\right\}.$$ 
Then if 
$$\sigma_0^2=\frac{{\alpha_U}}{{L_U^2}},\quad \gamma_0=1/2,  \quad \Upsilon_2^2= d\alpha_U^{-1},\quad \|\Delta (\nabla U)\|_{2,\infty}^2 \lesssim_{uc} {\alpha_U^{-1}} L_U^4 d \quad \textnormal{and} \quad |\nabla U(x_0)|^2\lesssim_{uc} \alpha_U d,  $$
\eqref{eq:pluspetitqueepsilon}  holds true for any $\varepsilon\in(0,\varepsilon_0)$, with $\mathfrak{C}_1\lesssim_{uc} 1$  and,
\begin{equation}\label{eq:costgibbs2}
{\cal C}({\cal Y})\le \frac{\log 2(2\sqrt{2}+1)}{\sqrt{2}-1} \frac{ {L_U^2}}{{\alpha_U^3}} d  \varepsilon^{-2}.
\end{equation}
\end{cor}}

\begin{proof}
We apply the second part of Theorem \ref{maincoro2}  with $b_{\sigma_0}=-\sigma_0^2 \nabla U$, $L_{\sigma_0}=\sigma_0^2 L_U$ and $\alpha_{\sigma_0}=\sigma_0^2 \alpha_U$. The additional assumption on $\|\Delta (\nabla U)\|_{2,\infty}^2$ implies that $\sigma_0^2\|\Delta b_{\sigma_0}\|_{2,\infty}^2 \lesssim_{uc} \alpha_{\sigma_0}^{-1} L_{\sigma_0}^4 d$.  One also checks that the condition on $\|\nabla b_{\sigma_0}\|_{2,\infty}$ of Theorem \ref{maincoro2} holds true if $\|D^2U\|_{2,\infty}=\sup_{x\in\ER^d} \|D^2 U (x)\|_F\lesssim_{uc} {\alpha_U^{-3}} L_U^4 \sqrt{d}$ but this condition is always satisfied: actually, for a symmetric matrix $A$, 
$$\|A\|_F^2=\sum_{1\le i,j\le d} A_{i,j}^2={\rm Tr}(A^2)\le (\bar{\lambda}_A)^2 d$$
so that
\begin{equation}\label{ineq:linearalgebra}
\|D^2U\|_{2,\infty}\le \bar{\lambda}_U \sqrt{d}\le (\alpha_U^{-3}L_U^3) L_U \sqrt{d},
\end{equation}
since $\alpha_U^{-1} L_U\ge 1$.
{Finally the condition $|\nabla U(x_0)|^2\lesssim_{uc} \alpha_U d$ ensures that $|b_{\sigma_0}(x_0)|^2\lesssim_{uc} \sigma_0^2 \alpha_0 d. $
In this setting $\Upsilon_2^2=\sigma_0^2 \alpha_{\sigma_0}^{-1} d=\alpha_U^{-1} d$
, then for $\varepsilon\in(0,\varepsilon_0) $ we get the result since 
\begin{equation*}
\frac{2\sqrt{2}+1}{\sqrt{2}-1}  \frac{\sigma_0^2 L_{\sigma_0}^2}{\alpha_{\sigma_0}^3} \log \left( \frac{2 L_{\sigma_0}^2}{\alpha_{\sigma_0}}\right)d \varepsilon^{-2}=\frac{\log 2(2\sqrt{2}+1)}{\sqrt{2}-1} \frac{ {L_U^2}}{{\alpha_U^3}} d  \varepsilon^{-2}.
\end{equation*}
Note that the other parameters have the following form
$$ r_0= \sqrt{d{\alpha_U^{-1}}}, \quad R_\varepsilon= \lceil \log_2(r_0\varepsilon^{-1})\rceil,\quad T_r= \frac{d L_U^2 \log(2)}{\alpha_U^3} \varepsilon^{-2}  2^{-\frac{3}{2}r}, \quad r\in\{0,\ldots,R_\varepsilon\},$$
and $\tau=\tau_1|\log(\varepsilon)|+\tau_2$ with
$\tau_1=L_U^2 \alpha_U^{-2}$ and $\tau_2=\frac{1}{2}L_U^2 \alpha_U^{-2}\log\left(d{\alpha_U^{-1}}\right)$.}

%
\end{proof}
\begin{rem}\label{rem:bettercomplexf}{$\rhd$ Once again, the best choice for $x_0$ is $x_0=\textnormal{Argmin}_{x\in\ER^d} U(x)$. If this $x_0$ is not explicit, we use a classical optimization preprocess in order to start the procedure with an initial point which is not so far from the minimizer of $U$. }\\
$\rhd$ It is worth noting that ${\cal C}({\cal Y})$ has the same dependence {on} ${L_U}$ and ${\alpha_U}$ as in Corollary \ref{cor:gibbsI}.  This implies that, up to the additional condition on $\Delta (\nabla U)$, this result strictly improves   Corollary \ref{cor:gibbsI} since the logarithmic term disappeared. This additional condition is in fact very reasonable in practice. For instance, owing to \eqref{ineq:linearalgebra}, we remark that  it is satisfied if $\|\Delta (\nabla U)\|_{2,\infty}^2 \lesssim_{uc} {\alpha_U^{-1}} L_U^3 \|D^2 U\|_{2,\infty}^2$, $i.e.$ if
$$\sup_{(i,j,x)\in\{1,\ldots,d\}^2\times\ER^d} |\partial_{i,j,j}^3 U(x)|\lesssim_{uc} {\alpha_U^{-1}} L_U^3\sup_{(i,j,x)\in\{1,\ldots,d\}^2\times\ER^d} |\partial_{i,j}^2 U(x)|.$$

\end{rem}
\begin{rem}\label{rem:comparison} Let us end this section with some comments and some comparisons with the literature.
To the best of our knowledge, this paper is the first which provides an algorithm for the approximation of Gibbs distribution with an $\varepsilon$-complexity {of the order} $d\varepsilon^{-2}$. In the literature, it seems that the most comparable paper is   \cite{durmusmoulineshigh} where a part of the work is devoted to occupation measures of Euler schemes and where the authors obtain a complexity in ${\cal O}(d\varepsilon^{-4})$ (or ${\cal O}(d\varepsilon^{-3})$ if $f$ is bounded). The dependence {on} $\alpha$ and $L$ is also mentioned by the authors and a careful reading of their results {seems to lead to $\alpha_U^{-4} L_U^2$}. {This shows that, in this strongly convex setting, our multilevel procedure is able to improve the dependence in $\varepsilon$ without affecting the dependence in the dimension.}\\

\noindent In fact, in the literature, the study of the dependence {on} the dimension of Langevin methods is usually focused on the (Wasserstein/Total Variation) distance between the random variable produced by the algorithm and the Gibbs distribution. This is why the authors usually define the complexity as the number of iterations to sample one random variable whose distribution is at a distance lower than $\varepsilon$ from the target $\pi$. 
In Wasserstein distance, it seems that the best bounds for this number of iterations are {of the order} $d\varepsilon^{-2}$ ($d\varepsilon^{-1}$ for the total variation distance, see $e.g.$ \cite{dalalyan2020bounding} or \cite{durmusmoulineshigh}). Nevertheless, if one wants to deduce from these bounds a Monte-Carlo method which generates an approximation of $\int f(x)\pi(dx)$ with an MSE lower than $\varepsilon^2$, one needs to compute $N_\varepsilon\approx {\rm Var}_\pi(f) \varepsilon^{-2}$. As aforementioned, for a general $1$-Lipschitz function, ${\rm Var}_\pi(f)$ is ``of the order $d$'' so that the real complexity which would be deduced from these Wasserstein bounds is in fact ${\cal O}(d^2\varepsilon^{-3})$ (or ${\cal O}(d\varepsilon^{-3})$ in the particular case where $f$ is bounded since ${\rm Var}_\pi(f)$ is bounded in this case).\\

\noindent To conclude, let us remark that many papers now focus on the non strongly convex setting or at least try to develop a ``more robust strongly convex setting''.  Actually, in spite of the optimization of $\sigma_0$, our results show that, even in the strongly convex settings,  the complexities are very sensitive to the contraction and Lispchitz parameters so that in cases where $\alpha$ is very small or  $L$ {is} very large, the {computation cost} may explode. In this {case}, the complexity in ${\cal O}(d\varepsilon^{-2})$ becomes a little ``symbolic'' and some other ideas must be developed to manage this problem. On this topic, we propose an opening in Section \ref{subsection:opening} in a particular example where numerical computations (and heuristics)  show that an  increase of the value of $\pas_0$ leads to strongly better performances.
%

\noindent More generally, {extending} multilevel methods {to} such pathologic situations {seems to require to be able to preserve the} contraction property $\HDEUX$. Actually, $\HDEUX$ seems to be fundamental for the control of the variance of the correcting levels. Then, even if  in numerical  simulations, we remarked that the $L^2$-confluence of the Euler schemes seems to be still effective in some non strongly convex settings {(on this topic, see \cref{sub:non-convex-setting})}, the theoretical extension is a clearly difficult task.  {In the weakly convex setting, a first idea could be to adapt the penalized Langevin algorithm proposed in  \cite{karagulyan2020penalized}. In this paper, the authors  regularize a weakly convex potential by a uniformly strongly convex one and hence, approximate the regularized target. This approach would probably extend to our multilevel setting.
Still in the weakly convex setting, another viewpoint has been proposed in \cite{gadat2020cost} by considering convex potentials (of the \textit{Polyak-Lojasiewicz} type) with positive but vanishing at infinity Hessian matrix.  In this case, refined convex arguments seem to lead to some weak forms of the main assumptions $\HUN$, $\HDEUX$, $\HTROIS$ and $\HQU$. In particular, one may preserve the difficult $L^2$-confluence assumption $\HDEUX$ and  quantitative bounds may probably follow but with a worse dependence in the dimension. Finally, in the non-convex setting,  the recent paper \cite{Flammarion_Wainwright} proved quantitative bounds for the unadjusted Langevin algorithm with some arguments based on the comparison of the semi-groups of continuous-time and discretized dynamics and some log-Sobolev\footnote{Even if this assumption is usually difficult to check in practice without strong convexity, such an assumption opens the way to quantitative bounds in the non-convex setting.} contraction assumptions on the target probability.  Precisely, such  results could lead to assumptions  $\HTROIS$ and $\HUN$ respectively. Unfortunately, the $L^2$-confluence assumption $\HDEUX$ requires other type-arguments and makes this issue an open problem.}

%

\end{rem}

\subsection{Numerical illustrations}\label{sec:numillu}
This section is devoted to some numerical illustrations in some toy models. We only investigate the setting of Corollary \ref{cor:gibbs2} which produces a lower complexity (see Remark \ref{rem:bettercomplexf})  and focus on two examples. In the first classical Ornstein-Uhlenbeck, we  {detail} the choices of parameters and discuss the practical efficiency with respect to the theoretical one. In the second example, we focus on a model with a non-quadratic potential and where the constants $\alpha_U$ and $L_U$ are really different from one in order to emphasize the interest  the optimization of $\sigma_0$ proposed in  Corollary \ref{cor:gibbs2}.
\subsubsection{Ornstein-Uhlenbeck}\label{sub:ououou}
We propose to compute $\int_{\ER^d} f(x)\pi(dx)$ where $\pi={\cal N}(0,I_d)$ and $f(x)=|x|$ ($=(\sum_{i=1}^d x_i^2)^{\frac{1}{2}}$). When $d=2k$, $k\in\mathbb{N}^*$,
\begin{equation}\label{eq:IFpi}
I_f=\int_{\ER^d} f(x)\pi(dx)=\frac{(2k)!\sqrt{2\pi}}{2^{2k}k! (k-1)! }.
\end{equation}
The distribution $\pi$ being the invariant distribution of the Ornstein-Uhlenbeck process solution to $\textnormal{d}X_t=-X_t\textnormal{d}t +\sqrt{2}\textnormal{d}\B_t$,
the idea is certainly to apply the multilevel procedure to this process. Note that since $U(x)=|x|^2/2$, we have $\alpha_U=L_U=1$ so that the positive number  $\sigma_0$ of Corollary \ref{cor:gibbs2} is equal to $1$. Taking the parameters given in this corollary, we set:
$$\gamma_0=\frac{1}{2},\quad x_0=0, \quad \Upsilon_2^2=d,$$
so that for any $d\ge2$, for any $\varepsilon\in(0,1)$,
$$ R_\varepsilon=\lceil \log_2(\varepsilon^{-1}\sqrt{{d}} )\rceil,\quad T_r=d\varepsilon^{-2} 2^{-\frac{3}{2}r},\; r=0,\ldots,R,\quad \textnormal{and}\quad \tau=|\log \varepsilon|+\frac{1}{2} \log d.$$
\begin{rem} Note that {in all the} simulations, we choose, for the sake of simplicity to  set $T_r=L_U^2\alpha_U^{-3} d\varepsilon^{-2} 2^{-\frac{3}{2}r}$ instead of $T_r=\log(2) L_U^2\alpha_U^{-3}  d\varepsilon^{-2} 2^{-\frac{3}{2}r}$. By Theorem \ref{maintheo}$(ii)$ (where there is a lower-bound on $\Tfrak$), this does not change the conclusion {except} the cost which is divided by $\log 2$.
\end{rem}

With $d=10$ and $\varepsilon=0.1$, we first provide a simulation giving the contributions of each level. In this case, $I_f\approx 3.084$. The multilevel procedure applies with $R=5$.  In Table \ref{table1}, we give the number of iterations of the Euler scheme for each level and the evolution of the estimation after each level. 
\begin{table}[htbp]
\begin{center}
\begin{tabular}{| l | c | r | c| c|c|c|}
     \hline
     Level&0&1 & 2 & 3 &4&5\\ \hline
     Number of iterations &2000 & 2124 & 1497 &1059&747&531\\ \hline
      Estimation& 3.579 & 3.315 & 3.204 &3.149&3.118&\bf{3.105}\\
     \hline
     
   \end{tabular}
   \end{center}
   \caption{\label{table1}Evolution of the multilevel procedure with $r$. Theoretical value$\approx 3.084$. }
 
\end{table}
The total number of iterations of the Euler scheme (complexity) is equal to 7958 whereas the theoretical bound given in Corollary \ref{cor:gibbs2} is equal to $9243$ ( since $(2\sqrt{2}+1)(\sqrt{2}-1)^{-1}\approx9.423$). This difference comes from the fact that the bound on the complexity is obtained by a computation of the series 
$\sum_{r=1}^{+\infty} 2^{-\frac{r}{2}}$ whereas here it only involves $\sum_{r=1}^{5} 2^{-\frac{r}{2}}$. Note that the algorithm is compatible with parallelization procedures since the levels can be computed independently. {However, it is worth noting that the degree of parallelization of such a multilevel method is completely different from the traditional Multilevel-Monte-Carlo where the average is based on a massive number of Euler schemes (with a much shorter horizon) whose simulation can be completely parallelized.}\\

\noindent The table suggests that the procedure seems to be slightly ``oversized''  for the required precision and that the last levels bring corrections which are of order $10^{-2}$. This feeling is confirmed by a computation of  the empirical RMSE (Root Mean-Squared Error) with $N=50$ simulations of the multilevel procedure. We obtain:
$$ RMSE(I_f, d=10, \varepsilon=0.1)\approx0.026.$$
In other words, the method calibrated to obtain a precision $\varepsilon=0.1$ produces a precision $0.026$ in this particular example.
In Table \ref{table2}, we now provide several tests of the robustness of the algorithm by computing the empirical RMSE (with $N=50$ simulations) for different values of $d$ and $\varepsilon$ and two different starting points: $x_0=0$ (which is the theoretical best choice) and $x_0=(1,\ldots,1)$ satisfying $|\nabla U(x_0)|^2=|x_0|^2=d$ (so that the condition of Remark \ref{rem:bettercomplexf} is satisfied).
\begin{table}[htbp]
\begin{center}
\begin{tabular}{| c |c| c | r | c| c|}
     \hline
     d&$\varepsilon$&R &Complexity & $I_f$ & RMSE($x_0=0$/$x_0=(1,1,\ldots,1)$) \\ \hline
     10 &0.1 & 5 & $0.80*10^3$ & 3.084 &0.026/0.026\\ \hline
     10 &0.01 & 8 & $8.79*10^5$ & 3.084 &0.029/0.030\\ \hline
      100 & 0.1 & 7 & $8.60*10^4$ & 9.975&0.014/0.013\\ \hline
      100& 0.01&10&$9.02*10^6$&9.975&0.001/0.001\\ \hline
      1000&0.1&8& $8.79*10^5$ &31.615& 0.016/0.016  \\

           \hline
 \end{tabular}
 \end{center}
\caption{\label{table2} Computation of an estimation of the RMSE for different values of $d$ and $\varepsilon$. }
 \end{table}
Once again, we can remark that (at least on this example), the numerical results outperform the required precisions. Furthermore, the performances are very robust to $d$ and $\varepsilon$ (which is coherent with Corollary \ref{cor:gibbs2}). Furthemore, it is worth noting that the performances with $x_0=0_{\ER^d}$ or $x_0=(1,1,\ldots,1)$ are almost equal. 



\subsubsection{A logistic-type perturbation}\label{subsec:logistic}
In this second example, we consider the potential $U:\ER^d\mapsto\ER$ defined by
$$ U(\beta)=U_1(\beta)+ \frac{\lambda |\beta|^2}{2},\quad \beta\in\ER^d,\quad\textnormal{with}\quad  U_1(\beta)=\log(1+e^{\mathbf{x}^T\beta})$$
and $\mathbf{x}\in\ER^d$.
In this second model, we added to the quadratic function $\beta\mapsto\frac{\lambda |\beta|^2}{2}$ the function $U_1$ whose gradient is nothing but a logistic function. Such a potential is in the spirit of the ones which appear in the posterior distribution of Bayesian logistic regression ({see \cref{subsec:MCMCbayesian} below for details}) with Gaussian prior (in this perspective $\mathbf{x}$ may be viewed as a vector of covariates). We thus chose to keep the usual Bayesian notation $\beta$ for the variable but we will not investigate the real statistical model. \\

\noindent In view of our paper, this example is an interesting case since the theoretical results still apply but with some different $\alpha_U$ and $L_U$ (and a non quadratic potential). Let us compute $\nabla U_1$ and $D^2 U_1$:
$$\nabla U_1(\beta)= \mathbf{x} (1+e^{-\mathbf{x}^T\beta})^{-1}\quad \textnormal{and}\quad  D^2 U(\beta)= \frac{\mathbf{x} \mathbf{x}^T}{  (1+e^{-\mathbf{x}^T\beta}) (1+e^{\mathbf{x}^T\beta})}.$$
Then, for any $\beta\in\ER^d$,
$$\langle D^2 U_1(\beta)\beta,\beta\rangle= \frac{|\mathbf{x}^T\beta|^2}{  (1+e^{-\mathbf{x}^T\beta}) (1+e^{\mathbf{x}^T\beta})},$$
so that $U_1$ is a convex function but with $\inf_{\beta\in\ER^d}\underline{\lambda}_{U_1}(\beta)=0.$ On the other hand,  we deduce from the previous equality  and from Cauchy-Schwarz inequality that 
$$\forall \beta\in\ER^d,\quad \langle D^2 U_1(\beta)\beta,\beta\rangle\le \frac{|\mathbf{x}|^2}{  (1+e^{-\mathbf{x}^T\beta}) (1+e^{\mathbf{x}^T\beta})} |\beta|^2\le \frac{|\mathbf{x}|^2}{5} |\beta|^2,$$
where in the last inequality, we used that $(1+e^u)(1+e^{-u})=3+2{\rm cosh} u\ge 5.$ From what precedes, we deduce that we can set
$$\alpha_U=\lambda\wedge 1\quad \textnormal{and} \quad L_U=\lambda+\frac{|\mathbf{x}|^2}{5}.$$
\begin{rem} Note that even though $U_1$ is a convex function, we say that $U_1$ is a perturbation of the quadratic potential since it does not modify the contraction parameter but it increases the value of $L_U$.
\end{rem}
Finally, let us consider the last condition on $\Delta(\nabla U)$. We have
$$|\partial_{i,j,k}^3 U(\beta)|=|\mathbf{x}_i\mathbf{x}_j\mathbf{x}_k\frac{-2{\rm sinh}( \mathbf{x}^T\beta)}{(3+2 {\rm cosh} ( \mathbf{x}^T\beta))^2}|\le \frac{1}{25}|\mathbf{x}_i\mathbf{x}_j\mathbf{x}_k|$$
so that 
$$\sum_{1\le i,j\le d}|\partial_{i,j,j}^3 U(\beta)|^2\le \frac{1}{25} |x|^2(\sum_{j=1}^d x_j^4)\le \frac{1}{25} |x|^4\max\{x_j^2,j=1,\dots d\}\le \frac{ |x|^6}{25} \le 5 L_U^3\le 5\alpha_U^{-1} L_U^4 d.$$
The last bound being very rough, this means that the behavior of $D^3 U$ will have few consequences on the performances of the algorithm.
Now, we choose to throw randomly $\mathbf{x}$ and normalize it in order that $\frac{|\mathbf{x}|^2}{5}=a$ for a given $a$.
We set $\mathbf{x}=\frac{\sqrt{5a} Z}{\|Z\|}$ where $Z\sim {\cal N}(0,I_d).$
We apply Corollary \ref{cor:gibbs2} with 
$$\sigma_0^2=\frac{\lambda}{(\lambda+\frac{|\mathbf{x}|^2}{5})^2},$$
and an initial condition obtained after a standard gradient descent with constant step. At the end of the procedure, we check that  $|\nabla U(\beta_0)|^2\le \alpha_U  {d}$ in order to satisfy the assumptions of  Corollary \ref{cor:gibbs2} (note that the convergence of the gradient descent is very fast in this strongly convex setting).\\

\noindent  Here, we choose to consider the function $f={\rm Id}$ (``in the spirit'' of the \textit{posterior means} in Bayesian statistics).
We first compute a sharp estimation of the vector $(\int \beta_i \pi(d\theta))_{i=1}^d$ that we denote by $\bar{\beta}_{true}$\footnote{This estimation has been obtained with $\varepsilon=0.01$.} and then compute a Monte-Carlo approximation (with $N=20$ simulations) of the (normalized) expectation of the  $L^2$-distance between the estimation and the target :
$d^{-\frac{1}{2}}\|\bar{\beta}_{true}-\bar{\beta}_{estim}\|_{2}=d^{-\frac{1}{2}}\ES[ |\bar{\beta}_{true}-\bar{\beta}_{estim}|^2]^{\frac{1}{2}}$ where $\bar{\beta}_{estim}$ denotes the approximation of $\bar{\beta}_{true}$ produced by the multilevel procedure and where the reader has to keep in mind that $|\,.\,|$ denotes the Euclidean norm on $\ER^d$.  We propose a simulation with the parameters:
$$ d=10,100, \quad a=2,\quad \lambda=1/4\quad \alpha_U=1/4,\quad\textnormal{and}\quad L_U=9/4.$$
Even with these not so pathologic values, we can remark that this unfortunately strongly increases the cost of computation with respect to the Ornstein-Uhlenbeck case since it multiplies it by $L_U^2/\alpha_U^3=324$. Nevertheless, the procedure  still works since we obtain for $\varepsilon=0.1,$ 
$$d^{-\frac{1}{2}}\|\bar{\beta}_{true}-\bar{\beta}_{estim}\|_{2}\approx \begin{cases} 0.042&\textnormal{ with $d=10$ }\\
 0.023 & \textnormal{ with $d=100$}.
 \end{cases}$$
\subsubsection{Towards some strategies to reduce the impact of \texorpdfstring{$\alpha_U$}{alphaU} and \texorpdfstring{$L_U$}{LU}.}\label{subsection:opening} As aforementioned, bad values of $\alpha_U$ and $L_U$ may seriously affect the complexity of the procedure (being proportional to  $L_U^2/\alpha_U^3$). This problem is not specific to the multilevel approach but should be certainly tackled in order to produce less time-consuming algorithms in this case.\\

\noindent  In our setting, we remarked  in the numerical computations that oppositely to the nicely calibrated Ornstein-Uhlenbeck process, the contributions provided by the correcting levels are two small with respect to the required precision. For instance, in the above example, for $\varepsilon=0.1$, the correction related to the first level is already of order $10^{-2}$  (whereas in Table \ref{table1}, the first levels bring a correction of order $10^{-1}$). This suggests that the step is too small. More precisely, even though $\pas_0=1/2$ in Corollary \ref{cor:gibbs2}, the factor $\sigma_0^2=\alpha_U/(2L_U^2)$ induces a very small evolution of the dynamics (but is theoretically optimal in terms of $\alpha_U$ and $L_U$). \\

\noindent From a theoretical point of view, we are in fact limited by the constraint $\pas\le \alpha/(2L^2)$ which appears in Theorems \ref{maincoro1} and \ref{maincoro2}. Going back to the proofs, one can remark that this constraint is of first importance in several arguments and firstly in Lemma \ref{lem:boundEuler}$(i)$ for the $L^2$-stability of the Euler scheme. Actually, for a too large step, the Euler scheme explodes since the first order error produced by the discretization of the drift term becomes stronger than the contraction coming from Assumption $\Cs$. Note that this problem also appears in models where the potential is superquadratic\footnote{{A function $V:\ER^d\mapsto\ER$ is said to be superquadratic if $\lim_{|x|\rightarrow+\infty} \frac{V(x)}{|x|^2}=+\infty$.}}. Some solutions are proposed in the literature by introducing alternative schemes such as implicit discretizations in \cite{mattingly_stuart} or explicit Euler schemes with randomized (decreasing adaptive) step sequence as in \cite{Le07a}. Such alternative schemes (and other ones) should be probably investigated in view of improvements of the procedure in the general case. \\

\noindent In our specific case,  $U(x)=U_1(x)+U_2(x)$ where $U_1$ has a bounded gradient and $U_2(x)=\lambda |x|^2/2$ and it is in fact possible to alleviate the constraint on the step which guarantees a $L^2$-stability, without modifying the scheme.
Actually, since in this case,
$$|b(x)-b(y)|^2\le 2\lambda^2 |x-y|^2+ 8\|\nabla U_1\|_\infty^2,$$
a careful reading of the associated proof shows that the scheme is always long-time stable for any $\pas\le 1/(4\tcr{\lambda})$ with
\begin{equation*}
 \ES[|\Xge_{t}^{\pas,x}-x^\star|^2]\lesssim_{uc} |x-x^\star|^2 e^{-\frac{\alpha}{2}t}+\frac{|b(x^\star)|^2}{\laminf^2}+ \frac{\|U_1\|_\infty^2\pas+\sigma^2 d}{\laminf}.
 \end{equation*}
In the case $b(x^\star)=0$, this implies that the bound of  Lemma \ref{lem:boundEuler}$(i)$ remains of the same order as soon as   $\pas\le \min((4\alpha)^{-1}, \frac{\sigma^2 d}{\|U_1\|_\infty^2})$.
Note that such improvements of the domain of stability of the Euler scheme   may be possible (with other constraints) in the case where $U_2$ is $L_2$-Lispchitz with $L_2\ll L_1$ and $|\nabla U_1(x)|^2=o(U_2(x))$ when $x\rightarrow+\infty$ (case which usually appears in applications). 
However, Lemma \ref{lem:boundEuler}$(i)$ is not the only part of the proof where the constraint $\pas\le \alpha/L^2$ appears. In particular, it plays an important role for the control of the distance between the paths of the Euler scheme (which in turns leads to the control of the confluence properties which allow to control the variance). A a consequence, a potential improvement of Corollary \ref{cor:gibbs2} in this particular setting would require further investigations. \\

\noindent In order to give some little more substance to these perspectives, let us finish with a numerical computation. In the spirit of Corollary \ref{cor:gibbs2}, we keep $\sigma_0^2=\alpha_U/L_U^2$ but  replace $\pas_0=1/2$ by $\pas_0=\min((4\alpha_{\sigma_0})^{-1},\frac{\sigma_0^2 d}{\|U_1\|_\infty^2})$ in order to saturate the condition $\pas\le \min((4\alpha)^{-1}, \frac{\sigma^2 d}{\|U_1\|_\infty^2})$. With $d=100$, this leads in our example to  $\pas_0=(4\alpha_{\sigma_0})^{-1}=L_U^2/(4\alpha_U^2)$ and by \eqref{eq:complex4}, to a complexity of the order 
$\sigma_0^2 L_{\sigma_0}^4\pas_0^{-1}\alpha_{\sigma_0}^{-4} d\varepsilon^{-2}=\alpha_U^{-1}d\varepsilon^{-2}$ (instead of $L_U^2\alpha_U^{-3}d\varepsilon^{-2}$). With $d=100$ and with the notations and values of the previous section, this yields
$$d^{-\frac{1}{2}}\|\bar{\beta}_{true}-\bar{\beta}_{estim}\|_{2}\approx 0.017$$
for $\varepsilon=0.1$ (on $N=20$ simulations). Thus, it seems to preserve the efficiency of the theoretically checked method but with a number of iterations which has been divided by $L_U^2\alpha_U^{-2}$ ($=81$ in this particular case).
Going deeper in the numerical and theoretical perspectives on this topic is the purpose of a future paper.

\subsubsection{Comparison with some other MCMC methods for Bayesian learning}\label{subsec:MCMCbayesian}
{Let us continue this numerical section with some simulations that compare our proposed estimator to benchmark methods such as the Unadjusted Langevin Algorithm (ULA) and the Metropolis-Adjusted Langevin Algorithm (MALA). The aim is to compute a Bayesian estimator with the help of $n$ observations $Y=(y_1, \dots,y_n)\in \ER^n$ and many covariates $X=(x_1,\dots,x_n)\in \ER^{d\times n}$. The purpose of the Bayesian paradigm is to find a law modeling the parameter $\beta$. This research is based on the choice of a \textit{prior} law $\pi_0$ that characterizes what their value might be. By Bayes' rule, the \textit{posterior} density $\pi_n$ is given by:}
{
\begin{equation}\label{posteriorlaw}
\pi_n(\beta) \propto \pi_0(\beta) \exp \left( - \sum_{i=1}^n \ell_\beta(x_i,y_i) \right).
\end{equation}
}
{
where in logistic regression, the function $\ell_\beta$ is defined by:
\begin{equation*}
\ell_\beta(x,y)= y\log\left(\mathfrak{s}(\beta^T x)\right)+ (1- y)\log\left(1-\mathfrak{s}( \beta^T x )\right)
\end{equation*}
with $\mathfrak{s}$ denoting the sigmo\?id function defined by $\mathfrak{s}(t)=(1+e^{-t})^{-1}$, $t\in \ER$. 
We can consider a large variety of prior laws depending, for example, on our knowledge of the problem. Recently, great interest has been given to the so-called  Exponentially Weighted Aggregate (EWA) where $\pi_0(\beta)\propto e^{-\lambda \mathrm{Pen}(\beta)}$ (see \cite{art:dalalyan2018exponentially}),
with $\mathrm{Pen} : \ER^d \to \ER$ being a regularization function and $\lambda>0$. Here, we consider the Bayesian Ridge Logistic Regression by setting 
$\mathrm{Pen}(\beta)=|\beta|^2/2$ (we thus penalize by the square of the Euclidean norm) and aim to compute the posterior mean:
$$\hat{\beta}^{{\rm EWA}}=\int \beta \pi_n(\beta) d\beta=\int \beta \frac{e^{-U_n(\beta)}}{Z} \lambda_{d}(d\beta),$$
where $U(\beta):=\sum_{i=1}^n \ell_\beta(x_i,y_i) + \lambda|\beta|^2$ and $Z=\int e^{-U_n(\beta)}\lambda_{d}(d\beta)$.}
%
%
%

\noindent {We test our algorithm on a  heart disease public database\footnote{These data come from four different geographic placesCleveland, Hungary, Switzerland, and Long Beach V. \href{https://www.kaggle.com/datasets/johnsmith88/heart-disease-dataset}{https://www.kaggle.com/datasets/heart-disease-dataset}}, also used in \cite{durmusmoulineshigh} which contains $13$ covariates\footnote{Note that the ordinal covariates are replaced by dummy variables.} supposed to be correlated to heart diseases. Consequently, the target will be the presence or not of heart disease in the patient. To predict the target, more than $1000$ patients are observed.} \\

\noindent {As mentioned before, we compare the performances of our Multilevel-Langevin pathwise average (MLPA) with ULA (see \emph{e.g.} \cite{durmusmoulineshigh} and \cite{dalalyan}) and MALA (see \emph{e.g.} \cite{roberts_tweedie, bou2013nonasymptotic,art:WainwrightMetrop,chewi2021optimal,preart:DurmusMoulineMALA}). We recall that ULA is based on a classical Monte-Carlo average of Euler schemes whereas MALA is a Metropolis-Hasting-type algorithm  (see \cite{art:metropolis1953equation}) where proposals are based on the Euler scheme of a Langevin dynamics (see \cite{roberts_tweedie} for details). }
{We compute ULA and MALA with the following parameters:
\begin{equation}
    \Xge_0=0_{\ER^d}, \quad \pas=\varepsilon \sqrt{\frac{\alpha_U^{3}}{2 \blip_U^{4}d}}, \quad T=\frac{1}{2 \alpha_U}\log\left(\frac{d \varepsilon^{-2}}{2 \alpha_U}\right), \quad N=\frac{2d}{\varepsilon^{2}\alpha_U} \quad \textnormal{and} \quad \varepsilon=0.1,
\end{equation}
where $\pas$ denotes the discretization step, $T$ is the final time of each path, $N$ is the number of Monte Carlo sampling, and $\alpha_U$ and $L_U$ denote respectively the smallest and the greatest eigenvalue of the Hessian of U\footnote{They can be computed exactly as in Subsection \ref{subsec:logistic}.}  Relying on the results of \cite{durmusmoulineshigh},  these parameters are optimal choices for ULA to provide a RMSE of order $\varepsilon$. . Finally, we compare these estimations with MLPA applied with the following parameters (which lead to a RMSE of order $\varepsilon$ by \cref{cor:gibbs2}):}
{
\begin{equation*}
    \begin{split}
          &\gamma_0=\frac{1}{2}, \quad \lev_\varepsilon=\left\lceil \log_2 \left(\varepsilon^{-1}\sqrt{\frac{d}{\alpha_U}} \right)\right\rceil, \quad \ind=0,\ldots,\lev_\varepsilon, \quad \pas_\ind=\pas_0 2^{-\ind}
         \\& \Xge_0=0 ,\quad   T_\ind=d\varepsilon^{-2}\frac{\blip_U^2}{\alpha_U^3} 2^{-\frac{3}{2}\ind},\; \quad \textnormal{and}\quad \tau=\frac{1}{\alpha_U}\log( \varepsilon^{-1})+\frac{1}{2\alpha_U} \log \left(\sqrt{\frac{d}{8 \alpha_U}}\right).
    \end{split}
\end{equation*}
Our comparisons of ULA, MALA and MLPA are resumed in Table \ref{chap3table1} below where we provide the number of iterations of the Euler scheme and the empirical RMSE obtained after $50$ simulations of each method with $\varepsilon=0.1$.}
{\begin{table}[htbp]
\begin{center}
    \begin{tabular}{| l | c | c | c| }
        \hline
        Algorithm &MALA & ULA & MLPA \\ \hline
        Number of iterations & 320089 & 320089 & 62415 \\ \hline
        empirical RMSE & 0.0995 & 0.1458 & 0.0966\\
        \hline
    \end{tabular}
    \caption{\label{chap3table1} Comparison of ULA, MALA and MLPA, $\varepsilon=0.1$.}
    \end{center}
\end{table}}
{ Table \ref{chap3table1} illustrates the result shown in this paper. Indeed, we see that we achieve the same precision for the three methods, but the computational cost of the Multilevel method is very cheap compared to the two others.}

\subsubsection{Robustness in the non-convex setting}\label{sub:non-convex-setting}
{
In view of applications, a natural question occurs. Is such an algorithm able to remain efficient in a non-convex setting ? More precisely, does such a multilevel procedure have the ability of remaining more efficient than a standard one in the non-convex setting ? As explained in \cref{rem:comparison}, among the assumptions of \cref{maintheo}, the one which is the most difficult to check in a non-convex setting is the $L^2$-confluence hypothesis $\HDEUX$ ($L^2$-confluence)This is also the assumption which is the only one which is really specific to the multilevel procedure since its role is to control the variance of the correcting layers. The $L^2$-confluence is a very difficult problem and  it is clear that $\HDEUX$ is not true in general in the non-convex setting (see for instance the counter-example given in \cite[Proposition 3.1]{panloup_pages_lemaire}). Nevertheless, the example below shows that in some cases, the procedure may remain efficient:}\\

\noindent {Set $U(x)=\frac{1}{2}|x|^2-\log(1+|x|^2)$, $x\in\ER^d$. One easily checks that $\nabla U(x)= x\left(\frac{|x|^2-1}{|x|^2+1}\right)$ and that $U$ has a local maximum in $0$ and that each point $x$ satisfying $|x|=1$ is a local minima. Let us denote by $\nu\propto e^{-U}$ the related Gibbs distribution which is the invariant distribution of 
\begin{equation*}
\textnormal{d}X_t = -X_t \frac{|X_t|^2-1}{|X_t|^2+1} \textnormal{d}t + \sqrt{2}\textnormal{d}\B_t.
\end{equation*}
As in \cref{sub:ououou}, we choose  here to compute $\nu(f)$ with $f(x)=|x|$. In fact, $\nu(f)$ can be explicitly computed. We have
\begin{equation*}
\int_{\ER^d} f(x)\nu(dx)=\frac{d+2}{d+1} I_f,
\end{equation*}
where $I_f$ is given by \eqref{eq:IFpi}. Since in this setting, the contraction parameter $\alpha_U$ does not exist, we have to fix arbitrarily some parameters. We choose to do as if we had $\alpha_U=L_U=1$. This means that we fix the  parameters  as  in \cref{sub:ououou}. Table \ref{table4} below contains the empirical root-mean squared errors computed with $M=50$ computations related to a computation where $\varepsilon=0.1$. We thus remark that the method still works here (since the empirical error is lower than $\varepsilon$).
\begin{table}[htbp]
\begin{center}
\begin{tabular}{|c||c|c|}
\hline
&$d=100$ &  $d=1000$  \\
\hline
\hline
$\varepsilon=10^{-1}$ & 0.024   & 0.017\\
\hline
\end{tabular}
\end{center}
\caption{\label{table4} Empirical RMSE related to the computation of $\nu(f)$}
\end{table}
}

\section{A quantitative control of the error}\label{sec:quantitativecontrol}
{The proof of  \cref{maintheo}  is based on a classical bias-variance decomposition of the error with respect to the target. The originality of the proof lies in the sharp control of each term of the decomposition, according to the set of assumptions of \cref{maintheo}. Such controls are resumed in Proposition \ref{maintheointermed}} below where we provide, for some given step and time sequences,  an almost\footnote{By ``almost'', we mean that we omit the universal constants, $i.e.$ which do not depend on the parameters of the assumptions and of the diffusion. This will not perturb the sequel of the paper in which our main objective is to exhibit the dependency in the dimension.} quantitative control of the error in terms of the parameters involved in Assumptions {$\HUN$ to $\HQU$}. The proof of Theorem \ref{maintheo} will be then achieved in Section \ref{sec:maintheo}.


\begin{prop} \label{maintheointermed}
Let $f$ be a (non constant) Lipschitz function. Assume that {$\mathbf{(H_i)}$, $i=1,\ldots,4$,} hold for some given $\bea\in[1,2]$, $\eta_0\in[0,1/2]$ and $x_0\in\ER^d$. Assume that $\pas_0\in(0,\eta_0]$ with $\alpha\pas_0\le 1$, and that 
for every $\ind\in\{0,\ldots,\lev\}$, 
 $$\pas_\ind=\pas_0 2^{-\ind}\quad\textnormal{and}\quad  T_\ind =   T_0 2^{-\frac{1+\bea}{2}\ind},$$
 and that $\tau$ is a non-negative number satisfying $\tau\le T_\lev/2$. Then,

\noindent 
\begin{align*}
\frac{1}{[f]_1^2} \| \mathcal{Y}(\lev,\l(\pas_\ind\r)_\ind,\tau,\l(T_\ind\r)_\ind,f)- \pi(f) \|_2^2 & \le   \frac{\cuniv}{\alpha T_0}\left(\4^2+
\max(\2^2,\2\4)\pas_0^\bea\log\left(\pas_0^{-1}\right)
\left((\bea-1)^{-2}\wedge R^2 \right)\right)
\\ & + \l(  \3  \pas_\lev^\delta + \frac{\cuniv \1(x_0)}{\alpha T_0} e^{-\alpha\tau}2^{\frac{1+\bea}{2}\lev } \r)^2,
\end{align*}
where $\cuniv$ is a universal constant. The related complexity cost $\mathcal{C}(\mathcal{Y})$\footnote{By complexity cost, we recall that we mean the number of iterations of the Euler scheme which is needed to compute  $\mathcal{Y}(\lev,\l(\pas_\ind\r)_\ind,\tau,\l(T_\ind\r)_\ind,.)$.} satisfies:
\begin{equation}\label{eq:cout1}
\mathcal{C}(\mathcal{Y}) \le \begin{cases}  \frac{T_0}{\pas_0} \left((1+\frac{3}{2}(2^{\frac{\bea-1}{2}}-1)^{-1}\right)&\textnormal{if $\bea>1$}\\
\frac{T_0\left(1+\frac{3}{2}\lev\right)}{\pas_0}&\textnormal{if $\bea=1$}
\end{cases}
\end{equation} 
\end{prop}
The proof of the above result is the objective of the sequel of this section. By the bias/variance decomposition,
\begin{equation*}
\| \mathcal{Y}(\lev,\l(\pas_\ind\r)_\ind,\tau,\l(T_\ind\r)_\ind,f)- \pi(f) \|_2^2 = \mathbb{E}[\mathcal{Y}(\lev,\l(\pas_\ind\r)_\ind,\tau,\l(T_\ind\r)_\ind,f)-\pi(f)]^2+\mathrm{Var}(\mathcal{Y}(\lev,\l(\pas_\ind\r)_\ind,\tau,\l(T_\ind\r)_\ind,f)).
\end{equation*}
The sequel of the section is then divided into two parts successively studying   the bias and variance terms. The main respective results are Propositions \ref{biascontr} and \ref{varcontr}. Then, Proposition \ref{maintheointermed}  follows from a combination of these two propositions and from the following remark about the complexity cost: for some given parameters $R$ and $(T_r)_{r=0}^R$ and $(\pas_r)_{r=0}^R$, the complexity cost related to $\mathcal{Y}(\lev,\l(\pas_\ind\r)_\ind,\tau,\l(T_\ind\r)_\ind,.)$ satisfies:
$${\cal C}(\mathcal{Y})\le \frac{T_0}{\pas_0}+\sum_{\ind=1}^\lev \left(\frac{T_\ind}{\pas_\ind}+\frac{T_{\ind}}{\pas_{\ind-1}}\right)
= \frac{T_0}{\pas_0}\left(1+\frac{3}{2} \sum_{\ind=1}^\lev 2^{\frac{1-\bea}{2}\ind}\right),$$
This easily leads to \eqref{eq:cout1}.
%




\subsection{Step 1: Bias of the procedure}
{In the sequel, $\mathcal{Y}(\lev,\l(\pas_\ind\r)_\ind,\tau,\l(T_\ind\r)_\ind,f)$ is usually written $\mathcal{Y}$ for the sake of simplicity.}  We start with a telescopic-type decomposition:
\begin{equation}\label{telescop}
\begin{split}
\mathcal{Y}(f)-\pi(f) &= \frac{1}{T_0-\tau}\int_\tau^{T_0} f(\bar{X}_{\un{s}_{\gamma_0}}^{\gamma_0,x_0}) - \pi^{\pas_0} (f) ds
\\ & + \sum_{\ind=1}^\lev \l(\frac{1}{T_r-\tau}\int_\tau^{T_r}f(\bar{X}_{\un{s}_{\gamma_{r-1}}}^{\gamma_r,x_0})- \pi^{\pas_\ind} (f) ds-\frac{1}{T_r-\tau}\int_\tau^{T_r}f(\bar{X}_{\un{s}_{\gamma_{r-1}}}^{\gamma_{r-1},x_0})-\pi^{\pas_{\ind-1}} (f)  ds \r)
\\ & + \pi^{\pas_R}(f)-\pi(f). 
\end{split}
\end{equation}
\begin{rem}\label{rem:multilevelstand}
In a standard Multilevel Monte-Carlo procedure (in finite horizon), the expectation of the above sum would be equal to the last term only, $i.e.$ the bias would  be exactly $\pi^{\pas_R}(f)-\pi(f)$. In this long-time setting, the bias also contains long-time components which correspond to the (expectation of the) first and the second right-hand members of the above equality. Nevertheless, these long-time error terms will be negligible under the exponential contraction assumption $\HUN$. 
\end{rem}
\noindent Let us now study the bias generated by the first and second terms of the right-hand side of \eqref{telescop}.

\begin{lem}\label{couchcontr} {Assume $\HUN$. Let $\gamma\in(0,\eta_0]$. Let $\eta=\gamma$ or $\eta=2\gamma$. Let  $\tau$ and $T$ be positive numbers, such that $\tau\le \frac{T}{2}$.  Then, for all $x\in\ER^d$,} 
\begin{equation*}
\l| \frac{1}{T-\tau}\int_\tau^{T} \mathbb{E}_{x}[f(\bar{X}_{\un{s}_{\eta}}^{\gamma,x_0})] - \pi^{\pas} (f) ds
\r| \le \frac{2 e^{\alpha\eta} \1(x)[f]_1e^{-{\alpha\tau}}}{\alpha T}   ,
\end{equation*}
where $x\mapsto \1(x)$  is given by $\HUN$.
\end{lem}


\begin{proof}
Let $\eta=\gamma$ or $\eta=2\gamma$. We have
$$
\l| \frac{1}{T-\tau}\int_\tau^{T} \mathbb{E}_{x}[f(\bar{X}_{\un{s}_{\eta}}^{\gamma,x_0})] - \pi^{\pas} (f) ds
\r|\le \frac{1}{T-\tau}\int_\tau^{T}\l|\mathbb{E}_{x}[f(\bar{X}_{\un{s}_{\eta}}^{\gamma,x_0})] - \pi^{\pas} (f)\r| ds.
$$
By Assumption $\HUN$, it follows that
\begin{align*}
\l| \frac{1}{T-\tau}\int_\tau^{T} \mathbb{E}_{x}[f(\bar{X}_{\un{s}_{\eta}}^{\gamma,x_0})] - \pi^{\pas} (f) ds
\r|&\le  \frac{c_1(x)[f]_1}{T-\tau}\int_{\tau}^Te^{-\alpha {\un{s}_{\eta}}} ds\le \frac{c_1(x)[f]_1}{T-\tau}\sum_{k=\lfloor\frac{\tau}{\eta}\rfloor}^{\lceil \frac{T}{\eta}\rceil}e^{-\alpha \eta k}. 
\end{align*}
Then, a standard computation leads to
\begin{equation*}
\l| \frac{1}{T-\tau}\int_\tau^{T} \mathbb{E}_{x}[f(\bar{X}_{\un{s}_{\eta}}^{\gamma,x_0})] - \pi^{\pas} (f) ds\r|\le \frac{c_1(x)[f]_1}{T-\tau} \frac{e^{-\alpha \eta \lfloor\frac{\tau}{\eta}\rfloor}}{\alpha}. 
\end{equation*}
The result follows by using that $\tau\le T/2$ and that  $\eta \lfloor\frac{\tau}{\eta}\rfloor\ge \tau-\eta$.
\end{proof}

We are now ready to state a proposition about the control of the bias of the procedure.
\begin{prop}\label{biascontr}
Assume that $\HUN$, $\HTROIS$ and $\HQU$ hold for some given $\bea\in[1,2]$, $\eta_0\in[0,1]$ and $x_0\in\ER^d$. Let $\pas_0\in(0,\eta_0]$ and $R\in\mathbb{N}^*$. Let $\tau$ and $T_0$ be some positive numbers such that  $2\tau\le T_R$ where for each $r\in\{0,\ldots,R\}$,  
$$T_r= T_0 2^{-\frac{1+\bea}{2} r}.$$ 
 Then, for every Lipschitz continuous function $f:\ER^d\rightarrow\ER$,
\begin{equation}
\l|\mathbb{E}_{x_0}[\mathcal{Y}(\lev,\l(\pas_\ind\r)_\ind,\tau,\l(T_\ind\r)_\ind,f)]-\pi(f)]\r| \le \3 [f]_1  \pas_\lev^\delta + \frac{8e^{2\alpha\pas_0} \1(x_0)[f]_1}{\alpha T_0} e^{-{\alpha \tau}}2^{\frac{1+\bea}{2}\lev },
\end{equation}
where $\1(x)$ and $\3$ are given by Assumptions $\HUN$ and $\HTROIS$.
\end{prop}
\begin{rem} In the continuity of Remark \ref{rem:multilevelstand}, one retrieves that the right-hand side of the inequality is made of two terms, the first one being derived from $\pi^{\pas_R}(f)-\pi(f)$ and the second one coming from the long-time errors. Note that owing to the exponential convergence to the invariant distribution, an $\exp(-\alpha\tau)$-term appears, which strongly depends on the \textit{warm-start} $\tau$, $i.e.$ on the starting time of the pathwise average.\\
\noindent In order to obtain a complexity proportional to $\varepsilon^{-2}$, it will be necessary to take $\tau$ large enough in such a way that the long-time bias remains negligible. 
\end{rem}


\begin{proof}
Taking the expectation in \eqref{telescop}, we obtain:
\begin{align*}
\l| \mathbb{E}[\mathcal{Y}(f)-\pi(f)] \r| &\le \l|\frac{1}{T_0-\tau}\int_\tau^{T_0}\ES_{x_0}\l[ f(\bar{X}_{\un{s}_{\gamma_0}}^{\gamma_0,x_0})\r]- \pi^{\pas_0} (f) ds
  \r| 
\\ & + \sum_{\ind=1}^\lev  \frac{1}{T_r-\tau}\l|\int_\tau^{T_r}\ES\l[f(\bar{X}_{\un{s}_{\gamma_{r-1}}}^{\gamma_r,x_0})\r]- \pi^{\pas_\ind} (f) ds\r|\\
&+ \sum_{\ind=1}^\lev\frac{1}{T_r-\tau}\l|\int_\tau^{T_r}\ES\l[ f(\bar{X}_{\un{s}_{\gamma_{r-1}}}^{\gamma_{r-1},x_0})\r]-\pi^{\pas_{\ind-1}} (f)  ds\r| \\ 
&  + \l| \pi^{\pas_\lev}(f)-\pi(f) \r|{.}
\end{align*}
The last term is controlled with {the help of} Assumption $\HTROIS$ which ensures that
\begin{align*}
 \l| \pi^{\pas_\lev}(f)-\pi(f) \r| &\le \3 [f]_1 \pas_\lev^\delta.
\end{align*}
\noindent For the three first terms, we apply Lemma \ref{couchcontr}  with $(\pas,\eta,\tau,T)=(\pas_0,\pas_0,\tau,T_0)$, 
$(\pas,\eta,\tau,T)=(\pas_r,\pas_{r-1},\tau,T_r)$ and $(\pas,\eta,\tau,T)=(\pas_{r-1},\pas_{r-1},\tau,T_r)$, respectively in the first, second and third terms.
In each case, one can check that $\tau$ and $T$ satisfy the assumptions of Lemma \ref{couchcontr}. This  leads to:
\begin{align} \label{biais1}
\l| \mathbb{E}_{x_0}[\mathcal{Y}(f)]-\pi(f) \r| &\le \3 [f]_1  \pas_\lev^\delta +  \frac{4e^{2\alpha\pas_0} \1(x_0)[f]_1}{\alpha}  e^{-\alpha \tau} \sum_{\ind=0}^\lev \frac{1}{T_r}\\
\nonumber &\le  \3 [f]_1  \pas_\lev^\delta +  \frac{8e^{2\alpha\pas_0} \1(x_0)[f]_1}{\alpha T_0}  e^{-\alpha \tau} 2^{\frac{1+\bea}{2} R},
\end{align}
where in the second line, we used that
$\sum_{r=0}^R \rho^r\le 2\rho^{R}$ for any $\rho \ge 2$.
%
The result follows.
\end{proof}

\subsubsection{Study of the variance} 

Let us now focus on the study the variance of our estimator. The basic idea of multilevel strategies is in general to introduce some additive layers which can correct the bias without adding too much variance. In the setting of discretization of processes, this idea mainly relies on the capability of controlling the distance between  discretization schemes with different step sizes ($\pas$ and $\pas/2$ in our construction). Thus, our assumption $\HDEUX$ will play a fundamental role in this part. However, in our setting where we consider empirical averages, the variance also depends on the mixing properties of the involved dynamical system. Hence, our ergodicity assumption $\HUN$ will also be of first importance.\\

\noindent First, owing to the definition \eqref{eq:estimat}, to the independency of the Brownian motions related to each level $\ind$ and to the fact that $\pas_r=\frac{\pas_{r-1}}{2}$, one can check that the variance admits the following decomposition:
\begin{equation}\label{eq:decompofvar}
\mathrm{Var}(\mathcal{Y}(\lev,\l(\pas_\ind\r)_\ind,\tau,\l(T_\ind\r)_\ind,f)) = \mathrm{Var}\l(\frac{1}{T_0-\tau}\int_{\tau}^{T_0} f(\bar{X}_{\un{s}_{\pas_0}}^{\pas_0,x_0}) ds\r) 
+ \sum_{\ind=1}^\lev \mathrm{Var}\l(\frac{1}{T_r-\tau}\int_\tau^{T_r} G_s^{\pas_{r-1}}ds\r),
\end{equation}
where for some given $\gamma>0$ and $s\ge0$,
$$ G_s^\gamma=f(\bar{X}_{\un{s}_{\pas}}^{\frac{\pas}{2},x_0}) -f(\bar{X}_{\un{s}_{\pas}}^{\pas,x_0}).$$
In the following lemma, we focus on the second term:

\begin{lem}\label{propvarlvl} Let $f$ be a Lipschitz function. Let $0\le \tau<T$.
 Assume that $\HUN$, $\HDEUX$ and $\HQU$ hold for some given $\bea\in[1,2]$, $\pas_0\in[0,\eta_0]$ and $x_0\in\ER^d$.
Let $\pas\in(0,\pas_0]$ with $\alpha\pas\le 1$. Then,
\begin{equation*} 
 \mathrm{Var}\l(\frac{1}{T-\tau}\int_\tau^{T} G_s^{\pas}ds\r) \le 
 \frac{\mathfrak{c}_{var}[f]_1^2}{{T}-{\tau}} \pas^\bea \log\l(\frac{1}{\pas}\r),
\end{equation*}
with $\mathfrak{c}_{var} =   16 e \alpha^{-1}\mathrm{max} (\2^2 ,  \2 \4 )$ ($\2$ and $\4$ being given by $\HDEUX$ and $\HQU$).
\end{lem}


\begin{proof}
%
A standard computation shows that
\begin{align}
\mathrm{Var}\l(\frac{1}{T-\tau}\int_\tau^{T} G_s^{\pas}ds\r) 
 = \frac{2}{(T-\tau)^2} \int_{{\tau}}^{{T}} \int_u^{{T}} {\rm Cov}\l(G_s^{\pas}, G_u^{\pas}\r) \mathrm{d}s\mathrm{d}u.\label{eq:expvar}
\end{align}
The idea of the sequel of the proof is to provide two types of bounds for ${\rm Cov}\l(G_s^{\pas}, G_u^{\pas}\r) $, depending on the size of $s-u$ (small or large).\\

\noindent First, by the Cauchy-Schwarz inequality,
\begin{equation*}
{\rm Cov}\l(G_s^{\pas}, G_u^{\pas}\r)  \le \sqrt{{\rm Var} (G_s^{\pas}){\rm Var} (G_u^{\pas})}.
\end{equation*}
Then, by Assumption $\HDEUX$,
and the fact that $f$ is Lipschitz continuous, we get:
\begin{align}\label{eq:controlgs}
{\rm Var} (G_s^{\pas}) \le \ES[(G_s^{\pas})^2]\le [f]_1^2 \| \bar{X}^{\pas,x_0}_{\underline{s}_\pas}- \bar{X}^{\pas/2,x_0}_{\underline{s}_{\pas}} \|_2^2\le [f]_1^2\2^2\pas^{{\bea}}.
 \end{align}
Thus, 
\begin{equation} \label{way1}
{\rm Cov}\l(G_s^{\pas}, G_u^{\pas}\r) \le  [f]_1^2 \2^2 \pas^{\bea}.
\end{equation}
Second, when $s-u$ is large, one can make use of the ergodicity of the process. More precisely, let us first remark that
for a given step size $\pas$, $G_u^{\pas}$ is ${\cal F}_{\un{u}_{\pas}}$-measurable. Thus, for any $s\ge \un{u}_{\pas}$,
\begin{align*}
 \mathbb{E}\l[G_s^{\pas} G_u^{\pas} \r] = \mathbb{E}[\mathbb{E}[G_s^\pas|\mathcal{F}_{\un{u}_{\pas}}]G_u^\pas].
\end{align*} 
Setting 
\begin{equation}\label{phietax}
\phi(\eta,t,x)=\ES[f(\bar{X}_t^{\eta,x})]-\pi^\eta(f),
\end{equation}
we deduce from the Markov property that
$$ \mathbb{E}[G_s^\pas|\mathcal{F}_{\un{u}_{\pas}}]=\pi^{\frac{\pas}{2}}(f)+\phi(\frac{\pas}{2}, \un{s}_\pas-\un{u}_\pas, \bar{X}_{\un{u}_{\pas}}^{\frac{\pas}{2},x_0})-\left(\pi^{{\pas}}(f)+\phi({\pas}, \un{s}_\pas-\un{u}_\pas, \bar{X}_{\un{u}_{\pas}}^{\pas,x_0})\right).$$
Thus, by Assumption $\HUN$,
$$ \mathbb{E}\l[G_s^{\pas} G_u^{\pas} \r] = \l(\pi^{\frac{\pas}{2}}(f)-\pi^{{\pas}}(f)\r) \ES[G_u^{\pas} ]+ R_1(s,u,\pas)$$
with 
\begin{align*}
 |R_1(s,u,\pas)|&\le [f]_1 e^{-\alpha (\un{s}_\pas-\un{u}_\pas)}\l|\ES[ (c_1(\bar{X}_{\un{u}_{\pas}}^{\frac{\pas}{2},x_0})+c_1(\bar{X}_{\un{u}_{\pas}}^{\pas,x_0}))  G_u^{\pas}]\r|
\\
&\le 2 [f]_1 e^{-\alpha (\un{s}_\pas-\un{u}_\pas)}\|G_u^{\pas}\|_2\max \l(\|c_1(\bar{X}_{\un{u}_{\pas}}^{\frac{\pas}{2},x_0})\|_2,\|c_1(\bar{X}_{\un{u}_{\pas}}^{\pas,x_0})\|_2\r),
 \end{align*}
 where in the second line, we used Cauchy-Schwarz inequality. Then, by Assumption $\HQU$ and the same argument as in \eqref{eq:controlgs}, we deduce that
 \begin{equation}\label{covarg1}
 \l|\mathbb{E}\l[G_s^{\pas} G_u^{\pas} \r]-\l(\pi^{\frac{\pas}{2}}(f)-\pi^{{\pas}}(f)\r) \ES[G_u^{\pas} ]\r|\le 2 [f]_1 c_4 e^{-\alpha (\un{s}_\pas-\un{u}_\pas)}\|G_u^{\pas}\|_2
 \le 2[f]_1^2 c_2 c_4 \pas^{\frac{\bea}{2}}e^{-\alpha (\un{s}_\pas-\un{u}_\pas)}.
 \end{equation}
 Let us now consider $\mathbb{E}\l[G_s^{\pas}]\ES[ G_u^{\pas} \r].$ With similar arguments as above,
$$ \mathbb{E}\l[G_s^{\pas}]\ES[ G_u^{\pas} \r]=\l(\pi^{\frac{\pas}{2}}(f)-\pi^{{\pas}}(f)\r) \ES[G_u^{\pas} ]+R_2(s,u,\pas)$$
with 
 \begin{align*}
 |R_2(s,u,\pas)|&\le 2 [f]_1 e^{-\alpha \un{s}_\pas} c_1(x_0)\|G_u^{\pas}\|_2\le 2 [f]_1^2 c_2 c_4\pas^{\frac{\bea}{2}}e^{-\alpha \un{s}_\pas},
 \end{align*}
 since $c_1(x_0)\le c_4$ under Assumption $\HQU$.Thus, combining with  \eqref{covarg1}, we get
$$ {\rm Cov}\l(G_s^{\pas}, G_u^{\pas}\r)\le 4[f]_1^2 c_2 c_4 \pas^{\frac{\bea}{2}}e^{-\alpha (\un{s}_\pas-\un{u}_\pas)}\le 4 e^{\alpha \pas}[f]_1^2 c_2 c_4 \pas^{\frac{\bea}{2}} e^{-\alpha (s-u)},$$
since $\un{s}_\pas-\un{u}_\pas\ge s-u+\pas$. Combining this inequality with \eqref{way1}, we obtain for every $0\le u\le s\le T$:
\begin{equation*}
 {\rm Cov}\l(G_s^{\pas}, G_u^{\pas}\r)\le 4 e^{\alpha \pas}[f]_1^2 c_2\max(c_2,c_4)\pas^{\frac{\bea}{2}}\begin{cases} \pas^{\frac{\bea}{2}} &\textnormal{if 
 $s-u\le \frac{\bea}{2\alpha}|\log \pas|.$}\\
 e^{-\alpha (s-u)} &\textnormal{if 
 $s-u\ge \frac{\bea}{2\alpha}|\log \pas|.$}
 \end{cases}
 \end{equation*}
Now, let us plug this inequality into \eqref{eq:expvar}. Setting $\cfrak_\pas=8 e^{\alpha \pas}c_2\max(c_2,c_4)$,
\begin{align*}
\mathrm{Var}\l(\frac{1}{T-\tau}\int_\tau^{T} G_s^{\pas}ds\r)
&\le \frac{\cfrak_\pas [f]_1^2}{(T-\tau)^2} \left(\int_{{\tau}}^{{T}} \int_u^{u+\frac{\bea}{2\alpha}\log(\frac{1}{\pas})}  \pas^{\bea}  \mathrm{d}s\mathrm{d}u 
 + \int_{{\tau}}^{{T}} \int_{u+\frac{\bea}{2\alpha}\log(\frac{1}{\pas})}^{{T}}   e^{-\alpha (s-u)}  \mathrm{d}s \mathrm{d}u\right)
\\ & \le \frac{\cfrak_\pas [f]_1^2}{(T-\tau)^2}\left( \int_{{\tau}}^{{T}} \frac{\bea\pas^{\bea}}{2\alpha}\log\l(\frac{1}{\pas}\r) du +  \int_{{\tau}}^{{T}} \frac{\pas^\bea}{\alpha} \mathrm{d}u\right)
\\ & \le \frac{ \cfrak_\pas(\bea/2+1)[f]_1^2}{\alpha(T-\tau)} \pas^{\bea} \log\l(\frac{1}{\pas}\r).
\end{align*}
The result follows by using that $\bea\le 2$ and $\alpha\pas\le 1$.
%
\end{proof}
{We are now ready to bound the variance of the multilevel procedure. This is the purpose of the next proposition.}
\begin{prop}\label{varcontr}
 Let $f$ be a Lipschitz function. Assume that $\HUN$, $\HDEUX$ and $\HQU$ hold for some given $\bea\in[1,2]$, $\pas_0\in[0,\eta_0]$ and $x_0\in\ER^d$ with $\alpha\eta_0\le 1$. Assume that for every $\ind\in\{0,\ldots,\lev\}$, 
 $$\pas_\ind=\pas_0 2^{-\ind}\quad\textnormal{and}\quad  T_\ind =   T_0 2^{-\frac{1+\bea}{2}\ind},$$
 and that $\tau$ is a positive number satisfying $\tau\le T_\lev/2$.
Then, 
\begin{equation*}
\mathrm{Var}(\mathcal{Y}(\lev,\l(\pas_\ind\r)_\ind,\tau,\l(T_\ind\r)_\ind,f) \le \cuniv  \frac{[f]_1^2}{\alpha T_0}\left(\4^2+
\max(\2^2,\2\4)\pas_0^\bea\log\left(\pas_0^{-1}\right)
\left((\bea-1)^{-2}\wedge R^2 \right)\right),
\end{equation*}
where $\cuniv$ is a universal constant.
\end{prop}
\begin{rem} When $\bea=1$, $(\bea-1)^{-2}\wedge R^2=R^2$. 
\end{rem}

\begin{proof}
At the price of replacing $f$ by $f-f(x_0)$ (which does not change the variance), we can assume without loss of generality that  $f(x_0)=0$. 
In view of the decomposition  obtained in  \eqref{eq:decompofvar}, we apply Lemma \ref{propvarlvl} for each level $\ind\in\{1,\ldots,\lev\}$ with $T=T_\ind$ and $\pas=\pas_{\ind-1}$. Using that $(T_\ind-\tau)^{-1}\le 2/T_\ind $ for every $\ind\in\{1,\ldots,\lev\}$, we obtain:
\begin{align}\nonumber
\mathrm{Var}(\mathcal{Y}(\lev,\l(\pas_\ind\r)_\ind,&\tau,\l(T_\ind\r)_\ind,f)) \le \mathrm{Var}\l(\frac{1}{T_0-\tau}\int_{\tau}^{T_0} f(\bar{X}_{\un{s}_{\pas_0}}^{\pas_0,x_0}) ds\r) 
 +
2\mathfrak{c}_{var} [f]_1^2 \sum_{\ind=1}^\lev \frac{\pas_{\ind-1}^\bea \log\l(\frac{1}{\pas_{\ind-1}}\r)}{{T_\ind}} \\
& \le  \mathrm{Var}\l(\frac{1}{T_0-\tau}\int_{\tau}^{T_0} f(\bar{X}_{\un{s}_{\pas_0}}^{\pas_0,x_0})ds \right) +
 \frac{2\mathfrak{c}_{var} [f]_1^2 \pas_0^\bea}{T_0} \sum_{\ind=1}^\lev 2^{\frac{1-\bea}{2}\ind}\left(\log\left(\frac{1}\pas_0\right)+\ind\right)\nonumber \\
 &\le  \mathrm{Var}\l(\frac{1}{T_0-\tau}\int_{\tau}^{T_0} f(\bar{X}_{\un{s}_{\pas_0}}^{\pas_0,x_0})ds \right) +
 \frac{2\mathfrak{c}_{var} [f]_1^2 \pas_0^\bea\log\left(2\pas_0^{-1}\right)}{T_0} \sum_{\ind=1}^\lev \ind 2^{\frac{1-\bea}{2}\ind}.\label{eq:vardec24}
\end{align}
When $\bea>1$,  one can check that
$$\sum_{\ind=1}^\lev \ind 2^{\frac{1-\bea}{2}\ind}\le \sum_{\ind\ge1} \ind 2^{\frac{1-\bea}{2}\ind} =\frac{2^{(1-\bea)/2}}{(1- 2^{(1-\bea)/2})^2}\le \frac{4}{(\log 2)^2} 2^{\frac{\bea-1}{2}} (\bea-1)^{-2}\le \frac{4\sqrt{2}}{(\log 2)^2} (\bea-1)^{-2},$$
where in the second inequality, we used that $1-e^{-x}\ge x e^{-x}$ for any $x\ge0.$ When $\bea\ge 1$,  $\sum_{\ind=1}^\lev \ind 2^{\frac{1-\bea}{2}\ind}\le \frac{\lev(\lev+1)}{2}$ so that
$$ \frac{2\mathfrak{c}_{var} [f]_1^2 \pas_0^\bea\log\left(2\pas_0^{-1}\right)}{T_0} 
\sum_{\ind=1}^\lev \ind 2^{\frac{1-\bea}{2}\ind}\le\cuniv\mathfrak{c}_{var}\frac{\pas_0^\bea\log\left(\pas_0^{-1}\right)}{T_0}
\left((\bea-1)^{-2}\wedge R^2\right).
$$
where $\cuniv$ is a universal constant.\\

Now, it remains to bound the first term of \eqref{eq:vardec24}
 (with the help of ergodicity arguments).   By similar arguments as in the proof of Lemma \ref{propvarlvl} (and with the notation $\phi$ introduced in \eqref{phietax}),
\begin{align*}
 \mathrm{Var}\Big(\frac{1}{T_0-\tau}\int_{\tau}^{T_0} f(&\bar{X}_{\un{s}_{\pas_0}}^{\pas_0,x_0}) ds\Big) 
=  \frac{2}{({T_0}-{\tau})^2} \int_{{\tau}}^{{T_0}} \int_{{\tau}}^{{T_0}} {\rm Cov}(f(\bar{X}_{\underline{s}_{\pas_0}}^{\pas_0,x_0}) ,
f( \bar{X}_{\un{u}_{\pas_0}}^{\pas_0,x_0} ))\mathrm{d}s\mathrm{d}u\\ 
& = \frac{2}{({T_0}-{\tau})^2} \int_{{\tau}}^{{T_0}} \int_{u}^{{T_0}}  \ES\l[\left(\phi(\pas_0,\un{s}_{\pas_0}-\un{u}_{\pas_0},\bar{X}_{\un{u}_{\pas_0}}^{\pas_0})+\phi(\pas_0,\un{s}_{\pas_0},x_0)\right) f( \bar{X}_{\un{u}_{\pas_0}}^{\pas_0,x_0} )\r]   \mathrm{d}s \mathrm{d}u
\\ 
& \le \frac{2[f]_1}{({T_0}-{\tau})^2} \int_{{\tau}}^{{T_0}} \int_{u}^{{T_0}}  \ES[\left(\1(\bar{X}_{\un{u}_{\pas_0}}^{\pas_0,x_0})+\1(x_0)\right) f( \bar{X}_{\un{u}_{\pas_0}}^{\pas_0,x_0} )]  e^{- \alpha(  \un{s}_{\pas_0}-\un{u}_{\pas_0}) } \mathrm{d}s \mathrm{d}u,
\end{align*}
where in the last line, we used Assumption $\HUN$. By Assumption $\HQU$, 
$$\sup_{u\ge0} \left(\|\1(\bar{X}_{\un{u}}^{\pas_0,x_0})\|_2+ \1(x_0)\right)\le 2\4.$$
As well, $f$ being a Lipschitz continuous function such that $f(x_0)=0$, we have $f(x)\le [f]_1|x-x_0|$ and by $\HQU$, we deduce that
 $$\sup_{u\ge0} \|f( \bar{X}_{\un{u}_{\pas_0}}^{\pas_0,x_0} )\|_2\le 2 \4 [f]_1.$$
Hence, by Cauchy-Schwarz inequality, we easily deduce that 
\begin{align*}
\mathrm{Var}\l(\frac{1}{T_0-\tau}\int_{\tau}^{T_0} f(\bar{X}_{\un{s}_{\pas_0}}^{\pas_0,x_0}) ds\r) 
\le  \frac{8e^{\alpha\pas_0}[f]_1^2 \4^2}{\alpha(T_0-\tau)}
 \le  \frac{16 e^{\alpha\pas_0} [f]_1^2 \4^2}{\alpha T_0},
 \end{align*}
since $\tau\le \frac{T_\lev}{2}\le \frac{T_0}{2}$.
\end{proof}

\section{Proof of Theorem \ref{maintheo}}\label{sec:maintheo}
In the next proposition, we provide a quantitative estimate of the complexity cost ${\cal C}_\varepsilon({\cal Y})$ (which corresponds to the number of iterations which are necessary to obtain $\| {\cal Y}(f)-\pi(f) \|_2\le \varepsilon$) and in particular of the constant $\mathfrak{C}$ defined in  Theorem \ref{maintheo}. In particular, Theorem \ref{maintheo} is a corollary of this result.

\begin{prop}\label{prop:theo1precis} Let the assumptions of Theorem \ref{maintheo} be in force. For a given $\varepsilon\in(0,1]$, let $R_\varepsilon$, $(\gamma_r)_{r=0}^{R_\varepsilon}$ and
$(T_r)_{r=0}^{R_\varepsilon}$ be defined by \eqref{eq:choixpar} with $\pas_0 \in(0,\eta_0]$.  Then,\\

\noindent $(i)$ If $\tau\in [\tau_1|\log(\varepsilon)|\wedge \frac{T_{R_\varepsilon}}{2}, \frac{T_{R_\varepsilon}}{2}]$ with $\tau_1>(1+\bea-2\delta)/(2\alpha\delta)$, there exist some positive constants $\mathfrak{C}_1$ and $\mathfrak{C}_2$ (independent of $\varepsilon$) such that 
\eqref{eq:pluspetitqueepsilon} and \eqref{relatedcosteps} hold true with
$\mathfrak{C}_2= {\cbea} \gamma_0^{-1}{{\Tfrak}}$ (where ${\cbea}$ defined in Theorem \ref{maintheo}).
\\

\noindent $(ii)$ Assume that the parameters given in \eqref{eq:choixpar} satisfy:
\begin{align}\label{eq:choiceoftheparameters}
& r_0\,{\ge}\,1\vee( \3 \pas_0^\delta), \quad\textnormal{and} \quad  \Tfrak\,{\ge}\, {\Tfrak_0}:= \frac{{\dbea}}{\alpha}{\max\l(\2^2 \pas_0^\bea\log(\pas_0^{-1}),\4^2\right)},
 \end{align}
 with ${\dbea}=(\bea-1)^{-2}$ if $\bea>1$ and ${\dbea}=1$ if $\bea=1$.
Set  
\begin{equation}\label{tauuntaudeux}
\tauun=\frac{ 1+ \bea - 2\delta}{\alpha \delta},\quad {\taudeux= 0\vee\frac{1}{\alpha}\log\l( {r_0^{\frac{1+\bea}{2\delta}}}({\dbea}\4)^{-1}\r)},\quad \varepsilon_0:=\max\{\varepsilon\in(0,1], \tauun|\log \varepsilon|+\taudeux\le \frac{1}{2}T_{R_\varepsilon}\}.
\end{equation}
Let  $\tau\in[\tauun|\log \varepsilon|+\taudeux\le \frac{1}{2}T_{R_\varepsilon}]$. Then,  \eqref{eq:pluspetitqueepsilon} and \eqref{relatedcosteps} hold true for any $\varepsilon\in(0,\varepsilon_0)$ with
$\mathfrak{C}_1\lesssim_{uc} 1$. 
%
%
\end{prop}

\begin{proof} At the price of replacing $\varepsilon$ by $\varepsilon/[f]_1$, we assume in whole the proof that $[f]_1=1$.\\

\noindent $(i)$ First, by \eqref{eq:cout1}, one remarks that if the parameters satisfy \eqref{eq:choixpar}, then, the related complexity cost ${\cal C}_{\varepsilon}({\cal Y})$
satisfies for every $\varepsilon\in(0,1]$,
\begin{equation}\label{eq:cost2}
\mathcal{C}_{\varepsilon}(\mathcal{Y}) \le
\begin{cases} \left((1+\frac{3}{2}(2^{\frac{\bea-1}{2}}-1)^{-1}\right)\frac{\Tfrak}{\pas_0} \varepsilon^{-2}&\textnormal{if $\bea>1$}\\
\frac{5}{2}\frac{\Tfrak}{\pas_0} \varepsilon^{-2}R_\varepsilon^3 &\textnormal{if $\bea=1$,}
\end{cases}
\end{equation} 
 This leads to the {value of  $\mathfrak{C}_2$}. On the other hand, we deduce from Proposition \ref{maintheointermed}\footnote{Note that by construction, $\tau\le T_{R_\varepsilon}/2$.} that a positive constant $\mathfrak{C}_1$ exists such that \eqref{eq:pluspetitqueepsilon}  holds true for any $\varepsilon\in(0,1]$ {if there exist some finite constants $ {\mathfrak{C}_{1,1}}$, ${\mathfrak{C}_{1,2}}$ and ${\mathfrak{C}_{1,3}}$ such that}
\begin{equation}\label{eq:allcontrols}
\begin{cases}
(a)\quad \3  \pas_{\lev_\varepsilon}^\delta\le {\mathfrak{C}_{1,1}}\varepsilon\\
(b)\quad\frac{\1(x_0)}{\alpha T_0} e^{-{\alpha \tau}}2^{\frac{1+\bea}{2}\lev_\varepsilon }\le {\mathfrak{C}_{1,2}}\varepsilon\\
(c)\quad
 \frac{1}{\alpha T_0}\left(\4^2+
\max(\2^2,\2\4)\pas_0^\bea\log\left(\pas_0^{-1}(\bea-1)^{-2}\wedge R_\varepsilon^2\right)
\right)
\le{\mathfrak{C}_{1,3}}\varepsilon^2,
\end{cases}
\end{equation}
where $\lev_\varepsilon=\lceil \delta^{-1} \log_2(r_0 \varepsilon^{-1})\rceil$. 
Note that we used that under the assumptions, $\alpha\pas_0\le 1$.
For $(a)$, the result is obvious since by construction,
\begin{equation}\label{eq:gestionc11}
\3 \pas_{\lev_\varepsilon}^\delta\le  \3\pas_0^{\delta} 2^{-\log_2( r_0 \varepsilon^{-1})}=:{\mathfrak{C}_{1,1}}\varepsilon\quad\textnormal{with} \quad {\mathfrak{C}_{1,1}}=\3\pas_0^{\delta} r_0^{-1}.
\end{equation}
{For $(c)$, using the elementary inequality $2c_2c_4\le c_2^2+c_4^2$ and the fact that $\sup_{x\in(0,1],\bea\in [1,2]} {x^\bea|\log x|}\le 1$, we remark that
$$
\4^2+
\max(\2^2,\2\4)\pas_0^\bea\log(\pas_0^{-1})\left((\bea-1)^{-2}\wedge R_\varepsilon^2\right)\lesssim_{uc} \max\l(\4^2, \2^2\pas_0^\bea\log(\pas_0^{-1})\right)\left((\bea-1)^{-2}\wedge R_\varepsilon^2\right).$$}
{Then, owing to the definition of $T_0$, we deduce that $(c)$ holds true with
\begin{equation}
{\mathfrak{C}_{1,3}}\lesssim_{uc} \begin{cases}
2 (\bea-1)^{-2} {(\alpha{\Tfrak)}}^{-1}{ \max\l(\4^2, \2^2\pas_0^\bea\log(\pas_0^{-1})\right)}
 &\textnormal{if $\bea>1$}\\
{2{(\alpha{\Tfrak)}^{-1}} \max\l(\4^2, \2^2\pas_0^\bea\log(\pas_0^{-1})\right)} &\textnormal{if $\bea=1$}.
\end{cases}
\label{eq:ccccc}
\end{equation}}
Note that for $\bea=1$, we used that $(\bea-1)^{-2}\wedge {\lev_\varepsilon^2}=R_\varepsilon^2$ and that $T_0=\Tfrak\varepsilon^{-2} R_\varepsilon^2$.
Finally, for $(b)$, first remark that $2^{\frac{1+\bea}{2}\lev_\varepsilon }\le (2^\delta r_0 \varepsilon)^{{\frac{1+b}{2\delta}}}$. Then, if $\tau:=\tau(\varepsilon)\ge \tau_1 |\log(\varepsilon)|$ with $\tau_1\ge0$, we get
{\begin{equation}\label{wskljlkjdskz}
\varepsilon^{-1}\frac{\1(x_0)}{\alpha T_0} e^{-{\alpha \tau}}2^{\frac{1+\bea}{2}R_\varepsilon}\le
 \frac{\1(x_0)}{\alpha\Tfrak} (2^\delta r_0 )^{{\frac{1+b}{2\delta}}}\varepsilon^{1+\alpha\tau_1-\frac{1+\bea}{2\delta}}.
\end{equation}
In the case $\bea=1$, we used that $R_\varepsilon\ge 1$.}
Set $\kappa=1+\alpha\tau_1-\frac{1+\bea}{2\delta}$.  Since $\tau_1>(1+\bea-2\delta)/(2\alpha\delta)$, we have $\kappa>0$. Thus,
$$
\sup_{\varepsilon\in(0,1]}\varepsilon^{-1}\frac{\1(x_0)}{\alpha T_0} e^{-{\alpha \tau}}2^{\frac{1+\bea}{2}R_\varepsilon}\le{ \frac{\1(x_0)}{\alpha\Tfrak} (2^\delta r_0 )^{{\frac{1+\bea}{2\delta}}}} <+\infty.
$$

\noindent This implies that ${\mathfrak{C}_{1,2}}$ is finite as soon as $\tau(\varepsilon)\ge \tau_1 |\log(\varepsilon)|$ for any $\varepsilon\in(0,1]$. This result easily extends to the case where  $\liminf_{\varepsilon\rightarrow0} \frac{\tau(\varepsilon)}{\tau_1|\log(\varepsilon)|}>0$ (with the convention $1/0=+\infty$ if $\tau_1=0$). Thus, the result is still true if $\tau\in[\tau_1|\log \varepsilon|\wedge T_{R_\varepsilon}/2, T_{R_\varepsilon}/2]$. Actually, under $\HTROIS$, one can check 
that $|\log(\varepsilon)|=o(T_{R_\varepsilon})$.

$(ii)$ First, let us remark that under the assumptions of this statement, $\tau\le T_{R_\varepsilon}/2$ for any $\varepsilon\in(0,\varepsilon_0]$.  It now remains to check that ${\mathfrak{C}_{1,1}}$, ${\mathfrak{C}_{1,2}}$ and ${\mathfrak{C}_{1,3}}$ defined in $(i)$ are bounded by universal constants.

For $(a)$, this is obvious by \eqref{eq:gestionc11} (since {${\mathfrak{C}_{1,1}}\le 1$ when $r_0\ge 1\vee c_3\pas_0^\delta$}). For $(c)$, one also remarks that $\Tfrak$ is defined in such a way that ${\mathfrak{C}_{1,3}}$ is bounded by a universal constant. Finally, for $(b)$, one  can check (with a slight adaptation of \eqref{wskljlkjdskz})  that when $\tau\ge \tau_1|\log(\varepsilon)|+\tau_2$ with $\tau_1=(1+\bea-2\delta)/(2\alpha\delta)$ then,
{\begin{equation}\label{eq:lastoneouf}
\varepsilon^{-1}\frac{\1(x_0)}{\alpha T_0} e^{-{\alpha \tau}}2^{\frac{1+\bea}{2}R_\varepsilon}\le
 \frac{\1(x_0)}{\alpha\Tfrak} 2^{{\frac{1+\bea}{2}}} r_0^{\frac{1+\bea}{2\delta}} e^{-\alpha \tau_2}.\end{equation}
}
Thus, ${\mathfrak{C}_{1,2}}$ is bounded by a universal constant if
{$$\tau_2\ge\frac{1}{\alpha}\log\l( \frac{\1(x_0)r_0^{\frac{1+\bea}{2\delta}}}{\alpha\Tfrak}\r).$$}
Now, since $c_1(x_0)\le c_4$ and  $\alpha\Tfrak={\dbea}{\max\l(\4^2, \2^2\pas_0^\bea\log(\pas_0^{-1})\right)}\ge {\dbea} c_4^2$, we can slightly simplify the condition by taking
{$$\tau_2= 0\vee\frac{1}{\alpha}\log\l( r_0^{\frac{1+\bea}{2\delta}}({\dbea}\4)^{-1}\r).$$}
\begin{rem} \label{rem:withnonexplicitconstants} In the sequel, we usually know the constants $c_2$, $c_3$ and $c_4$ up to some universal constants. More precisely, we will build our algorithm with
$\tilde{c}_i=\lambda_i c_i$ where $\lambda_1$, $\lambda_2$ and $\lambda_3$ denote some universal positive constants. A careful reading of the proof shows that with the new parameters $$\tilde{r}_0{\,\ge\,} 1\vee (\tilde{c}_3\gamma_0^\delta),\quad \tilde{R}_\varepsilon=\lceil \delta^{-1}\log_2(\tilde{r}_0 \varepsilon^{-1})\rceil,\quad {\tilde{\Tfrak}{\,\ge\,}  \frac{{\dbea}}{\alpha}\max(\tilde{c}_2^2 \pas_0^\bea\log(\pas_0^{-1}),\tilde{c}_4^2)}$$
and $\tilde{\tau}=\tau_1|\log \varepsilon|+\tilde{\tau}_2$ with,
{$$\tilde{\tau}_2= 0\vee\frac{1}{\alpha}\log\l( \tilde{r}_0^{\frac{1+\bea}{2\delta}}({\dbea}\tilde{c}_4)^{-1}\r),$$}
the conclusion of Proposition \ref{prop:theo1precis}$(ii)$ (and thus of Theorem \ref{maintheo}$(ii)$) is still true with  $\tilde{\mathfrak{C}}_2= {\cbea} \gamma_0^{-1}\tilde{\Tfrak}$ and with a new universal constant $\tilde{\mathfrak{C}}_1$.\\

\noindent For the sake of completeness, let us give some arguments. First,   the fact that $\tilde{\mathfrak{C}}_2= {\cbea} \gamma_0^{-1}\tilde{\Tfrak}$  follows from  \eqref{eq:cost2}. Then, to prove that $\tilde{\mathfrak{C}}_1\lesssim_{uc} 1$, one has to check  that the controls of \eqref{eq:allcontrols} are still true with the new parameters of the algorithm (with some new  universal constants $\tilde{\mathfrak{d}}_{i}$, $i=1,2,3$). For $(a)$, we have
$\3 \pas_{\lev_\varepsilon}^\delta\le \tilde{\mathfrak{d}}_{1}\varepsilon$ with  $\tilde{\mathfrak{d}}_{1}=\3\pas_0^{\delta} \tilde{r}_0^{-1}$ and it is easy to check (considering separately the cases $\tilde{c}_3\gamma_0^\delta\le 1$ and $\tilde{c}_3\gamma_0^\delta\ge 1$) that $\tilde{\mathfrak{d}}_{1}\le \cuniv= \max(\lambda_1,\lambda_1^{-1})$.
For $(c)$, one checks that the formula \eqref{eq:ccccc} is still correct replacing $\Tfrak$ by $\tilde{\Tfrak}$.  If $\bea>1$, this means that $(c)$ holds with 
$${\mathfrak{C}_{1,3}}=2 {\max(\tilde{c}_2^2 \pas_0^\bea\log(\pas_0^{-1}),\tilde{c}_4^2)^{-1}}{\max({c}_2^2 \pas_0^\bea\log(\pas_0^{-1}),{c}_4^2)}\le \cuniv=\frac{2}{\min (\lambda_2^2,\lambda_4^2)}$$
and the same bound occurs with $\bea=1$. Finally, for $(b)$, using that $\1(x_0)\le c_4=\lambda_4^{-1}\tilde{c}_4$ and that {$\alpha\tilde{\Tfrak}\ge {\dbea} \tilde{c}_4^2$}, we can replace Inequality \eqref{eq:lastoneouf} by :
{\begin{equation*}
\varepsilon^{-1}\frac{\1(x_0)}{\alpha \tilde{T}_0} e^{-{\alpha \tilde{\tau}}}2^{\frac{1+\bea}{2}\tilde{R}_\varepsilon}\lesssim_{uc}
 \lambda_4^{-1}  ({\dbea}\tilde{c}_4)^{-1}\tilde{r}_0^{\frac{1+\bea}{2\delta}} e^{-\alpha \tilde{\tau}_2},\end{equation*}}
and the definition of $\tilde{\tau}_2$ is exactly what we need to bound $\tilde{\mathfrak{d}}_2$ by a universal constant. 

\end{rem}

\end{proof}

\section{Proof of the results in the strongly convex setting}\label{sec:maintheo23}
This section is divided into two parts. In the first one, we prove that $\Cs$ leads to  a series of bounds which in turn imply $\HUN$, $\HDEUX$, $\HTROIS$ and $\HQU$. Then, in the second one (Section \ref{sec:resumedespreuves}), we  thus derive our main results from \cref{maintheo}.
\subsection{Contraction/Stability/Confluence bounds under \texorpdfstring{$\Cs$}{Cs} } 
\subsubsection{\texorpdfstring{$\HUN$}{H1} and \texorpdfstring{$\HQU$}{H4} under \texorpdfstring{$\Cs$}{Cs}}
\begin{lem}\label{lem:boundEuler}
Assume $\Cs$ and  $b$ $L$-Lipschitz {with $0<\alpha\le L$}. Let $x^\star\in\ER^d$. Then,\\

\noindent (i) For every  $(\pas,t,x)\in (0,\laminf/(2\blip^2)]\times \ER_+\times\ER^d$,
\begin{equation}\label{eq:lyapdiscret}
 \ES[|\Xge_{t}^{\pas,x}-x^\star|^2]\le |x-x^\star|^2 e^{-\frac{\alpha}{2}t}+|b(x^\star)|^2\l(\frac{1}{\blip^2} + \frac{2}{\laminf^2}\r)+\frac{2\sigma^2 d}{\laminf}.
 \end{equation}
In particular, the Euler scheme with step $\pas$ admits a unique invariant distribution $\pi^\pas$ as soon as $\pas\in(0,\laminf/(2\blip^2)]$\footnote{In fact, looking carefully into the proof, one can check that existence of $\pi^\pas$ may extend to $\pas \in(0,2\laminf/\blip^2]$.} 
$$\sup_{\pas\in \laminf/(2\blip^2)]} \pi^\pas(|.-x^\star|^2)\le 2|b(x^\star)|^2\l(\frac{1}{\blip^2} + \frac{2}{\laminf^2}\r)+\frac{\sigma^2 d}{\laminf}.$$
\noindent (ii) For all $x,y\in\ER^d$, for all $\pas\in(0,\laminf/(2\blip^2)]$, for all $t\ge0$,
$$  \ES[|\bar{X}_{t}^{\pas,x}-\bar{X}_{t}^{\pas,y}|^2]\le |x-y|^2 e^{-\alpha t},$$
and,
$${\cal W}_2(\bar{X}_{{\un{t}}}^{\pas,x},\pi^\pas)\le{\cal W}_2(\delta_x,\pi^\pas)e^{-\laminf {\un{t}}}$$
with,
$${\cal W}_2(\delta_x,\pi^\pas)\le c_1(x):=|x-x^\star|+\sqrt{|b(x^\star)|^2\l(\frac{1}{\blip^2} + \frac{2}{\laminf^2}\r)+\frac{2\sigma^2 d}{\laminf}}.$$

\noindent (iii) As a consequence, setting $\eta_0=\laminf/(2\blip^2)$, $\HUN$ holds with $c_1$ defined above and $\HQU$ holds with 
$c_4^2\lesssim_{uc} {\alpha^{-2}}|b(x_0)|^2+\sigma^2\alpha^{-1}{d}$.
\end{lem}

\begin{rem} 
 Let us remark that the $L^2$-bounds of $(ii)$ rely on pathwise controls of the Euler schemes. Furthermore, note that if $b(x_0)=0$, the dependence {on} $\alpha$ is improved.  This is of interest in the case where $b=-\nabla U$ and $U$ has a  minimum (unique under $\Cs$) which is known.
\end{rem}
\begin{proof}
$(i)$ Let $(\bar{X}_t^x)_{t\ge0}$ denote the Euler scheme with step $\pas$ starting from $x$. Let $\un{t}=\pas\max\{k\in\mathbb{N}, k\pas\le t\}$ and set $\eta_t=t-\un{t}$. For any $t\ge0$, we have
$$ \bar{X}_t^x-x^\star=\bar{X}_{\un{t}}^x-x^\star+\eta_t b(\bar{X}_{\un{t}}^x)+\sigma (B_t-B_{\un{t}}).$$
Using that the Brownian motion has centered and independent increments, we have for every $t\ge0$,
$$\ES[|\bar{X}_t^x-x^\star|^2]=\ES[|\bar{X}_{\un{t}}^x-x^\star|^2]+2\eta_t \ES[\langle \bar{X}_{\un{t}}^x-x^\star,b(\bar{X}_{\un{t}}^x)\rangle]+\eta_t^2\ES[|b(\bar{X}_{\un{t}}^x)|^2]+\eta_t\sigma^2 d.$$
Adding and substracting $2 \eta_t \ES[ \langle \bar{X}_{\un{t}}^x, b(x^\star) \rangle] $ in the preceding equality and $b(x^\star)$ in $\ES[|b(\bar{X}_{{\un{t}}}^x)|^2]$ we get 
\begin{align*}\ES[|\bar{X}_{t}^x-x^\star|^2]&\le  \ES[|\bar{X}_{{\un{t}}}^x-x^\star|^2]+2\eta_t \ES[\langle \bar{X}_{{\un{t}}}^x-x^\star,b(\bar{X}_{{\un{t}}}^x)-b(x^\star)\rangle]
\\ & +\eta_t^2\ES[|b(\bar{X}_{{\un{t}}}^x)-b(x^\star)|^2]+\eta_t^2|b(x^\star)|^2 +2 \eta_t \ES[\langle \bar{X}_{{\un{t}}}^x-x^\star, b(x^\star) \rangle] + \eta_t\sigma^2 d.
\end{align*}
Using that $b$ is $L$-Lipschitz, Assumption $\Cs$ and  the elementary inequality $\langle u,v\rangle\le (2\laminf)^{-1}|u|^2+(\laminf/2)|v|^2$ (with $u=b(x^\star)$ and $v=\bar{X}_{{\un{t}}}^x-x^\star$), this yields:
\begin{equation}\label{eq:controlforanyt}
\ES[|\bar{X}_{t}^x-x^\star|^2]\le \ES[|\bar{X}_{{\un{t}}}^x-x^\star|^2]\l(1-\laminf\eta_t+\eta_t^2 \blip^2\r)+\eta_t \l(|b(x^\star)|^2(\eta_t + \laminf^{-1})+\sigma^2 d\r).
\end{equation}
If $\pas\in(0,\laminf/(2\blip^2)]$, then,  $1-\laminf\pas+\pas^2 \blip^2\le 1-\frac{1}{2}\laminf \pas$. Hence, setting $u_k=\ES[|\bar{X}_{k\pas}^x-x^\star|^2]$, we get
$$ u_{k+1}\le u_k\l(1-\frac{\laminf \pas}{2}\r)+\pas \l(|b(x^\star)|^2\l(\frac{\laminf}{2\blip^2} + \frac{1}{\laminf}\r)+\sigma^2 d\r),$$
and an induction leads to 
$$ u_k\le |x-x^\star|^2\l(1-\frac{\alpha\pas}{2}\r)^k+\frac{2}{\alpha} \l(|b(x^\star)|^2\l(\frac{\laminf}{2\blip^2} + \frac{1}{\laminf}\r)+\sigma^2 d\r).$$
Then, Inequality \eqref{eq:lyapdiscret}
 follows for $t=k\pas$ by using that $1-x\le e^{-x}$ for $x\ge0$, and extends to any $t\ge0$ by \eqref{eq:controlforanyt}.\\
 
 \noindent Inequality \eqref{eq:lyapdiscret} implies in particular that $\sup_{t\ge0} \ES[|\bar{X}_t-x^\star|^2]<+\infty$, which  in turn classically  ensures the existence of $\pi^\pas$ and the fact that $\pi^\pas(|.-x^\star|^2)\le \limsup_{t\rightarrow+\infty}<\ES[|\bar{X}_t-x^\star|^2]$. Uniqueness is obvious since the diffusion is not degenerated. 

\noindent $(ii)$ With the same notations as in $(i)$,
$$\bar{X}_{t}^x-\bar{X}_{t}^y=\bar{X}_{\un{t}}^x-\bar{X}_{\un{t}}^y+\eta_t (b(\bar{X}_{\un{t}}^x)-b(\bar{X}_{\un{t}}^y)).$$
Expanding the square of the right-hand member and using Assumption $\Cs$, this yields:
\begin{align*}
|\bar{X}_{t}^x-\bar{X}_{t}^y|^2\le |\bar{X}_{{\un{t}}}^x-\bar{X}_{{\un{t}}}^y|^2(1-2\laminf\eta_t)+\eta_t^2|b(\bar{X}_{{\un{t}}}^x)-b(\bar{X}_{{\un{t}}}^y)|^2.
\end{align*}  
Since $b$ is a Lipschitz continuous function, we deduce that
$$|\bar{X}_{t}^x-\bar{X}_{t}^y|^2\le |\bar{X}_{{\un{t}}}^x-\bar{X}_{{\un{t}}}^y|^2 \l(1-2\laminf\eta_t+\eta_t^2 \blip^2 \r).$$
Since $\pas\le \laminf/\blip^2$, we have $1-2\laminf\eta_t+\eta_t^2 \blip^2\le 1-\laminf\eta_t$ for any $t\ge0$. The first inequality thus follows by induction and by the inequality $1-x\le e^{-x}$ for $x\ge0$.\\

\noindent Let us consider the second inequality of $(ii)$: by the invariance of the distribution $\pi^\pas$ and  the definition of ${\cal W}_2$, we have
$${\cal W}_2({\cal L}(\bar{X}_{\un{t}}^{\pas,x}),\pi^\pas)\le \sqrt{\ES_{Y_0\sim\pi^{\pas}}[|\bar{X}_{\un{t}}^{\pas,x}-\bar{X}_{\un{t}}^{\pas,Y_0}|^2]}=\sqrt{\int \ES[|\bar{X}_{\un{t}}^{\pas,x}-\bar{X}_{\un{t}}^{\pas,y}|^2]\pi^\pas(dy)},$$
and the result follows from the previous bound.  Finally, for the last inequality of $(ii)$, one uses Minkowski inequality to obtain:
$${\cal W}_2(\delta_x,\pi^\pas)\le |x-x^\star|+\sqrt{\int |y-x^\star|^2\pi^\pas(dy)},$$
but by $(i)$ and the convergence in distribution of the Euler scheme towards  $\pi^\pas$,
$$\int |y-x^\star|^2\pi^\pas(dy)\le \limsup_{k\rightarrow+\infty}\ES[|\bar{X}_{\un{t}}^{x}-x^\star|^2]\le |b(x^\star)|^2\l(\frac{1}{\blip^2} + \frac{2}{\laminf^2}\r)+\frac{2\sigma^2 d}{\laminf}.$$

\noindent $(iii)$ This is a direct consequence of $(i)$ and $(ii)$, applied with $x^\star=x_0$ {and using that $\alpha\le L$}.
\end{proof}

\subsubsection{Proof of \texorpdfstring{$\HDEUX$}{H2} and \texorpdfstring{$\HTROIS$}{H3}}
In view of $\HDEUX$, we begin with  a fundamental ``one-step'' lemma where we consider the error between the diffusion and its discretization on one step only.  To this end, we consider for  $x,y \in \ER^d$ the couple $(\X_t^x,\Xge_t^y)_{t\ge0}$ defined by
$$
\begin{cases}
&\X_t^x=x+\int_0^t b(\X_s^x) ds+\sigma \B_t\\
&\Xge_t^y=y+t b(y)+\sigma \B_t.
\end{cases}
$$

\begin{lem} \label{prop:onestep}  Let $\pas>0$. 

\noindent \textit{(i)}
$$
\E[|\X_\pas^x-\Xge_\pas^y|^2]\le |x-y|^2 e^{-\laminf \pas}+\frac{\pas^2L^2}{{\laminf}}\left({\pas}  |b(y)|^2+{\sigma^2  d}\right).
$$

\noindent \textit{(ii)}
$$\E[|\X_\pas^x-\Xge_\pas^y|^2]\le |x-y|^2 e^{-\laminf \pas}+c_\pas(x,y)\pas^3.$$
with 
$$ c_\pas (x,y)= {\frac{2}{3}}\left(\frac{ \blip^2}{\laminf} |b(y)|^2+\frac{\sigma^4}{\laminf}   \|\Delta b\|_{2,\infty}^2
+ \sigma  \blip \|\nabla b\|_{2,\infty} \left(\sqrt{\pas} S(x,\pas) +\sigma  \sqrt{ d}\right)\right),$$
where $\|\nabla b\|_{2,\infty} $ and  $\|\Delta b\|_{2,\infty}$ are defined by \eqref{def:jaclap} and $S(x,\pas )=\sup_{u\in[0,\pas]} \E[|b(\X_u^x)|^2] ^{\frac{1}{2}}$.


\end{lem}
\begin{proof} Set
$$ F_{x,y}(t)=\frac{1}{2}\E[|\X_t^x-\Xge_t^y|^2].$$
By the Lebesgue differentiability theorem, 
\begin{align}
 F'_{x,y}(t)&=\E\left[\langle \X_t^x-\Xge_t^y, {{b}}(\X_t^x)-{{b}}(y)\rangle\right]\\
 &=\E\left[\langle \X_t^x-\Xge_t^y, {{b}}(\X_t^x)-{{b}}(\Xge_t^y)\rangle\right]+\E\left[\langle \X_t^x-\Xge_t^y, {{b}}(\Xge_t^y)-{{b}}(y)\rangle\right]\nonumber \\
 &\le -{2\laminf} F_{x,y}(t)+\E\left[\langle \X_t^x-\Xge_t^y, {{b}}(\Xge_t^y)-{{b}}(y)\rangle\right],\label{eq:basiceq}
\end{align}
where in the last line, we used $\Cs$. The sequel of the proof is then dedicated to the second part of the last line. To this end, we write
\begin{align}
\E\left[\langle \X_t^x-\Xge_t^y, {{b}}(\Xge_t^y)-{{b}}(y)\rangle\right]&= \E\left[\langle \X_t^x-\Xge_t^y, {{b}}(\Xge_t^y)-b(y+\sigma \B_t)\rangle\right]\label{eq:ggh}\\
&+\E\left[\langle \X_t^x-\Xge_t^y,b(y+\sigma \B_t)-b(y)\rangle\right].\label{eq:ggh2}
\end{align}
For the right-hand side of \eqref{eq:ggh}, we use the elementary inequality, ${|uv|\le \frac{\laminf}{4}|u|^2+\frac{1}{\laminf}|v|^2}$ to obtain
\begin{equation}\label{eq:part111}
\E\left[\langle \X_t^x-\Xge_t^y, {{b}}(\Xge_t^y)-b(y+\sigma \B_t)\rangle\right]\le \frac{\laminf}{2} F_{x,y}(t)+\frac{t^2}{{\laminf}} \blip^2 |b(y)|^2.
\end{equation}
Let us now focus on \eqref{eq:ggh2}. \\

\noindent \textbf{First inequality}: To deduce $(i)$, we use the same inequality as above which yields
$$\E\left[\langle \X_t^x-\Xge_t^y,b(y+\sigma \B_t)-b(y)\rangle\right]\le \frac{\laminf}{2} F_{x,y}(t)+\frac{\sigma^2 L^2 t d}{{\laminf}}.$$
Then, plugging it into \eqref{eq:basiceq} together with \eqref{eq:part111} yields:
$$ F'_{x,y}(t)\le  - \laminf F_{x,y}(t)+\frac{t^2}{{\laminf}} \blip^2 |b(y)|^2+\frac{\sigma^2 L^2 t d}{{\laminf}}.$$
A standard Gronwall-type argument then leads to
$$ F_{x,y}(t)\le F_{x,y}(0)e^{-\alpha t}+\int_0^t\left(\frac{s^2}{{\laminf}} \blip^2 |b(y)|^2+\frac{\sigma^2L^2 s d}{{\laminf}}\right) e^{\alpha (s-t)} ds$$
and the result follows easily by using that for $r>-1$, $\int_0^t s^r e^{\alpha (s-t)}  ds \le \frac{t^{r+1}}{r+1}$ and by setting $t=\gamma$.\\

\noindent \textbf{Second inequality}: For $(ii)$, we need to give a sharper bound of \eqref{eq:ggh2}. To this end, we again apply ItÃ´ formula to $b(y+\sigma \B_t)-b(y)$:   writing $b=(b_1,\ldots, b_d)$, we have  for each $i\in\{1,\ldots,d\}$,
$$ b_i(y+\sigma \B_t)-b_i(y)=\sigma^2 \int_0^t \Delta b_i (y+\sigma B_s) ds+\sigma\int_0^t \langle\nabla b_i(y+\sigma B_s),dB_s\rangle.$$
On the one hand, setting $\Delta b=(\Delta b_i)_{i=1}^d$,
$$\E\left[\langle \X_t^x-\Xge_t^y,\sigma^2 \int_0^t \Delta b (y+\sigma B_s) ds\rangle \right]\le \frac{\laminf}{2} F_{x,y}(t)+  \frac{\sigma^4}{{\laminf}}  t^2 \|\Delta b\|_{2,\infty}^2,$$
where 
$$ \|\Delta b\|_{2,\infty}^2=\sup_{x\in\ER^d} \sum_{i=1}^d |\Delta  b_i(x)|^2.$$
On the other hand, setting ${\cal M}_t=\int_0^t \langle \nabla b(y+\sigma B_s), dB_s\rangle$ (with $\nabla b=(\nabla b_1,\ldots,\nabla b_d)^T$) and using that ${\cal M}$ is a martingale, we get
\begin{align*}
\E\left[\l\langle \X_t^x-\Xge_t^y,\sigma\int_0^t \langle \nabla  b(y+\sigma B_s), dB_s\rangle\r\rangle\right]&=0+
\sigma\E[ \langle \int_0^t b(\X_s^x)-b(y) ds, {\cal M}_t\rangle]\\
&=\sigma\int_0^t \E[\langle b(\X_s^x)-b(y), {\cal M}_s\rangle] ds.
\end{align*}
Again by the martingale property,
\begin{align*}
\E[\langle b(\X_s^x)-b(y), {\cal M}_s\rangle]&=\langle b(x)-b(y),\E[{\cal M}_s]\rangle+\E[\langle b(\X_s^x)-b(x), {\cal M}_s\rangle]\\
&= \E[\langle b(\X_s^x)-b(x), {\cal M}_s\rangle]
\end{align*}
so that by Cauchy-Schwarz inequality,
\begin{equation}\label{eq:begal2galere}
|\E[\langle b(\X_s^x)-b(y), {\cal M}_s\rangle]| = |\E[\langle b(\X_s^x)-b(x), {\cal M}_s\rangle]|\le \blip\E[|\X_s^x-x|^2]^{\frac{1}{2}}
\E[|{\cal M}_s|^2]^{\frac{1}{2}}.
\end{equation}
But, by Minkowski and Jensen inequalities,
\begin{align*}
\E[|\X_s^x-x|^2]^{\frac{1}{2}}&\le \E[|\int_0^s b(\X_u^x) du|^2]^{\frac{1}{2}}+ \sigma \E[|\B_s|^2]^{\frac{1}{2}}\\
&\le s\sup_{u\in[0,\pas]} \E[|b(\X_u^x)|^2]^{\frac{1}{2}}  +\sigma \sqrt{s d}
\end{align*}
and for the martingale term,
$$\E[ |{\cal M}_s|^2]=\int_0^s \E[\|\nabla b(y+\sigma \B_u)\|_F^2] du$$
where for a matrix $A$, $\|A\|_F$ denotes the Frobenius norm defined by $\|A\|_F=\sum_{i,j}|A_{i,j}|^2.$
Thus,
$$\E[ |{\cal M}_s|^2]^{\frac{1}{2}} \le \|\nabla b\|_{2,\infty}\sqrt{s},  $$
where 
$$\|\nabla b \|_{2,\infty}=\sup_{x\in\ER^d}\sqrt{\left( \sum_{i=1}^d |\nabla b_i(x) |^2\right)}=\sup_{x\in\ER^d}\|\nabla b(x)\|_F$$

\noindent Thus,  we deduce from what precedes and from \eqref{eq:begal2galere} that
$$|\E[\langle b(\X_s^x)-b(y), {\cal M}_s\rangle]|\le \blip \|\nabla b\|_{2,\infty} \left(s^{\frac{3}{2}} S(x,\pas) +\sigma s \sqrt{ d}\right),$$
where $S(x,\pas )=\sup_{u\in[0,\pas]} \E[|b(\X_u^x)|^2] ^{\frac{1}{2}}$.\\

\noindent Finally, from what precedes, we deduce that
$$ F'_{x,y}(t)\le  - \laminf F_{x,y}(t)+\frac{t^2}{{\alpha}}\left(\blip^2 |b(y)|^2+{\sigma^4}   \|\Delta b\|_{2,\infty}^2
+ \sigma \alpha \blip \|\nabla b\|_{2,\infty} \left(\sqrt{t} S(x,\pas) +\sigma  \sqrt{ d}\right)\right).
$$
A standard Gronwall argument then leads to the result. 
\end{proof}

We now iterate the one-step inequalities of Lemma \ref{prop:onestep}. For a given  $(\mathcal{F}_t)_{t\ge0}$-Brownian motion, we consider the couple $\l(\X_t^x,\Xge_t^{\pas,x} \r)_{t\ge 0}$ defined by 
\begin{equation}\label{eq:couple}
\begin{cases}
&\X_t^x=x+\int_0^t b(\X_s^x) ds+\sigma \B_t\\
&\Xge_t^{\pas,x}=x+ \int_0^t b(\Xge_{\underline{s}}^{\pas,x})ds+\sigma B_{t}.
\end{cases}
\end{equation}
\begin{prop}\label{prop:L2error} Assume  $\Cs$ and that $b$ is $L$-Lipschitz {with $0<\alpha\le L$}. Let $ x^\star \in \ER^d$ and $\pas\in(0,\frac{\laminf}{2\blip^2}\wedge 1)$. Then for every $n \ge 0$ and $x\in\ER^d$,\\

\noindent $(i)$
\begin{equation*}
\l\| \X_{n\pas}^x - \Xge_{n\pas}^{\pas,x} \r\|_2^2 \le \beta_1(x)\pas,
\end{equation*}
with
$$ \beta_1(x)= {\frac{\blip^2}\alpha} |x-x^\star|^2+
{2} |b(x^\star)|^2\l(\frac{1}{\alpha}+ \frac{\blip^2}{\laminf^3}\r)+\frac{{3}\blip^2\sigma^2 d}{\laminf^2}.$$

\noindent $(ii)$ For any $n\ge0$, for every $x\in\ER^d$,
\begin{equation*}
\l\| \X_{n\pas}^x - \Xge_{n\pas}^{\pas,x} \r\|_2^2 \le \beta_2(x)\pas^2,
\end{equation*}
with 
$$ \frac{3}{2}\beta_2(x)=   \frac{2L^4}{\alpha^2}|x-x^\star|^2+\mathfrak{a}_0 |b(x^\star)|^2
+\sigma \frac{L}{\sqrt{\alpha}}\|\nabla b\|_{2,\infty}|x-x^\star|+\mathfrak{a}_1 |b(x^\star)|+\mathfrak{a}_2 d,$$
where
\begin{align*}
& \mathfrak{a}_0={4}\left(\frac{L^2}{\alpha^2}+\frac{L^4}{\alpha^4}\right),\quad 
\mathfrak{a}_1=\sigma ({\alpha}^{-\frac 12}+{\alpha^{-\frac 32}L)}\|\nabla b\|_{2,\infty},\\
&\mathfrak{a}_2=\frac{\sigma^4}{\alpha^2}\|\Delta b\|_{2,\infty}^2 d^{-1}+L\|\nabla b\|_{2,\infty} \sigma^2 d^{-\frac{1}{2}}({\alpha}^{-\frac{1}{2}}+\alpha^{-1})+\frac{2 L^4\sigma^2 }{\alpha^3}.
\end{align*}

\noindent (iii) As a consequence, $\HDEUX$ holds with 
\begin{equation}
\begin{cases}
\bea=1\textnormal{ and } \2^2\lesssim_{uc} {\blip^2\alpha^{-3}} |b(x_0)|^2+L^2\sigma^2\alpha^{-2}d\\
\bea=2\textnormal{ and } \2^2\lesssim_{uc} {(L\alpha^{-1})^4}|b(x_0)|^2+\mathfrak{a}_1|b(x_0)|+\mathfrak{a}_2 d.
\end{cases}
\end{equation}
\end{prop}

\begin{proof}
$(i)$ Set $u_n=\ES[|X_{n\pas}^x-\bar{X}_{n\pas}^x|^2]$ (so that $u_0=0$). Using the Markov property and Lemma \ref{prop:onestep}$(i)$,
we get
$$ \forall n\ge0,\quad u_{n+1}\le u_n e^{-\alpha\pas}+{\pas^2}\beta_1(x,\gamma),$$
where $\beta_1(x,\gamma)= \frac{L^2}{{\alpha}}\left(\gamma  \sup_{n\ge0}\ES_x[|b(\bar{X}_{n\pas})|^2]+\sigma^2 d\right)$.
Thus, by induction, we get  
$$ u_n\le \gamma^2 \beta_1(x,\gamma) \sum_{k=0}^{n-1} e^{-\alpha k\pas}\le \frac{\gamma}{\alpha} {\beta_1(x,\gamma)},$$
where in the last inequality, we used that $1-e^{-x}\ge 1-x$ for any $x\ge0$.
Then, it remains to control $\sup_{n\ge0}\ES_x[|b(\bar{X}_{n\pas})|^2]$. By Lemma \ref{lem:boundEuler} and the fact that $b$ is Lipschitz continuous,
 \begin{equation}\label{bxtgamma}
 \ES[|b(\Xge_{t}^{\pas,x})|^2]\le 2\bun^2\ES[|\Xge_{t}^{\pas,x}-x^\star|^2]+2|b(x^\star)|^2\le 2\blip^2 |x-x^\star|^2 e^{-\frac{\alpha}{2}t}+4|b(x^\star)|^2\l(1+ \frac{\bun^2}{\laminf^2}\r)+\frac{4(\blip\sigma)^2 d}{\laminf}.
 \end{equation}
Then, if $\pas\in(0,\laminf/(2\blip^2)]$ (so that {$\pas \blip^2/\alpha\le 1/2$}),
$${\frac{\gamma}{\alpha}} \bun^2 \sup_{n\ge0}\ES_x[|b(\bar{X}_{n\pas})|^2]\le {\blip^2} |x-x^\star|^2+
{2} |b(x^\star)|^2\l(1+ \frac{L^2}{\laminf^2}\r)+\frac{{2}(\blip\sigma)^2 d}{\laminf}$$
and 
$$\beta_1(x,\pas)\le {\blip^2 }|x-x^\star|^2+
{2} |b(x^\star)|^2\l(1+ \frac{\blip^2}{\laminf^2}\r)+\frac{{3}\blip^2\sigma^2 d}{\laminf}.$$
The first result follows.\\

\noindent $(ii)$ With the same notations and the same strategy as in $(i)$, we deduce from Lemma \ref{prop:onestep}$(ii)$ that, 
$$ u_n\le \frac{\gamma^2}{\alpha} \beta_2(x,\gamma)$$
with 
$$\beta_2(x,\gamma)\le \sup_{n\ge0} \ES[c_\pas(X_{n\gamma}^x,\bar{X}_{n\gamma}^x))].$$
By the definition of $c_\pas$, we deduce that
\begin{equation}
\begin{split}\label{eq:betadeux}
{\frac{3}{2}}\beta_2(x,\gamma)&\le \frac{ \blip^2}{\laminf} \sup_{n\ge0} \ES|b(\bar{X}_{n\gamma}^x)|^2+\frac{\sigma^4}{\laminf}   \|\Delta b\|_{2,\infty}^2+\sigma^2  \blip \|\nabla b\|_{2,\infty}\sqrt{ d}\\
&+ \sigma  \blip \|\nabla b\|_{2,\infty} \sqrt{\pas}  \sup_{n\ge0} \ES S(X_{n\gamma}^x,\pas).
\end{split}
\end{equation}
But, since $S^2(z,\pas)=\sup_{0\le t\le \pas} \ES[|b(X_t^z)|^2]$, we deduce from the Markov property and Jensen inequality that
$$\ES S(X_{n\gamma}^x,\pas)\le \sup_{n\gamma t\le (n+1)\gamma}\left(\ES_x[|b(X_t)|^2]\right)^{\frac{1}{2}}$$
so that
\begin{equation}\label{eq:sxgamma}
\sup_{n\ge0} \ES S(X_{n\gamma}^x,\pas)\le \sup_{t\ge0} \left(\ES_x[|b(X_t)|^2]\right)^{\frac{1}{2}}.
\end{equation}
By ItÃ´ formula,
$$ \ES[|X_t^x-x^\star|^2]=|x-x^\star|^2+\int_0^t 2\ES \langle X_s^x-x^\star,b(X_s^x)\rangle ds+\sigma^2 d.$$
 Writing $\langle z-x^\star,b(z)\rangle=\langle z-x^\star,b(z)-b(x^\star)\rangle +\langle z-{x^\star}, b(x^\star)\rangle$ and 
using $\Cs$ and the inequality ${\langle u,v\rangle\le \alpha/2 |u|^2+1/(2\alpha)|v|^2}$, we get
$$2\ES \langle X_s^x-x^\star,b(X_s^x)\rangle \le -\alpha |X_s^x-x^\star|^2+\frac{1}{{\alpha}}|b(x^\star)|^2.$$
Hence, a standard Gronwall-type argument leads to
$$ \ES[|X_t^x-x^\star|^2]\le |x-x^\star|^2 e^{-\alpha t}+\frac{1}{{\alpha^2}}|b(x^\star)|^2+\frac{\sigma^2 d}{\alpha}.$$
Thus, using that $b$ is $L$-Lipschitz, 
$$ \ES_x[|b(X_t)|^2]\le 2\bun^2\ES_x[|X_t-x^\star|^2]+2|b(x^\star)|^2\le 2\bun^2|x-x^\star|^2+{2}|b(x^\star)|^2\left({1}+\frac{\bun^2}{\alpha^2}\right)+\frac{2(\bun\sigma)^2 d}{\alpha},$$
which in turn implies that
$$ \sup_{n\ge0} \ES S(X_{n\gamma}^x,\pas)\le \sqrt{2}\bun |x-x^\star|+{\sqrt{2}}|b(x^\star)|({1}+{\bun}{\alpha^{-1}})+\bun\sigma \sqrt{\frac{2d}{\alpha}}.$$
Then, since {$\pas\le \frac{\alpha}{2 L^2}$},
$$ \sup_{n\ge0} \sqrt{\gamma}\ES S(X_{n\gamma}^x,\pas)\le \sqrt{\alpha} |x-x^\star|+|b(x^\star)|(\sqrt{\alpha} L^{-1}+{\alpha^{-\frac{1}{2}}})+\sigma \sqrt{d}.$$
Then, plugging the above inequality  and  \eqref{bxtgamma} into \eqref{eq:betadeux}, we obtain the announced result.

\noindent $(iii)$ To prove this last statement, we write:
$$\ES[|\bar{X}_{n\pas}^{\pas,x_0}-\bar{X}_{n\pas}^{\frac{\pas}{2},x_0}|^2]^{\frac{1}{2}}\le \ES[|\bar{X}_{n\pas}^{\pas,x_0}-X_{n\pas}^{x_0}|^2]^{\frac{1}{2}}+\ES[|X_{n\pas}^{x_0}-\bar{X}_{n\pas}^{\frac{\pas}{2},x_0}|^2]^{\frac{1}{2}}.$$
Hence, by $(i)$ applied with $x=x^\star=x_0$, $\HDEUX$ holds with $\bea=1$ and $\2=\sqrt{\beta_1(x_0)}(1+2^{-\frac{1}{2}})$. By $(ii)$ again applied with $x=x^\star=x_0$, $\HDEUX$ holds with $\bea=2$ and $\2=\frac{3}{2}\sqrt{\beta_2(x_0)}$. Then, the result respectively follows from the bounds on $\beta_1$ and $\beta_2$ previously obtained {and from the fact that $\alpha\le L$}. 

\end{proof}

\noindent Now, let us focus on $\HTROIS$. We recall that $\pi$ and $\pi^\pas$ respectively denote the invariant distributions of the diffusion and of the Euler scheme with step $\pas$.

\begin{prop}\label{prop:htrois}
Assume $\Cs$ and suppose that $b$ is Lipschitz continuous function.  Then for every $\pas\in\l(0,\frac{\laminf}{2\blip^2}\r]$, for every $x^\star\in\ER^d$,
\begin{equation*}
{\cal W}_1(\pi, \pi^\pas)  \le \begin{cases} \sqrt{\beta_1(x^\star)\gamma}&\textnormal{if $\bea=1$ and $\delta=1/2$}\\
\sqrt{\beta_2(x^\star)}\, \gamma&\textnormal{if $\bea=2$ and $\delta=1$}.
\end{cases}
\end{equation*}
where $\beta_1$ and $\beta_2$ are defined in Proposition \ref{prop:L2error}. Thus, $\HTROIS$ holds with 
\begin{equation*}
 \3^2=\begin{cases}\frac{1}{2}(\frac{1}{\alpha}+\frac{\blip^2}{\alpha^3})\inf_{x\in\ER^d} |b(x)|^2+ \frac{L^2\sigma^2 d}{2\alpha^2}&\textnormal{if $\delta=\frac{1}{2}$ and $\bea=1$,} \\
((L\alpha^{-1})^{-2}+(L\alpha^{-1})^4)\inf_{x\in\ER^d} |b(x)|^2+\mathfrak{a}_1\inf_{x\in\ER^d} |b(x)|+\mathfrak{a}_2 d&\textnormal{if $\delta=1$ and $\bea=2$,}
\end{cases}
\end{equation*}
where $\mathfrak{a}_1$ and $\mathfrak{a}_2$ are defined in Proposition \ref{prop:L2error}$(iii)$.
\end{prop}
\begin{rem} Even though $\HTROIS$, is an assumption related to the weak error, we chose here to prove it with the nice strong error bounds obtained in Proposition \ref{prop:L2error}. This approach is certainly specific to the setting given by Assumption $\Cs$ and sharper weak error expansions should be used in more general settings (see for instance \cite[Theorem 1]{Flammarion_Wainwright}).
\end{rem}
\begin{proof}
Let $x \in \ER^d$, and $n \ge 0$. Let $f$ be a Lipschitz continuous function. By the triangle inequality,   
\begin{equation}\label{eq:decomp11243}
\l| \pi(f) - \pi^\pas(f) \r| \le \l|\pi(f) -\ES \l[ f ( \X_{n\pas}^{x} ) \r] \r| +\l|\ES \l[ f ( \X_{n\pas}^{x} ) \r]- \ES \l[ f \l( \bar{X}_{n}^{\pas,x} \r) \r] \r| \l|+ \ES \l[ f \l( \bar{X}_{n}^{\pas,x}\r) \r] - \pi^\pas(f) \r|.
\end{equation}
By Lemma \ref{lem:boundEuler}, we know that under $\Cs$, 
$$\l|\pi(f) -\ES \l[ f ( \X_{n\pas}^{x} ) \r] \r|+|\ES \l[ f \l( \bar{X}_{n}^{\pas,x}\r) \r] - \pi^\pas(f)| \xrightarrow{n\rightarrow+\infty}0.$$
Thus, since $f$ is Lispchitz continuous, 
$$\l| \pi(f) - \pi^\pas(f) \r| \le \limsup_{n\rightarrow+\infty} \l|\ES \l[ f ( \X_{n\pas}^{x} ) \r]- \ES \l[ f \l( \bar{X}_{n}^{\pas,x} \r) \r]\r|
\le [f]_1\limsup_{n\rightarrow+\infty} \| \X_{n\pas}^{x}-\bar{X}_{n\pas}^{\pas,x}\|_1.$$
But by Proposition \ref{prop:L2error}$(i)$and $(ii)$ applied with $x=x^\star$, we get respectively
$$\limsup_{n\rightarrow+\infty} \| \X_{n\pas}^{x}-\bar{X}_{n}^{\pas,x}\|_2\le \sqrt{\beta_1(x^\star)}\sqrt{\gamma}
\quad\textnormal{and}\quad\limsup_{n\rightarrow+\infty} \| \X_{n\pas}^{x}-\bar{X}_{n}^{\pas,x}\|_2\le \sqrt{\beta_2(x^\star)}\, \pas.
$$
Hence, since ${\cal W}_1(\pi,\pi^\pas)=\sup_{f, [f]_1\le 1} |\pi(f)-\pi^\pas(f)|$ and $\| \X_{n\pas}^{x}-\bar{X}_{n\pas}^{\pas,x}\|_1\le \| \X_{n\pas}^{x}-\bar{X}_{n}^{\pas,x}\|_2$, the first inequality follows. For the second part of the proposition, it is enough to remark that the inequality is true for every $x^\star\in\ER^d$.

\end{proof}

\subsection{Proof of the main results of Section \ref{sec:langevinsc}}\label{sec:resumedespreuves}
We are now ready to prove our main results under $\Cs$.\smallskip

\noindent \textbf{Proof of Proposition \ref{prop:stronconvbea1} and Theorem \ref{maincoro1}.}
{The bound on $c_1$ of   Proposition \ref{prop:stronconvbea1} follows from Lemma \ref{lem:boundEuler}$(ii)$ applied with $x^\star=x_0$. For the one on 
$\max(\frac{\alpha \2^2}{L^2}, \frac{\alpha\3^2}{L^2},\4^2 )$,  it is enough  to apply Lemma \ref{lem:boundEuler}$(iii)$, Proposition \ref{prop:L2error}$(iii)$ (with $\bea=1$) and Lemma \ref{prop:htrois} (with $\delta=1/2$).}\\

\noindent  {To prove Theorem \ref{maincoro1}, we deduce from Proposition \ref{prop:stronconvbea1} and from Remark \ref{rem:withnonexplicitconstants} that we can apply Theorem \ref{maintheo}$(ii)$ with $\bea=1$, $\delta=1/2$, $\tilde{c}_4={\Upsilon_1}$ and $\tilde{c}_3^2=\tilde{c}_2^2={\Upsilon_1^2}{L^2}{\alpha^{-1}}$\footnote{These notations are introduced in Remark \ref{rem:withnonexplicitconstants} which manages the setting where the ``real'' constants are known up to some universal constants, which is the case in the bounds of Proposition \ref{prop:stronconvbea1}.} and $2\eta_0=\alpha L^{-2}\wedge 1$.  Setting $r_0=\Upsilon_1$ (defined in Proposition \ref{prop:stronconvbea1}) implies that 
$r_0\ge \tilde{c}_3\pas_0^{\frac{1}{2}}$ (since $\pas_0\le \alpha/L^2$). One also remarks that 
$ T_r=\Tfrak\varepsilon^{-2} R_\varepsilon^2 2^{-r}$ with 
$$\Tfrak=\alpha^{-1}\Upsilon_1^2 \log(\pas_0^{-1})\ge \dbea \max(\tilde{c}_2^2\pas_0\log(\pas_0^{-1}),\tilde{c}_4^2)$$
as required in Theorem \ref{maintheo}$(ii)$. The condition on $\varepsilon_0$ follows from the definition given in Theorem \ref{maintheo}$(ii)$ and from the fact that}
{ 
$$ T_{R_\varepsilon}\ge \frac{\Upsilon_1^2\log (\pas_0^{-1})}{2\alpha} r_0^{-2} R_\varepsilon^2=\frac{\log(\pas_0^{-1})R_\varepsilon^2}{2}.$$
The bound \eqref{eq:complex2} then follows from  the fact that $\mathfrak{C}_2= \frac{5}{2}\pas_0^{-1}\Tfrak$. For the last part, we first remark that  under the additional conditions,
 $\Upsilon_1^2\lesssim_{uc} \sigma^2 \alpha^{-1}d$   so that if we set $\tilde{\Upsilon}_1^2= \sigma^2 (L\alpha^{-1})^2 d$, we can again use Remark \ref{rem:withnonexplicitconstants} to obtain the last bound.}\smallskip


\noindent \textbf{Proof of Proposition \ref{prop:stronconvbea2} and Theorem \ref{maincoro2}.}
{Let us begin by the bound on $\max(\frac{\alpha^2 \2^2}{L^4}, \frac{\alpha^2\3^2}{L^4},\4^2 )$ of Proposition \ref{prop:stronconvbea2}. By Proposition \ref{prop:L2error}$(ii)$ (applied with $x=x_0=x^\star$) and Proposition \ref{prop:htrois} (and the fact that $\alpha\le L$), one checks  that
$$\max\l(\frac{\alpha^2 \2^2}{L^4}, \frac{\alpha^2\3^2}{L^4}\r)\lesssim_{uc}\frac{1}{\alpha^2} |b(x_0)|^2+ \frac{\sigma\sqrt{\alpha}}{L^3}\|\nabla b\|_{2,\infty}|b(x_0)|+\left(\frac{\sigma^4}{L^4}\|\Delta b\|_{2,\infty}^2+ \frac{{\sigma^2\alpha d^{\frac{1}{2}}}}{L^3}\|\nabla b\|_{2,\infty}+\frac{\sigma^{2} }{\alpha}\r) d.$$
Using Lemma \ref{lem:boundEuler}$(iii)$ for $c_4^2$, we obtain the result.}

{To prove Theorem \ref{maincoro2}, we deduce from Proposition \ref{prop:stronconvbea2} and from Remark \ref{rem:withnonexplicitconstants} that we can apply Theorem \ref{maintheo}$(ii)$ with $\bea=2$, $\delta=1$, $\tilde{c}_4={\Upsilon_2}$ and $\tilde{c}_3^2=\tilde{c}_2^2={\Upsilon_2^2}{L^4}{\alpha^{-2}}$ and $2\eta_0=\alpha L^{-2}\wedge 1$. Using that $\Upsilon_2\ge1$, the proposed values of $r_0$, $R_\varepsilon$,  $\tau_1$ and $\tau_2$ easily follow. For $\Tfrak$,  we use Proposition \ref{prop:stronconvbea2} which implies that
$$\frac{{\mathfrak{d}_2}}{\alpha}\max(\tilde{c}_2^2\pas_0^2 \log(\pas_0^{-1}),\tilde{c}_4^2)\le \frac{\Upsilon_2^2\log(\pas_0^{-1})}{\alpha},$$
and thus set $\Tfrak=\alpha^{-1}{\Upsilon_2^2\log(\pas_0^{-1})}.$ This implies that
$$ T_{R_\varepsilon}\ge  \frac{\Upsilon_2^2\log(\pas_0^{-1})}{\alpha} 2^{-\frac{3}{2}(\log_2(\Upsilon_2\varepsilon^{-1})+1)}=\frac{(\Upsilon_2\varepsilon^{-1})^{\frac{1}{2}}\log(\pas_0^{-1})} {2\sqrt{2}\alpha}.$$
}
Since $\tau_1\log \varepsilon^{-1}+\tau_2=(2\alpha)^{-1} \log(\Upsilon_2\varepsilon^{-2})
$, we deduce the proposed value of $\varepsilon_0$.}\\

\noindent {By  Theorem \ref{maintheo}$(ii)$, 
$$\mathfrak{C}_2=\mathfrak{c}_2\pas_0^{-1} \Tfrak= \frac{\frac{1}{2}+\sqrt{2}}{\sqrt{2}-1}\frac{{\Upsilon_2^2} \log(\pas_0^{-1})}{\gamma_0\alpha} \varepsilon^{-2}.$$
This is exactly \eqref{eq:complex3}. }\\
{Let us finally prove \eqref{eq:complex4}. By the  additional assumptions on $\nabla b$, $\Delta b$ and $b(x_0)$ (and the fact that $\alpha\le L$), one checks that,
\begin{equation*}
\frac{1}{\alpha^2} |b(x_0)|^2+ \frac{\sigma\sqrt{\alpha}}{L^3}\|\nabla b\|_{2,\infty}|b(x_0)|+\frac{\sigma^4}{L^4}\|\Delta b\|_{2,\infty}^2+ \frac{{\sigma^2\alpha d^{\frac{1}{2}}}}{L^3}\|\nabla b\|_{2,\infty}+\frac{\sigma^{2} d }{\alpha} \lesssim_{uc} \frac{\sigma^2 d}{\alpha}.
\end{equation*}
Thus, $\Upsilon_2^2$ defined in  Proposition \ref{prop:stronconvbea2} satisfies $\Upsilon_2^2\lesssim_{uc} \sigma^2\alpha^{-1} d$. Then, with the help of Remark \ref{rem:withnonexplicitconstants}, we can apply Theorem \ref{maintheo}$(ii)$ with  $\tilde{\Upsilon}_2^2=\sigma^2 \alpha^{-1} d$ and obtain the announced result.}\\
\smallskip

\noindent \textbf{Fundings.} The authors are grateful to the SIRIC ILIAD Nantes-Angers program supported by the French National Cancer Institute (INCA-DGOS-Inserm 12558 grant), for the funding of the Phd thesis of M. Eg\'ea, and to Manon Desloges for her help in the numerical development of the algorithm.
{\section{List of specific symbols}\label{sec:listsymbols}
In order to help the reading of this  paper, we list the specific symbols used in the paper and the page where they are defined.}\\
{\begin{table}[htb]
    \centering
    \hfill
    \begin{tabular}{c c c|}
       $\bar{X}_t^{\pas,x_0}$ &continuous-time Euler scheme& \pageref{pagesbargamma}\tabularnewline
$\underline{t}_\pas$ & discretization time & \pageref{pagesbargamma}\tabularnewline
$R$& number of correcting layers &\pageref{pagetr}\tabularnewline
$T_r$& length of the path involved in level $r$ &\pageref{pagetr}\tabularnewline
$\pas_r$& step of level $r$: $\pas_r=\pas_0 2^{-r}$  &\pageref{pagepasr}\tabularnewline
$B^{(r)}$ & Brownian motion of level $r$ &\pageref{eq:estimat2}\tabularnewline
$x_0$ & starting point of each Euler scheme &\pageref{pagexzero}\tabularnewline
$\mathcal{C}(\mathcal{Y})$ & complexity of the algorithm &\pageref{sec:setmain2}\tabularnewline
$\eta_0$ & maximal stepsize &\pageref{pageetazero}\tabularnewline
$c_i$ & constants in $\mathbf{(H_i)}$, $i=1,\ldots,4$. &\pageref{HUN}\tabularnewline
$\pi^\pas$ & inv. distrib. of the Euler scheme &\pageref{pagepipas}\tabularnewline
$\alpha$ & ergodicity exponent  in $\HUN$ & \pageref{HUN}\tabularnewline
   $\bea$ & confluence parameter in $\HDEUX$ & \pageref{defbea}\tabularnewline
   $\delta$ & (weak order) parameter  in $\HTROIS$ & \pageref{defbea} \tabularnewline
   $r_0$ & parameter related to $R$ & \pageref{eq:choixpar} \tabularnewline \end{tabular}
    \hfill
     \begin{tabular}{c c c}
     $\Tfrak$ & parameters related to $T_r$ & \pageref{eq:choixpar} \tabularnewline
   $\mathfrak{C}_i$ & complexity constant, $i=1,2$ & \pageref{eq:pluspetitqueepsilon} \tabularnewline
   $L$ & Lipschitz constant of $b$& \pageref{eq:upsilon}\tabularnewline
    $\Upsilon_1$ & parameter in \cref{prop:stronconvbea1}  & \pageref{eq:upsilon}\tabularnewline
    $\bar{\lambda}_U$ &highest eigenvalue of $D^2U$&\pageref{hyp:valeurpropre}\\
    $\underline{\lambda}_U$& lowest eigenvalue of $D^2U$&\pageref{hyp:valeurpropre}\tabularnewline
    $\alpha_U$ &$\bar{\lambda}_U\wedge 1$&\pageref{hyp:valeurpropre}\tabularnewline
 $L_U$ & Lipschitz constant of $\nabla U$&\pageref{hyp:valeurpropre}\tabularnewline
 $\|\,.\,\|_F$ &Frobenius norm & \pageref{prop:stronconvbea2}\tabularnewline
  $ \|\,.\,\|_{2,\infty}$& Infinity-$L^2$ norm & \pageref{prop:stronconvbea2}\tabularnewline
$\Upsilon_2$ & parameter in \cref{prop:stronconvbea2}  & \pageref{prop:stronconvbea2}\tabularnewline
 $\cuniv$ & universal constant& \pageref{maintheointermed}\tabularnewline
 $\dbea$, $\Tfrak_0$ &  constants related to $\Tfrak$  & \pageref{eq:choiceoftheparameters}\tabularnewline
   $\tau_i$ & warm-start parameter, $i=1,2$ & \pageref{tauuntaudeux}\tabularnewline
 \end{tabular}
    \hfill \null
    \label{uca}
\end{table}
}

\bibliographystyle{alpha}
\bibliography{biblio}

\end{document}